\documentclass[UTF8,10pt]{article}

\usepackage[left=1.91cm,right=1.91cm,top=2.54cm,bottom=2.54cm]{geometry}
\usepackage{amsfonts}
\usepackage{amsmath}
\usepackage{amssymb}
\usepackage{pgfplots}
\usepackage{tikz}
\usetikzlibrary{arrows}
\usepackage{titlesec}
\usepackage{bm}
\usepackage{appendix}
\usepackage{tikz-cd}
\usepackage{mathtools}
\usepackage{amsthm}
\usepackage{stmaryrd}
\usepackage{setspace}
\usepackage{enumitem}
\setenumerate[1]{itemsep=0pt,partopsep=0pt,parsep=\parskip,topsep=5pt}
\setitemize[1]{itemsep=0pt,partopsep=0pt,parsep=\parskip,topsep=5pt}
\setdescription{itemsep=0pt,partopsep=0pt,parsep=\parskip,topsep=5pt}
\usepackage[colorlinks,linkcolor=black,anchorcolor=black,citecolor=black]{hyperref}
\usepackage[numbers]{natbib}
\usepackage{setspace}
\usepackage{mathrsfs}
\usepackage{cases}
\usepackage{sectsty}
\sectionfont{\centering}
\usepackage{fancyhdr}
\usepackage{txfonts}
	\theoremstyle{definition}
	\newtheorem{thm}{Theorem}[section]

	\newtheorem{defn}{Definition}[section]
	\newtheorem{lem}[thm]{Lemma}
	\newtheorem{prop}[thm]{Proposition}
	
	\newtheorem{cor}[thm]{Corollary}
	
	\newtheorem{rmk}{Remark}[section]

	\renewcommand{\Re}{\textbf{Re }}  
	\renewcommand{\Im}{\textbf{Im }}  
	\newcommand{\R}{\mathbb{R}}  
	\newcommand{\Z}{\mathbb{Z}}
	
	\newcommand{\N}{\mathbb{N}}
	
	\newcommand{\T}{\mathbb{T}}
	\newcommand{\TT}{\mathcal{T}}
	\newcommand{\TP}{\overline{\p}{}}
	
	\newcommand{\TL}{\overline{\Delta}{}}

	\newcommand{\dive}{\text{div }}
	\newcommand{\ST}{\text{ ST}}
	\newcommand{\RT}{\text{ RT}}
	\newcommand{\VS}{\text{ VS}}

	\newcommand{\sgn}{\text{sgn}}
	\newcommand{\bd}[1]{\mathbf{#1}} 
	\newcommand{\RR}{\mathcal{R}}      
	\newcommand{\Om}{\Omega}
	
	\newcommand{\q}{\quad}
	\newcommand{\p}{\partial}
	\newcommand{\dd}{\mathfrak{D}}

	\newcommand{\nab}{\nabla}

	\newcommand{\lap}{\Delta}
	\newcommand{\vp}{\varphi}
	
	\newcommand{\no}{\nonumber}

	\newcommand{\eql}{\stackrel{L}{=}}
	\newcommand{\cnab}{\overline{\nab}}
	
	\newcommand{\jp}{\lee\TP\ree}
	
	\newcommand{\dx}{\,\mathrm{d}x}
	
	\newcommand{\dz}{\,\mathrm{d}z}
	\newcommand{\dt}{\,\mathrm{d}t}
	\newcommand{\dtau}{\,\mathrm{d}\tau}
	
	\newcommand{\deta}{\,\mathrm{d}\eta}

	\newcommand{\lee}{\langle}
	\newcommand{\ree}{\rangle}

	\newcommand{\kk}{\kappa}

	\newcommand{\ee}{\mathcal{E}}
	\newcommand{\eew}{\widetilde{\mathcal{E}}}
	\newcommand{\rr}{\mathfrak{R}}
	
	\newcommand{\cc}{\mathfrak{C}}

	\newcommand{\h}{\mathcal{H}}
	\newcommand{\WW}{\mathcal{W}}
	\newcommand{\VV}{\mathbf{V}}
	\newcommand{\BB}{\mathbf{B}}
	\newcommand{\PP}{\mathbf{P}}
	\newcommand{\LP}{\mathcal{P}}
	
	\newcommand{\EE}{\mathfrak{E}}
	\newcommand{\HE}{\mathcal{E}}
	\newcommand{\EW}{\widetilde{E}}
	
	\newcommand{\EEW}{\widetilde{\mathfrak{E}}}
	\newcommand{\fs}{\mathfrak{s}}
	\newcommand{\BS}{\mathbf{S}}
	
	\newcommand{\GG}{\mathbf{G}}
	\newcommand{\FF}{\mathbf{F}}
	\newcommand{\FP}{\mathfrak{F}_\psi}

	\newcommand{\hh}{\mathfrak{H}}

	\newcommand{\io}{\int_{\Omega}}

	\newcommand{\iopm}{\int_{\Omega^\pm}}
	\newcommand{\is}{\int_{\Sigma}}

	\newcommand{\ddt}{\frac{\mathrm{d}}{\mathrm{d}t}}
	\numberwithin{equation}{section}
	
	\newcommand{\eps}{\varepsilon}
	\providecommand{\jump}[1]{\left\llbracket #1 \right\rrbracket }
	\providecommand{\len}[1]{\lee #1 \ree }
	\providecommand{\ino}[1]{\left\| #1 \right\| }
	\providecommand{\bno}[1]{\left| #1 \right| }
	\usepackage{xcolor}

	\newcommand{\lam}{\lambda}
	\newcommand{\Lam}{\Lambda}

	\newcommand{\pk}{\widetilde{\varphi}}

	\newcommand{\NN}{\mathbf{N}}

	\newcommand{\QQ}{\mathbf{Q}}
	\newcommand{\vb}{\bar{v}}
	\newcommand{\wb}{\bar{w}}
	
	\newcommand{\ff}{\mathcal{F}}
	\newcommand{\ffp}{\mathcal{F}_p}
	\newcommand{\ffpm}{\mathcal{F}_p^\pm}

	\newcommand{\gapm}{{\gamma,\pm}}

	\newcommand{\es}{{\varepsilon, \sigma}}

	\newcommand{\pp}{\p^{\varphi}}
	\newcommand{\nabp}{\nab^{\varphi}}
	\newcommand{\nabper}{\nab^{\varphi,\perp}}
	
	\newcommand{\dnp}{\mathfrak{N}_\psi^+}
	\newcommand{\dnm}{\mathfrak{N}_\psi^-}
	\newcommand{\dnpm}{\mathfrak{N}_\psi^\pm}
	\newcommand{\dnmp}{\mathfrak{N}_\psi^\mp}
	\newcommand{\dnw}{\widetilde{\mathfrak{N}}}
	\newcommand{\bp}{(b\cdot\nab^{\varphi})}

	\newcommand{\bc}{\bar{b}}

	\newcommand{\fn}{{\mathfrak{n}}}
	\newcommand{\fm}{{\mathfrak{m}}}
	\newcommand{\MM}{{\mathfrak{M}}}
	\newcommand{\fb}{\mathbf{b}}
	
	\newcommand{\fw}{\mathbf{w}}
	\newcommand{\fu}{\mathbf{u}}
	\newcommand{\fbc}{\bar{\mathbf{b}}}
	
	\newcommand{\fwc}{\bar{\mathbf{w}}}
	\newcommand{\fuc}{\bar{\mathbf{u}}}

	\newcommand{\bmu}{\bar{\mu}}
	
	\newcommand{\bpm}{(b^\pm\cdot\nab^{\varphi})}

	\newcommand{\nabpp}{\nab^{\varphi_0}}
	
	\newcommand{\lapp}{\lap^{\vp}}

	\newcommand{\Dtbpm}{\overline{D_t^\pm}}

	\newcommand{\Dtp}{D_t^{\varphi}}

	\newcommand{\Dtpm}{D_t^{\vp\pm}}

	\newcommand{\dvt}{\,\mathrm{d}\mathcal{V}_t}

\begin{document}
\bibliographystyle{plain}
\title{\textbf{On the Incompressible Limit of Current-Vortex Sheets \\ with or without Surface Tension}}
\author{
{\sc Junyan Zhang}\thanks{School of Mathematical Sciences. University of Science and Technology of China. 96 Jinzhai Road, Baohe District, Hefei, Anhui 230026, China. Email: \texttt{yx3x@ustc.edu.cn}.}
 }
\date{\today}
\maketitle

\setcounter{tocdepth}{1}

\begin{abstract}
	This is the second part of the two-paper sequence, which aims to present a comprehensive study for current-vortex sheets in ideal compressible magnetohydrodynamics (MHD). The local well-posedness of current-vortex sheets with surface tension has been proved in the first part of the paper sequence \cite{Zhang2023CMHDVS1}. In this paper, we prove the incompressible and zero-surface-tension limits under certain stability conditions. The proof of uniform estimates relies on the analysis of the evolution equation of the free interface via a paradifferential approach, the wave equation of the pressure and a weighted anisotropic structure in vorticity analysis.% To our knowledge, this is the first result that rigorously justifies the incompressible limit of inviscid vortex sheets and also free-surface ideal MHD flows.
\end{abstract}



\tableofcontents

\section{Introduction}
Current-vortex sheet, as one of the characteristic discontinuities in ideal compressible MHD, is described by a free-interface problem of two-phase compressible ideal MHD flows where the magnetic fields are tangential to the interface and no mass flow transfers across the interface. This model has been widely used in both solar physics (e.g., the heliopause of the solar system, the night-side magnetopause of the earth) \cite{CVSphy} and controlled nuclear fusion \cite{MHDphy}. In particular, the surface tension effect on the interface cannot be neglected when the plasma is a liquid metal \cite{MHDSTphy2}. 

Let $H>10$ be a given real number, $x=(x_1,\cdots,x_d)$ and $x':=(x_1,\cdots,x_{d-1})$ and the space dimension $d=2,3$. We define the regions $\Om^+(t):=\{x\in\T^{d-1}\times\R:\psi(t,x')<x_d<H\}$, $\Om^-(t):=\{x\in\T^{d-1}\times\R:-H<x_d<\psi(t,x')\}$ and the moving interface $\Sigma(t):=\{x\in\T^{d-1}\times\R:x_d=\psi(t,x')\}$ between $\Om^+(t)$ and $\Om^-(t)$. For $T>0$, we denote $\Om_T^\pm:=\bigcup\limits_{0\leq t\leq T}\{t\}\times\Om^\pm(t)$ and $\Sigma_T:=\bigcup\limits_{0\leq t\leq T}\{t\}\times\Sigma(t)$. The following free-interface problem describes the motion of $\Sigma(t)$ as a \textit{current-vortex sheet} (or an \textit{MHD tangential discontinuity}) 
\begin{equation}\label{CMHDVS}
\begin{cases}
\varrho^\pm (\p_t+u^\pm\cdot\nab) u^\pm-B^\pm\cdot\nabla B^\pm+\nabla Q^\pm=0,~~Q^\pm:=P^\pm+\frac{1}{2}|B^\pm|^2 &~~\text{ in }\Om_T^\pm,\\
(\p_t+u^\pm\cdot\nab)\varrho^\pm+\varrho^\pm\nab\cdot u^\pm=0&~~\text{ in }\Om_T^\pm,\\
(\p_t+u^\pm\cdot\nab) B^\pm=B^\pm\cdot\nabla u^\pm-B^\pm\nab\cdot u^\pm&~~\text{ in }\Om_T^\pm,\\
\nab\cdot B^\pm=0&~~\text{ in }\Om_T^\pm,\\
(\p_t+u^\pm\cdot\nab) \fs^\pm=0&~~\text{ in }\Om_T^\pm,\\
P^\pm=P^\pm(\varrho^\pm,\mathfrak{s}^\pm),~~\frac{\p P^\pm}{\p\varrho^\pm}>0,~~\varrho^\pm\geq\overline{\rho}_0>0&~~\text{ in }\overline{\Om_T^\pm},\\
\jump{Q}=\sigma\h&~~\text{ on }\Sigma_T,\\
B^\pm\cdot N=0&~~\text{ on }\Sigma_T,\\
\p_t\psi=u^\pm\cdot N&~~\text{ on }\Sigma_T,\\
u_d^\pm=B_d^\pm=0&~~\text{ on }[0,T]\times\Sigma^\pm,\\
(u^\pm,B^\pm,\varrho^\pm,\mathfrak{s}^\pm)|_{t=0}=(u_0^\pm,B_0^\pm,\varrho_0^\pm,\mathfrak{s}_0^\pm)~~\text{ in }\Om^\pm(0),\quad\quad \psi|_{t=0}=\psi_0 &~~\text{ on }\Sigma(0).
\end{cases}
\end{equation}
Here $\nabla:=(\p_{x_1},\cdots,\p_{x_d})$ is the standard spatial derivative. $D_t:=\p_t+u\cdot\nabla$ is the material derivative. The fluid velocity, the magnetic field, the fluid density, the fluid pressure and the entropy are denoted by $u=(u_1,\cdots,u_d)$, $B=(B_1,\cdots,B_d)$, $\varrho$, $P$ and $\mathfrak{s}$ respectively. The quantity $Q:=P+\frac12|B|^2$ is the total pressure. Note that the fourth equation in \eqref{CMHDVS} is just an initial constraint instead of an independent equation. The fifth equation of \eqref{CMHDVS} is derived from the equation of total energy and Gibbs relation and we refer to \cite[Ch. 4.3]{MHDphy} for more details. To close system \eqref{CMHDVS}, we need to introduce the equation of state
\begin{equation}\label{eos}
P=P(\varrho,\mathfrak{s})\text{ satisfying }\frac{\p P}{\p\varrho}>0.
\end{equation} The assumption $\varrho\geq\bar{\rho_0}>0$ (for some constant $\bar{\rho_0}>0$) together with $\frac{\p P}{\p\varrho}>0$ guarantees the hyperbolicity of system \eqref{CMHDVS}. 

The seventh to the tenth equations in \eqref{CMHDVS} are the boundary conditions, in which $N:=(-\p_1\psi,\cdots,-\p_{d-1}\psi, 1)^\top$ is the normal vector to $\Sigma(t)$ (pointing towards $\Om^+(t)$), $\sigma\geq 0$ is the constant coefficient of surface tension and the quantity $\mathcal{H}:=\cnab\cdot\left(\frac{\cnab\psi}{\sqrt{1+|\cnab\psi|^2}}\right)$ is twice the mean curvature of $\Sigma(t)$ with $\cnab=(\p_1,\cdots,\p_{d-1})$. The jump of a function $f$ on $\Sigma(t)$ is denoted by $\jump{f}:=f^+|_{\Sigma(t)}-f^-|_{\Sigma(t)}$ with $f^\pm:=f|_{\Om^\pm(t)}$. 
\begin{itemize}
\item The first boundary condition $\jump{Q}=\sigma \h$ shows that the jump of total pressure is balanced by surface tension. 
\item The second boundary condition $B^\pm\cdot N=0$ shows that both plasmas are perfect conductors. 
\item The third boundary condition $\p_t\psi=u^\pm\cdot N$ shows that there is no mass flow across the interface and thus the two plasmas are physically contact and mutually impermeable.  
\item The slip boundary conditions are imposed on the rigid boundaries $\Sigma^\pm:=\T^{d-1}\times\{\pm H\}$. 
\end{itemize}We refer to \cite[Appendix A]{TW2021MHDCDST} for detailed derivation of these boundary conditions from the Rankine-Hugoniot conditions. Also note that the conditions $\nab\cdot B^\pm=0$ in $\Om^\pm(t)$, $B^\pm\cdot N|_{\Sigma(t)}=0$ and $B^\pm_d=0$ on $\Sigma^\pm$ are constraints for initial data so that system \eqref{CMHDVS} is not over-determined. These initial constraints can propagate within the lifespan of solutions. 

We also require the initial data to satisfy certain compatibility conditions. Let $(u_0^\pm,B_0^\pm,\varrho_0^\pm,\mathfrak{s}_0^\pm,\psi_0):=(u^\pm,B^\pm,\varrho^\pm,\mathfrak{s}^\pm,\psi)|_{t=0}$ be the initial data of system \eqref{CMHDVS}. We say the initial data satisfies the compatibility condition up to $m$-th order $(m\in\N)$ if
\begin{equation}\label{CMHDcomp}
\begin{aligned}
(D_t^\pm)^j \jump{Q}|_{t=0}=\sigma(D_t^\pm)^j\mathcal{H}|_{t=0}~~\text{ on }\Sigma(0),~~0\leq j\leq m,\\
(D_t^\pm)^j \p_t\psi|_{t=0}=(D_t^\pm)^j(u^\pm\cdot N)|_{t=0}~~\text{ on }\Sigma(0),~~0\leq j\leq m,\\
\p_t^j u_d^\pm=0~~\text{ on }\Sigma^\pm,~~0\leq j\leq m.
\end{aligned}
\end{equation} In fact, the fulfillment of the first condition implicitly requires the fulfillment of the second one. With these compatibility conditions, one can show that the magnetic fields also satisfy (cf. \cite[Section 4.1]{Trakhinin2008CMHDVS})
\[
(D_t^\pm)^j(B^\pm\cdot N)|_{t=0}=0~~\text{ on }\Sigma(0)\text{ and }\Sigma^\pm,~~0\leq j\leq m.
\] 

System \eqref{CMHDVS} admits a conserved $L^2$ energy
\[
\begin{aligned}
E_0(t):=\sum_{\pm}\frac12\int_{\Om^\pm(t)} \varrho^\pm|u^\pm|^2+|B^\pm|^2+ 2 \mathfrak{P}(\varrho^\pm,\fs^\pm) + \varrho^\pm |\fs^\pm|^2\dx + \sigma \text{ Area}(\Sigma(t))
\end{aligned}
\]where $\mathfrak{P}(\varrho^\pm,\fs^\pm)=\int_{\bar{\rho}_0}^{\varrho^\pm}\frac{P^\pm(z,\fs^\pm)}{z^2}\dz$. See the first part of this paper sequence \cite[Section 3.1]{Zhang2023CMHDVS1} for the proof.

The local well-posedness and the incompressible limit of \eqref{CMHDVS} for each fixed $\sigma>0$ has been established in \cite[Section 3.1]{Zhang2023CMHDVS1}. This paper is a continuation of \cite{Zhang2023CMHDVS1}: we aim to prove the zero-surface-tension limit for the solution to \eqref{CMHDVS} under certain stability conditions and improve the low Mach number limit result in \cite{Zhang2023CMHDVS1} such that the uniform boundedness (with respect to the Mach number) of high-order ($\geq 2$ order) time derivatives can be dropped. The motivation to study these limit processes are stated in Section \ref{sect goals}.

\subsection{Reformulation in flattened domains}
\subsubsection{Flattening the fluid domains}
As in \cite{Zhang2023CMHDVS1}, we convert \eqref{CMHDVS} into a PDE system defined in $\Om^\pm := \T^{d-1}\times\{0<\pm x_d<H\}.$ We consider a family of diffeomorphisms $\Phi(t, \cdot): \Om^\pm\to \Om^\pm(t)$ characterized by the moving interface. Let $\Phi(t,x',x_d) = \left(x', \varphi(t,x_d)\right)$
where
\begin{align}
 \varphi (t,x) = x_d+\chi(x_d)\psi (t,x') \label{change of variable vp}
\end{align}and $\chi\in C_c^\infty ([-H,H])$ is a smooth cut-off function satisfying the following bounds:
\begin{equation}
{\|\chi'\|_{L^\infty(\R)}\leq \frac{1}{\|\psi_0\|_{\infty}+20},\quad \sum_{j=1}^8 \|\chi^{(j)}\|_{L^\infty(\R)}\leq C, }\quad\chi=1\,\,\,\, \text{on}\,\, (-1, 1) \label{chi}
\end{equation}
for some generic constant $C>0$. We assume $|\psi_0|_{L^{\infty}(\T^2)}\leq 1$. One can prove that there exists some $T_0>0$ such that $\sup\limits_{[0,T_0]}|\psi(t,\cdot)|_{L^{\infty}(\T^2)}<10<H$, the free interface is still a graph within the time interval $[0,T_0]$ and
$$\p_d \varphi(t, x', x_d) = 1+\chi'(x_d)\psi(t,x')=1-\frac{1}{20}\times 10 \geq \frac12 ,\quad t\in[0,T_0],
$$
which ensures that $\Phi(t)$ is a diffeomorphism in $[0,T_0]$.  We then introduce the new variables
\begin{align}
v^\pm(t, x) = u^\pm(t, \Phi(t,x)),\quad b^\pm(t,x)=B^\pm(t,\Phi(t,x)),\quad \rho^\pm(t,x) = \varrho^\pm(t, \Phi(t,x)),\notag\\
S^\pm(t, x) = \mathfrak{s}^\pm(t, \Phi(t,x)), \quad q^\pm(t,x) = Q^\pm(t, \Phi(t,x)), \quad p^\pm(t,x)=P(t,\Phi(t,x))
\end{align}
that represent the velocity fields, the magnetic fields, the density functions, the entropy functions, the total pressure functions and the fluid pressure functions defined in the fixed domains $\Omega^\pm$ respectively. Also, we introduce the differential operators
\begin{align}
\nab^{\vp}=(\pp_1,\cdots,\pp_d),~\text{where } &~\pp_a = \p_a -\frac{\p_a \varphi}{\p_d\varphi}\p_d,~~a=t, 1,\cdots,d-1;~~\p_d^{\vp}= \frac{1}{\p_d \varphi} \p_d.\label{nabp 3}
\end{align}
We set the tangential gradient operator and the tangential derivatives as $\cnab := (\p_1, \cdots,\p_{d-1}),~\TP_i:=\p_i,~i=1,\cdots,d-1.$ Invoking \eqref{nabp 3}, we can alternatively write the material derivative $\Dtp$ as
\begin{equation}
\Dtpm = \p_t + \vb^\pm\cdot\cnab+\frac{1}{\p_d \varphi} (v^\pm\cdot \NN-\p_t\varphi)\p_d, \label{Dt alternate}
\end{equation}
where $\vb^\pm:=(v_1^\pm,\cdots,v_{d-1}^\pm)^\top$ consists of the horizontal components of the fluid velocity and 
$\vb^\pm \cdot \cnab := \sum\limits_{j=1}^{d-1}v_j^\pm\p_j$. $\NN:= (-\p_1\varphi, \cdots,-\p_{d-1} \varphi, 1)^\top$ is the extension of the normal vector $N$ into $\Om^\pm$. This formulation will be helpful for us to simplify the analysis.

\subsubsection{Parametrization of the equation of state}\label{sect eos} We assume the fluids in $\Om^+$ and $\Om^-$ satisfy the same equation of state. Specifically, we parametrize the equation of state to be $\rho_\lam(p,S):=\rho(p/\lam^2,S)$ where $\lam>0$ is proportional to the sound speed $c_s:=\sqrt{\p_p\rho}$ and $\rho$ is a $C^8$ function in its arguments satisfying $\frac{\p \rho}{\p p}>0$ as well as the non-degeneracy condition $\rho\geq \bar{\rho_0}>0$ in $\overline{\Om}$ for some constant $\bar{\rho_0}$. By chain rule, it is straightforward to see 
	\begin{align}\label{eos ineq 1}
	0<\frac{\p }{\p p}\rho_\lam(p,S)\leq C\lam^{-2}.
	\end{align}
	and 
	\begin{align}\label{eos ineq 3}
	|(\p_p)^k \rho_\lam(p,S)|\leq C\lam^{-2k},\q\q |(\p_S)^k \rho_\lam(p, S)|\leq C,\q\q 1\leq k\leq 8,
	\end{align} for some $C>0.$  For example, a polytropic gas satisfies the above assumptions whose the equation of state is parametrized  in terms of $\lam>0$: 
	\begin{equation}\label{eos11}
	p_{\lam}(\rho,S)=\lam^2\left(\rho^{\gamma}\exp(S/C_V)-1\right),~~~\gamma> 1,~~C_V>0.
	\end{equation}

\paragraph*{The formulation used in this manuscript.}		
For sake of clean notations,  we would introduce the quantity $\ff^\pm:= \log \rho^\pm$ to replace $\rho$ and introduce the parameter $\eps:=1/\lam$ to replace $\lam$  in the continuity equation, that is, $\ff_\eps(p,S):=\log\rho_{\frac{1}{\eps}}(p, S)$. Since $\frac{\p p^\pm}{\p\rho^\pm}>0$ and $\rho^\pm>0 $ imply $\frac{\p\ff^\pm}{\p p^\pm}=\frac{1}{\rho^\pm}\frac{\p \rho^\pm}{\p p^\pm}>0$,  then the continuity equation is equivalent to 
\begin{equation} \label{continuity eq f}
\frac{\p\ff_\eps^\pm}{\p p^\pm}( p^\pm,S^\pm) \Dtpm p^\pm  +\nabp\cdot v^\pm=0. 
\end{equation}\eqref{eos ineq 1}-\eqref{eos ineq 3} lead to the following inequalities: There exists a constant $A>0$ such that
		\begin{align} \label{pp property}
			0<\frac{\p \ff_\eps}{\p p}( p,S)\leq&~ A\eps^2,\\
			|\p_p^k \ff_\eps(p,S)|\leq A\eps^{2k},~~~|\p_S^k \ff_\eps(p, S)|\leq&~A,\q\q 1\leq k\leq 8.
		\end{align}
In what follows, we slightly abuse the terminology and call $\lam$ the sound speed and call $\eps$ the Mach number. When discussing the incompressible limit ($\lam\gg1$ or equivalently $\eps\ll 1$), we sometimes write $\ffpm:=\frac{\p\ff_\eps^\pm}{\p p} (p^\pm,S^\pm)=\eps^2$ for simplicity. 

System \eqref{CMHDVS} is now converted into
\begin{equation}\label{CMHDVS0}
\begin{cases}
\rho^\pm \Dtpm v^\pm -\bpm b^\pm+\nabp q^\pm=0,~~q^\pm=p^\pm+\frac12|b^\pm|^2&~~\text{ in }[0,T]\times \Omega^\pm,\\
\ff_p^{\pm}\Dtpm p^\pm+\nabp\cdot v^\pm=0 &~~\text{ in }[0,T]\times \Omega^\pm,\\
p^\pm=p^\pm(\rho^\pm,S^\pm),~~\ff^\pm=\log \rho^\pm,~~\ff_p^\pm>0,~\rho^\pm\geq \bar{\rho_0}>0 &~~\text{ in }[0,T]\times \Omega^\pm, \\
\Dtpm b^\pm-\bpm v^\pm+b^\pm\nabp\cdot v^\pm=0&~~\text{ in }[0,T]\times \Omega^\pm,\\
\nabp\cdot b^\pm=0&~~\text{ in }[0,T]\times \Omega^\pm,\\
\Dtpm S^\pm=0&~~\text{ in }[0,T]\times \Omega^\pm,\\
\jump{q}=\sigma\cnab \cdot \left( \frac{\cnab \psi}{\sqrt{1+|\cnab\psi|^2}}\right) &~~\text{ on }[0,T]\times\Sigma, \\
\p_t \psi = v^\pm\cdot N &~~\text{ on }[0,T]\times\Sigma,\\
b^\pm\cdot N=0&~~\text{ on }[0,T]\times\Sigma,\\
v_d^\pm=b_d^\pm=0&~~\text{ on }[0,T]\times\Sigma^\pm,\\
{(v^\pm,b^\pm,\rho^\pm,S^\pm,\psi)|_{t=0}=(v_0^\pm, b_0^\pm, \rho_0^\pm, S_0^\pm,\psi_0)}.
\end{cases}
\end{equation}

Since the material derivatives are tangential to the boundary, that is, $\Dtpm=\Dtbpm:=\p_t+\vb^\pm\cdot\cnab$ on $\Sigma$ and $\Sigma^\pm$, the compatibility conditions \eqref{CMHDcomp} for initial data up to $m$-th order $(m\in\N)$ are now written as:
\begin{equation}\label{comp cond}
\begin{aligned}
\jump{\p_t^j q}|_{t=0}=\sigma\p_t^j\mathcal{H}|_{t=0}~~\text{ on }\Sigma,~~0\leq j\leq m,\\
\p_t^{j+1}\psi|_{t=0}=\p_t^{j}(v^\pm\cdot N)|_{t=0}~~\text{ on }\Sigma,~~0\leq j\leq m,\\
\p_t^j v_d^\pm|_{t=0}=0~~\text{ on }\Sigma^\pm,~~0\leq j\leq m.
\end{aligned}
\end{equation}The above conditions also imply $\p_t^{j}(b^\pm\cdot N)|_{t=0}=0$ on $\Sigma$ and $\Sigma^\pm$ for $0\leq j\leq m$. See \cite[Section 4]{Trakhinin2008CMHDVS} for details.

\subsubsection{Stability conditions for the zero-surface-tension limit}
We need to add some extra stability conditions on the free interface when surface tension is neglected, that is, when $\sigma=0$. We introduce the quantities $$a^\pm:=\sqrt{\rho^\pm\left(1+\left(\frac{c_A^\pm}{c_s^\pm}\right)^2\right)}$$ where $c_A^\pm:=|b^\pm|/\sqrt{\rho^\pm}$ represents the Alfv\'en speed (the speed of magneto-sonic waves), $c_s^\pm:=\sqrt{\p p^\pm/\p\rho^\pm}$ represents the sound speed.  The stability conditions are
\begin{align}
d=3:&~~0<a^\pm\left|\bc^\mp\times\jump{\vb}\right|<|\bc^+\times\bc^-|~~\text{ on }[0,T]\times\Sigma, \label{3D stable}\\
d=2:&~~\left(\frac{|b_1^+|}{a^+}+\frac{|b_1^-|}{a^-}\right)>|\jump{v_1}|>0 ~~\text{ on }[0,T]\times\Sigma, \label{2D stable}
\end{align} where we view the horizontal magnetic field $\bc=(b_1,b_2,0)^\top$ and the horizontal velocity  $\vb=(v_1,v_2,0)^\top$ as vectors lying on $\T^2\times\{x_3=0\}\subset\R^3$ to define the exterior product. The ``$>0$" part in \eqref{3D stable} and \eqref{2D stable} is necessary because vortex sheets naturally require the tangential discontinuity of velocity is nonzero. Thus, the stability conditions require that the strength of the magnetic fields cannot be too weak. Moreover, the condition for 3D case implies that $b^+$ and $b^-$ are not collinear on $\Sigma$ and the condition for 2D case requires certain quantitative relation between the strength of magnetic fields and the jump of tangential velocities. Note that the stability conditions are just initial constraints that can propagate within a short time interval. We will explain in later sections why such stability conditions are needed.

\subsection{An overview of previous results}\label{sect background}
For incompressible Euler equations, vortex sheets without surface tension tend to be violently unstable, which exhibit the so-called Kelvin-Helmholtz instability. We refer to \cite{Ebin1988,Wu2002VS, Wu2004VS} and references therein for the special case of 2D irrotational flows without surface tension and \cite{AM2007VS,CCS2007,SZ3} for the case of nonzero vorticity and nonzero surface tension. For compressible Euler equations, vortex sheets, unlike shock fronts, are characteristic discontinuities and the uniform Kreiss-Lopatinski\u{\i} condition is never satisfied, which leads to a potential loss of normal derivatives. For 3D Euler equations, compressible vortex sheets exhibit an analogue of Kelvin-Helmholtz instability \cite{Miles1, Miles2, SyrovatskiiEuler}; whereas for 2D Euler equations, Coulombel-Secchi \cite{Secchi2004CVS,Secchi2008CVS} proved the existence of ``supersonic" vortex sheets when the Mach number for the rectilinear background solution  exceeds $\sqrt{2}$, which is also a critical Mach number for the neutral stability. For the case of nonzero surface tension, we refer to Stevens \cite{Stevens2016CVS} for the local existence.

The Kelvin-Helmholtz instability can also be suppressed by magnetic fields. Unlike the transversal magnetic fields in MHD contact discontinuities ($B\cdot N|_{\Sigma(t)}\neq 0$) \cite{WangXinMHDCD}, the tangential magnetic fields must satisfy some constraints when the surface tension is neglected. For 3D incompressible ideal MHD, the following Syrovatski\u{\i} condition \cite{SyrovatskiiMHD} is assumed for local well-posedness
\begin{equation}\label{syrov}
\varrho^+|B^+\times\jump{u}|^2+\varrho^-|B^-\times\jump{u}|^2<(\varrho^++\varrho^-)|B^+\times B^-|^2.
\end{equation}Coulombel-Morando-Secchi-Trebeschi \cite{CMST2012MHDVS} proved the a priori estimate for the nonlinear problem under a more restrictive condition
\begin{equation}\label{syrov2}
\max\left\{\left|\frac{B^+}{\sqrt{\varrho^+}}\times\jump{u}\right|,\left|\frac{B^-}{\sqrt{\varrho^-}}\times\jump{u}\right|\right\}<\left|\frac{B^+}{\sqrt{\varrho^+}}\times \frac{B^-}{\sqrt{\varrho^-}}\right|.
\end{equation}Sun-Wang-Zhang \cite{SWZ2015MHDLWP} proved local well-posedness of the nonlinear problem under the original Syrovatski\u{\i} condition \eqref{syrov}. See also Liu-Xin \cite{LiuXin2023MHDVS} and Li-Li \cite{sb} for the study for both $\sigma>0$ and $\sigma=0$ cases.

For compressible current-vortex sheets in ideal MHD, it is unclear whether there is any \textit{necessary and sufficient condition} for the linear stability when the surface tension is neglected. Trakhinin \cite{Trakhinin2005CMHDVS} introduced a sufficient condition for the problem linearized around a background planar current-vortex sheet $(\hat{v}^\pm, \hat{b}^\pm, \hat{\rho}^\pm, \hat{S}^\pm)$ in flattened domains $\R^2\times\R_\pm$, which reads
\begin{equation}\label{3dstable}
|\jump{\hat{v}}|<|\sin(\bm{\alpha}^+-\bm{\alpha}^-)|\min\left\{\frac{\bm{\gamma}^+}{|\sin\bm{\alpha}^-|},\frac{\bm{\gamma}^-}{|\sin\bm{\alpha}^+|}\right\}\quad \text{ on }\{x_3=0\},
\end{equation}where $\bm{\gamma}^\pm:=\frac{c_A^\pm}{\sqrt{1+({c_A^\pm}/{c_s^\pm})^2}}$, $c_A^\pm:=|\hat{b}^\pm|/\sqrt{\hat{\rho}^\pm}$ represents the Alfv\'en speed, $c_s^\pm:=\sqrt{\p \hat{p}^\pm/\p\hat{\rho}^\pm}$ represents the sound speed, and $\bm{\alpha}^\pm$ represents the oriented angle between $\jump{\hat{v}}$ and $\hat{b}^\pm$. Indeed, \eqref{3dstable} is equivalent to \eqref{3D stable} for the constant-coefficient linearized problem. Also note that the formal incompressible limit of this condition as $\hat{\rho}^\pm\to 1$ is exactly \eqref{syrov2}, and it is easy to see \eqref{syrov2} implies \eqref{syrov}. The well-posedness was proven by Chen-Wang \cite{ChenWangCMHDVS} under a more restrictive condition and Trakhinin \cite{Trakhinin2008CMHDVS} under \eqref{3D stable}  with the help of Nash-Moser iteration. See also \cite{MSTY2023CMHDVS} for the 2D problem without surface tension under the stability condition \eqref{2D stable}. 

However, the study of compressible current-vortex sheets with surface tension is unavailable prior to the first part \cite{Zhang2023CMHDVS1} of this two-paper sequence. Also, the local existence results in previous works about free-surface compressible ideal MHD were all established by Nash-Moser iteration. There were quite few results about the energy estimates without loss of regularity: To the author's knowledge, the author's work \cite{Zhang2021CMHD} presents the only available result (except \cite{Zhang2023CMHDVS1}) about the energy estimates without loss of regularity for one-phase MHD without surface tension, but the energy estimates in \cite{Zhang2021CMHD} are never uniform in Mach number. {\bf The study of singular limits (e.g., the incompressible limit) of both inviscid vortex sheets and the free-surface ideal MHD flows remains open}. 

\subsection{Our goals}\label{sect goals}
We aim to present a comprehensive study of current-vortex sheets in ideal MHD and to develop a systematic approach to study the well-posedness and the singular limits for compressible inviscid fluids. The local well-posedness and the incompressible limit of \eqref{CMHDVS} for each fixed $\sigma>0$ has been established in \cite{Zhang2023CMHDVS1}. In this paper, we aim the prove the following results:
\begin{itemize}
\item The incompressible and zero-surface-tension limits under the stability conditions \eqref{3D stable} or \eqref{2D stable} in 3D or 2D respectively. It should be noted that these two limit process are independent of each other, that is, the energy estimates are uniform in both Mach number and $\sigma$ under \eqref{3D stable} or \eqref{2D stable}.
\item The improved incompressible limit such that the uniform boundedness (with respect to Mach number) of high-order ($\geq 2$) time derivatives of velocity fields can be dropped. 
\end{itemize}
To our knowledge, this paper presents \textit{the first} result about the incompressible limit for both inviscid vortex sheets and free-surface MHD flows.  

The study of singular limits of magneto-fluids has been a long-term research topic as stated in Majda's book \cite[Section 2.4]{Majdalimit}. For ideal compressible MHD, the singular-limit theory (e.g., the low Mach number limit, the small Alfv\'en number limit), which has shown great importance in the modelling of nuclear fusion or solar winds, is far less developed even for the fixed-domain problems, especially when the magnetic fields are tangential to the boundary. One of the major difficulty under the low Mach number setting is the \textit{``mismatch" of function spaces} for the local existence: The linearized problem of compressible ideal MHD in a fixed domain with $B\cdot N=0$ on the boundary is ill-posed in standard Sobolev spaces $H^l(l\geq 2)$ (see \cite{OS1998MHDill}), whereas the incompressible problem is well-posed in standard Sobolev spaces (e.g., see \cite{GuWang2016LWP}).  The anisotropic Sobolev spaces defined in Section \ref{sect anisotropic}, introduced by Shu-Xing Chen \cite{ChenSX}, were adopted in previous works about ideal compressible MHD (e.g., \cite{1991MHDfirst,Trakhinin2008CMHDVS,ChenWangCMHDVS,Zhang2021CMHD}), but none of the energy estimates obtained in these works is uniform in Mach number. The above issues are not resolved until the appearance of the author's recent preparatory work \cite{WZ2023CMHDlimit} (joint with J. Wang, for the fixed-domain problem in the reference domain $\T^{d-1}\times(-1,1)$) and this paper sequence.

When the fluid domain is $\R^d,\T^d$ or a domain with a fixed boundary, there have been numerous results concerning the low Mach number limit of compressible inviscid fluids (e.g., the pioneering works \cite{Klainerman1982limit, Ebin1982, Schochet1986limit, Metivier2001limit, Alazard2005limit} about Euler equations), including the case of ``ill-prepared (general) data" ($\dive v_0=O(1),~\p_t v|_{t=0}=O(1/\eps)$). However, for free-boundary problems, the study of singular limits remains largely open: There is even no robust framework to establish the low Mach number limit for inviscid fluids with ``well-prepared data"  ($\dive v_0=O(\eps),~\p_t v|_{t=0}=O(1)$). The results in previous works \cite{LL2018priori, Luo2018CWW, DL2019limit, Zhang2020CRMHD, LuoZhang2022CWWST} and the first part \cite{Zhang2023CMHDVS1} of this paper sequence require the uniform boundedness with respect to Mach number up to \textit{top-order} time derivatives. In the author's previous work \cite{Zhang2021elasto}, the boundedness assumption on high-order $(\geq 3)$ time derivatives\footnote{When $\sigma=0$, the Rayleigh-Taylor sign condition $-\nab_N q|_{\Sigma(t)}>0$ is necessary for the local existence. To propagate the Rayleigh-Taylor sign, one has to require $\p_t q\sim \p_t^2 v$ to be uniformly bounded. Very recently, Gu-Wang \cite{GuWang2023LWP} derived the $L_t^2$-type bound of $\p_t q$ for free-surface Euler equations with heat conduction, but this is completely contributied by the parabolic feature of the heat-conductive part.} was dropped for one-phase elastodynamic flows \textit{without} surface tension (also applicable to Euler equations), in which the framework is not applicable to two-phase fluids or fluids with surface tension. 

The zero-surface-tension limit, despite not a singular limit, rigorously shows that either surface tension or suitable magnetic field brings suppression effect on the analogue of Kelvin-Helmholtz instability for compressible vortex sheets. Also, in Section \ref{stat syrov}, we will see the comparison between the stabilization mechanism brought by surface tension and the one brought by magnetic fields satisfying \eqref{3D stable} or \eqref{2D stable}. 

In this paper, we developed \textit{a different framework (also applicable to Euler equations) that can help us drop the boundedness assumption of high-order time derivatives regardless of surface tension}. This framework is also applicable to compressible vortex-sheet problems if the two-phase fluids (not only for MHD, but also for Euler equations) are isentropic and the density functions are approaching the same constant as the Mach number goes to 0. In Section \ref{stat para} and Section \ref{stat ill}, we will see this method is the combination of a paradifferential approach applied to the evolution equation of the free interface and the analysis of wave equations for the pressure functions in $\eps^2$-weighted Sobolev spaces. %We will also see this new approach leaves open the possibility to establish the low Mach number limit of free-surface inviscid fluids with ``ill-prepared" (general) initial data.

\subsection{Main results}

\subsubsection{Anisotropic Sobolev spaces}\label{sect anisotropic}

Following the notations in \cite{ChenSX, WZ2023CMHDlimit}, we first define the anisotropic Sobolev space $H_*^m(\Omega^\pm)$ for $m\in\N$ and $\Om^\pm=\T^{d-1}\times\{0<\pm x_d<H\}$. Let $\omega=\omega(x_d)=(H^2-x_d^2)x_d^2$ be a smooth function on $[-H,H]$. The choice of $\omega(x_d)$ is not unique, as we just need $\omega(x_d)$ vanishes on $\Sigma\cup\Sigma^\pm$ and is comparable to the distance function near the interface and the boundaries. Then we define $H_*^m(\Omega^\pm)$ for $m\in\N^*$ as follows
\[
H_*^m(\Omega^\pm):=\left\{f\in L^2(\Omega^\pm)\bigg| (\omega\p_d)^{\alpha_{d+1}}\p_1^{\alpha_1}\cdots\p_d^{\alpha_d} f\in L^2(\Omega^\pm),~~\forall \alpha \text{ with } \sum_{j=1}^{d-1}\alpha_j +2\alpha_d+\alpha_{d+1}\leq m\right\},
\]equipped with the norm
\begin{equation}\label{anisotropic1}
\|f\|_{H_*^m(\Omega^\pm)}^2:=\sum_{\sum\limits_{j=1}^{d-1}\alpha_j +2\alpha_d+\alpha_{d+1}\leq m}\|(\omega\p_d)^{\alpha_{d+1}}\p_1^{\alpha_1}\cdots\p_d^{\alpha_d} f\|_{L^2(\Omega)}^2.
\end{equation} For any multi-index $\alpha:=(\alpha_0,\alpha_1,\cdots,\alpha_{d},\alpha_{d+1})\in\N^{d+2}$, we define
\[
\p_*^\alpha:=\p_t^{\alpha_0}(\omega\p_d)^{\alpha_{d+1}}\p_1^{\alpha_1}\cdots\p_d^{\alpha_d},~~\lee \alpha\ree:=\sum_{j=0}^{d-1}\alpha_j +2\alpha_d+\alpha_{d+1},
\]and define the \textbf{space-time anisotropic Sobolev norm} $\|\cdot\|_{m,*,\pm}$ to be
\begin{equation}\label{anisotropic2}
\|f\|_{m,*,\pm}^2:=\sum_{\lee\alpha\ree\leq m}\|\p_*^\alpha f\|_{L^2(\Omega^\pm)}^2=\sum_{\alpha_0\leq m}\|\p_t^{\alpha_0}f\|_{H_*^{m-\alpha_0}(\Omega^\pm)}^2.
\end{equation}

We also write the interior Sobolev norm to be $\|f\|_{s,\pm}:= \|f(t,\cdot)\|_{H^s(\Omega^\pm)}$ for any function $f(t,x)\text{ on }[0,T]\times\Omega^\pm$ and denote the boundary Sobolev norm to be $|f|_{s}:= |f(t,\cdot)|_{H^s(\Sigma)}$ for any function $f(t,x')\text{ on }[0,T]\times\Sigma$. 

From now on, we assume the dimension $d=3$, that is, $\Om^\pm=\T^2\times\{0<\pm x_3<H\} $, $\Sigma^\pm=\T^2\times\{x_3=\pm H\} $ and $\Sigma=\T^2\times\{x_3=0\}$. We will see the 2D case follows in the same manner as the 3D case up to slight modifications in the vorticity analysis and the analysis of stability condition when $\sigma=0$. 

\subsubsection{The existence theorem for compressible current-vortex sheets with surface tension}
Let us first record a theorem in \cite{Zhang2023CMHDVS1} about well-posedness and uniform-in-$\eps$ estimates of \eqref{CMHDVS0} for a fixed $\sigma>0$.
\begin{thm}\label{thm STLWP}
Fix the constant $\sigma>0$. Let $v_0^\pm, b_0^\pm, \rho_0 ^\pm,S_0^\pm \in H_*^8(\Om^\pm)$ and $\psi_0\in H^{9.5}(\Sigma)$ be the initial data of \eqref{CMHDVS0} satisfying 
\begin{itemize}
\item the compatibility conditions \eqref{comp cond} up to 7-th order;
\item the constraints $\nab^{\vp_0}\cdot b_0^\pm=0$ in $\Om^\pm$, $b^\pm\cdot N|_{\{t=0\}\times(\Sigma\cup\Sigma^\pm)}=0$ ;
\item $|\jump{\vb_0}|>0$ on $\Sigma$, $|\psi_0|_{L^{\infty}(\Sigma)}\leq 1$, and $E(0)\le M$ for some constant $M>0$.
\end{itemize}Then there exists $T_\sigma>0$ depending only on $M$ and $\sigma$, such that \eqref{CMHDVS0} admits a unique solution $(v^\pm(t),b^\pm(t),\rho^\pm(t),S^\pm(t),\psi(t))$ verifies the energy estimate
		    \begin{equation}
		    	\sup_{t\in[0,T]}E(t) \le C(\sigma^{-1})P(E(0))
		    \end{equation}and $\sup\limits_{t\in[0,T_\sigma]}|\psi(t)|<10<H$, where $P(\cdots)$ is a generic polynomial in its arguments. The energy $E(t)$ is defined to be

	        \begin{equation}
	        		\begin{aligned}\label{energy lwp}
	        			E(t):=&~E_4(t) + E_5(t) + E_6(t) + E_7(t) + E_8(t),\\
	        			E_{4+l}(t):=&\sum_\pm\sum_{\lee\alpha\ree=2l}\sum_{k=0}^{4-l}\left\|\left(\eps^{2l}\TT^{\alpha}\p_t^{k}\left(v^\pm, b^\pm,S^\pm,(\ffpm)^{\frac{(k+\alpha_0-l-3)_+}{2}}p^\pm\right)\right)\right\|^2_{4-k-l,\pm}\\
 &+\sum_{k=0}^{4+l}\left|\sqrt{\sigma}\eps^{2l}\p_t^{k}\psi\right|^2_{5+l-k}\quad 0\leq l\leq 4,
	        		\end{aligned}
	        \end{equation}
            where $k_+:=\max\{k,0\}$ for $k\in\R$ and we denote $\TT^{\alpha}:=(\omega(x_3)\p_3)^{\alpha_4}\p_t^{\alpha_0}\p_1^{\alpha_1}\p_2^{\alpha_2}$ to be a high-order tangential derivative for the multi-index $\alpha=(\alpha_0,\alpha_1,\alpha_2,0,\alpha_4)$ with length (for the anisotropic Sobolev spaces) $\lee \alpha\ree=\alpha_0+\alpha_1+\alpha_2+2\times0+\alpha_4$.	The quantity $\eps$ is the parameter defined in Section \ref{sect eos}. Moreover, the $H^{9.5}(\Sigma)$-regularity of $\psi$ can be recovered in the sense that 
\begin{align}
\sum_{l=0}^4\sum_{k=0}^{3+l}\left|\sigma\eps^{2l}\p_t^{k}\psi\right|^2_{5.5+l-k}\leq P(E(t)),\quad \forall t\in[0,T_\sigma].
\end{align}

	\end{thm} 

\subsubsection{Main result 1: Incompressible and zero-surface-tension limits}

 For any fixed $\sigma>0$, the energy estimates obtained in Theorem \ref{thm STLWP} are already uniform in $\eps$. When taking the limit $\sigma\to 0$, we shall impose suitable stability conditions on $\Sigma$ to ensure the well-posedness of ``$\sigma=0$-problem". Assume there exists a constant $\delta_0\in(0,\frac18)$ such that
\begin{align}
\delta_0\leq a^\pm\left|\bc^\mp\times\jump{\vb}\right|\leq (1-\delta_0)|\bc^+\times\bc^-|~~\text{ on }[0,T]\times\Sigma, \label{syrov 3D com}
\end{align} where 
\begin{align}
a^\pm:=\sqrt{\rho^\pm\left(1+\left(\frac{c_A^\pm}{c_s^\pm}\right)^2\right)}
\end{align} and $c_A^\pm:=|b^\pm|/\sqrt{\rho^\pm}$ represents the Alfv\'en speed (the speed of magnetosonic wave), $c_s^\pm:=\sqrt{\p p^\pm/\p\rho^\pm}$ represents the sound speed. It should be noted that \eqref{syrov 3D com} is not an imposed boundary condition for the ``$\sigma=0$"-problem. Instead, it is just a constraint for initial data which can propagate within a short time. In other words, we only need to assume 
\begin{align}
2\delta_0\leq (a^\pm|_{t=0})\left|\bc_0^\mp\times\jump{\vb_0}\right|\leq (1-2\delta_0)|\bc_0^+\times\bc_0^-|~~\text{ on }\Sigma. \label{syrov 3D com 1}
\end{align}

Under the stability condition \eqref{syrov 3D com}, we can establish the uniform-in-$(\es)$ energy estimates.
\begin{thm}[\textbf{Uniform-in-$(\es)$ estimates}] \label{thm CMHDlimit2}
Under the hypothesis of Theorem \ref{thm STLWP}, if the stability condition \eqref{syrov 3D com} holds, then there exists a time $T>0$ only depending on $M$, such that
\begin{align}
\sup_{0\leq t\leq T} \EW(t) \leq P(\EW(0)),
\end{align}where $\EW(t)$ is defined by
\begin{align}
\EW(t):=\sum_{l=0}^4\EW_{4+l}(t),\quad \EW_{4+l}(t)=E_{4+l}(t)+\sum_{k=0}^{4+l}\left|\eps^{2l}\p_t^{k}\psi\right|^2_{4.5+l-k}.
\end{align}
\end{thm}

Now we introduce the rigorous statement for the zero-surface-tension limit. For $\sigma \geq 0$, the motion of incompressible current-vortex sheets with surface tension are characterised by the equations of $(\xi^\sigma, w^{\pm,\sigma}, h^{\pm,\sigma})$ with initial data $(\xi_0^\sigma,w_0^{\pm,\sigma},h_0^{\pm,\sigma})$ and a transport equation of $\mathfrak{S}^{\pm,\sigma}$:
\begin{equation} \label{IMHDs}
\begin{cases}
\rr^{\pm,\sigma}(\p_t+w^{\pm,\sigma}\cdot\nab^{\Xi^\sigma})w^{\pm,\sigma}- (h^{\pm,\sigma}\cdot\nab^{\Xi^\sigma})h^{\pm,\sigma}+\nab^{\Xi^\sigma} \Pi^{\pm,\sigma}=0&~~~ \text{in}~[0,T]\times \Omega,\\
\nab^{\Xi^\sigma}\cdot w^{\pm,\sigma}=0&~~~ \text{in}~[0,T]\times \Omega,\\
(\p_t+w^{\pm,\sigma}\cdot\nab^{\Xi^\sigma}) h^{\pm,\sigma}=(h^{\pm,\sigma}\cdot\nab^{\Xi^\sigma})w^{\pm,\sigma}&~~~ \text{in}~[0,T]\times \Omega,\\
\nab^{\Xi^\sigma}\cdot h^{\pm,\sigma}=0&~~~ \text{in}~[0,T]\times \Omega,\\
(\p_t+w^{\pm,\sigma}\cdot\nab^{\Xi^\sigma})\mathfrak{S}^{\pm,\sigma}=0&~~~ \text{in}~[0,T]\times \Omega,\\
\jump{\Pi^{\sigma}}=\sigma\cnab \cdot \left( \frac{\cnab \xi^{\sigma}}{\sqrt{1+|\cnab\xi^{\sigma}|^2}}\right) &~~~\text{on}~[0,T]\times\Sigma,\\
\p_t \xi^{\sigma} = w^{\pm,\sigma}\cdot N^{\sigma} &~~~\text{on}~[0,T]\times\Sigma,\\
h^{\pm,\sigma}\cdot N^{\sigma}=0&~~~\text{on}~[0,T]\times\Sigma,\\
w_3^\pm=h_3^\pm=0&~~\text{ on }[0,T]\times\Sigma^\pm,\\
(w^{\pm,\sigma},h^{\pm,\sigma},\mathfrak{S}^{\pm,\sigma},\xi^{\sigma})|_{t=0}=(w_0^{\pm,\sigma},h_0^{\pm,\sigma},\mathfrak{S}_0^{\pm,\sigma}, \xi_0^{\sigma}), 
\end{cases}
\end{equation}where $\Xi^{\sigma}(t,x) = x_3+\chi(x_3) \xi^\sigma(t,x')$ is the extension of $\xi^\sigma$ in $\Omega$ and $ N^\sigma:=(-\TP_1\xi^\sigma, -\TP_2\xi^\sigma, 1)^\top$. The quantity $\Pi^\pm:=\bar{\Pi}^\pm+\frac12|h^\pm|^2$ represent the total pressure functions for the incompressible equations with $\bar{\Pi}^\pm$ the fluid pressure functions. The quantity $\rr^\pm$ satisfies the evolution equation $(\p_t+w^{\pm,\sigma}\cdot\nab^{\Xi^\sigma})\rr^{\pm,\sigma}=0$ with initial data $\rr_0^{\pm,\sigma}:=\rho^{\pm,\sigma}(0,\mathfrak{S}_0^{\pm,\sigma})$.

Denote $(\psi^\es, v^{\pm,\es}, b^{\pm,\es}, \rho^{\pm,\es}, S^{\pm,\es})$ to be the solution of \eqref{CMHDVS0} indexed by $\sigma$ and $\eps$. Under the stability conditions, we prove that $(\psi^\es, v^{\pm,\es}, b^{\pm,\es}, \rho^{\pm,\es}, S^{\pm,\es})$ converges to $(\xi^0, w^{\pm,0}, h^{\pm,0}, \rr^{\pm,0}, \mathfrak{S}^{\pm,0})$ as $\eps, \sigma \rightarrow 0$ provided the convergence of initial data. Here $(\xi^0, w^{\pm,0}, h^{\pm,0}, \rr^{\pm,0}, \mathfrak{S}^{\pm,0})$ represents the solution to incompressible current-vortex sheets system \eqref{IMHDs} with initial data $(\xi_0^0, w_0^{\pm,0}, h_0^{\pm,0}, \mathfrak{S}_0^{\pm,0})$ when $\sigma=0$. %The uniform estimates and Aubin-Lions compactness lemma immediately lead to the following  re

\begin{cor}[\textbf{Incompressible and zero-surface-tension limits}] \label{cor CMHDlimit3}
Let $(\psi_0^\es, v_0^{\pm,\es}, b_0^{\pm,\es}, \rho_0^{\pm,\es}, S_0^{\pm,\es})$ be the initial data of \eqref{CMHDVS0} for each fixed $(\eps, \sigma)\in \R^+\times \R^+$, satisfying 
\begin{enumerate}
\item [a.] The sequence of initial data $(\psi_0^\es, v_0^{\pm,\es}, b_0^{\pm,\es}, S_0^{\pm,\es}) \in H^{9.5}(\Sigma)\times H_*^8(\Omega^\pm)\times H_*^8(\Omega^\pm)\times H_*^8(\Omega^\pm)$ satisfies the hypothesis of Theorem \ref{thm STLWP}. 
\item [b.] $(\psi_0^\es, v_0^{\pm,\es}, b_0^{\pm,\es}, S_0^{\pm,\es}) \to (\xi_0^0, w_0^{\pm,0}, h_0^{\pm,0}, \mathfrak{S}_0^{\pm,0})$ in $H^{5.5}(\Sigma) \times H^4(\Omega^\pm)\times H^4(\Omega^\pm)\times H^4(\Omega^\pm)$ as $\eps, \sigma\to 0$. 
\item [c.] The incompressible initial data satisfies $|\jump{\wb_0^0}|>0$ on $\Sigma$, the constraints $\nab^{\xi_0}\cdot h_0^{0,\pm}=0$ in $\Om^\pm$, $h^{0,\pm}\cdot N^0|_{\{t=0\}\times(\Sigma\cup\Sigma^\pm)}=0$, the stability condition 
\begin{align}
\label{syrov 3D in 1} 2\delta_0\leq \sqrt{\rr_0^{\pm,0}}\left|\bar{h}_0^{\mp,0}\times\jump{\bar{w}_0^0}\right|\leq (1-2\delta_0)|\bar{h}_0^{+,0}\times \bar{h}_0^{-,0}|~~ \text{ on }\Sigma,
\end{align}where $\delta_0>0$ is the same constant as in \eqref{syrov 3D com}.
\end{enumerate} 
Then it holds that 
\begin{align}
(\psi^\es, v^{\pm,\es}, b^{\pm,\es}, S^{\pm,\es})\to(\xi^0, w^{\pm,0}, h^{\pm,0},\mathfrak{S}^{\pm,0}),
\end{align} weakly-* in $L^\infty([0,T]; H^{4.5}(\Sigma)\times(H^{4}(\Om^\pm))^3)$ and strongly in  $C([0,T]; H_{\text{loc}}^{4.5-\delta}(\Sigma)\times(H_{\text{loc}}^{4-\delta}(\Om^\pm))^3)$ after possibly passing to a subsequence. Here $T>0$ is the time obtained in Theorem \ref{thm CMHDlimit2}.
\end{cor}

\begin{rmk}[Stability conditions in 2D]
When taking the zero-surface-tension limit, the stability condition for compressible current-vortex sheets in 2D is
\begin{align}
\left(\frac{|b_1^+|}{a^+}+\frac{|b_1^-|}{a^-}\right)\geq (1+\delta_0)|\jump{v_1}|>0 ~~\text{ on }[0,T]\times\Sigma, \label{syrov 2D com}
\end{align}which is again propagated by the initial constraint
\begin{align}
\left(\frac{|b_1^+|}{a^+}+\frac{|b_1^-|}{a^-}\right)\bigg|_{t=0}\geq (1+2\delta_0)|\jump{v_{01}}|>0 ~~\text{ on }\Sigma. \label{syrov 2D com 1}
\end{align} The corresponding stability condition for the incompressible data is
\begin{align}
\label{syrov 2D in 1} \left(\frac{|h_{01}^{+,0}|}{\sqrt{\rr_0^{+,0}}}+\frac{|h_{01}^{-,0}|}{\sqrt{\rr_0^{-,0}}}\right)\geq (1+2\delta_0)\bno{\jump{w_{01}^{0}}}>0.
\end{align}
\end{rmk}

\subsubsection{Main result 2: Dropping boundedness assumptions of high-order time derivatives}\label{stat weaker E}

 The uniform-in-$\eps$ estimates obtained in Theorem \ref{thm STLWP} and Theorem \ref{thm CMHDlimit2} require $\nabp \cdot v_0=O(\eps^2)$ and $\p_t^{k}v|_{t=0}=O(1)$ for $k\leq 4$. Such assumption is much stronger than the widely used definition of ``well-prepared" initial data ($\nabp\cdot v_0=O(\eps)$ and $\p_t v|_{t=0}=O(1)$). When $\jump{\rho}=O(\eps)$ on the interface $\Sigma$ (see Remark \ref{small rho} below), we can still prove the incompressible limit under the assumption $\nabp\cdot v_0=O(\eps),~\p_t v|_{t=0}=O(1)$ without any boundedness assumptions on high-order ($\geq 2$) time derivatives. However, the energy functional should also be modified. We define 
\begin{align}
\label{energy total} \EE(t):=\EE_{4}(t)+E_{5}(t)+E_{6}(t)+E_{7}(t)+E_{8}(t)\\
\label{energy es total}\EEW(t):=\EEW_{4}(t)+\EW_{5}(t)+\EW_{6}(t)+\EW_{7}(t)+\EW_{8}(t)
\end{align} where $\EE_4(t)$ is defined as the following
	        \begin{equation}
	        		\begin{aligned}\label{energy weak}
\EE_4(t)=&\sum_\pm\ino{(v^\pm,b^\pm,p^\pm)}_{4,\pm}^2+\ino{\p_t(v^\pm,b^\pm,\eps p^\pm)}_{3,\pm}^2 +\sum_{k=2}^{4}\ino{\eps\p_t^k\left(v^\pm,b^\pm,(\ffpm)^{\frac{(k-3)_+}{2}}p^\pm\right)}_{4-k,\pm}^2\\
&+\bno{\sqrt{\sigma}\psi}_{5}^2+\bno{\sqrt{\sigma}\p_t\psi}_{4}^2+\sum_{k=2}^{4}\bno{\sqrt{\sigma}\eps\p_t^k\psi}_{5-k}^2,
	        		\end{aligned}
	        \end{equation}and 
\begin{align}\label{energy es weak}
\EEW_4(t)=\EE_4(t)+\bno{\psi}_{4.5}^2+\bno{\p_t\psi}_{3.5}^2+\bno{\p_t^2\psi}_{2.5}^2+\bno{\eps\p_t^3\psi}_{1.5}^2+\bno{\eps\p_t^4\psi}_{0.5}^2.
\end{align}

\begin{thm}[\textbf{Improved uniform estimates}]\label{thm prepared}
Assume the fluids in $\Om^\pm$ are isentropic and the initial density functions satisfy $|\jump{\rho_0}|_{1.5}\leq C_0\eps$ on $\Sigma$ for some $C_0>0$. Under the hypothesis of Theorem \ref{thm STLWP}, the assumption $\EE(0)\le M'$ for some constant $M'>0$, there exists $T_\sigma'>0$ depending only on $M'$ and $\sigma^{-1}$ such that the solution $(v^\pm(t),b^\pm(t),\rho^\pm(t),\psi(t))$ to system \eqref{CMHDVS0} verifies the uniform-in-$\eps$ energy estimate
		    \begin{equation}
		    	\sup_{t\in[0,T_\sigma']}\EE(t) \leq C(\sigma^{-1})P(\EE(0)).
		    \end{equation}Furthermore, under the stability condition \eqref{syrov 3D com} and $\EEW(0)\leq M'$, there exists $T'>0$ depending only on $M'$ such that such that the solution $(v^\pm(t),b^\pm(t),\rho^\pm(t),\psi(t))$ to system \eqref{CMHDVS0} verifies the uniform-in-$(\es)$ energy estimate
		    \begin{equation}
		    	\sup_{t\in[0,T']}\EEW(t) \leq P(\EEW(0)).
		    \end{equation}
\end{thm} 
\begin{rmk}
Since $\p_t v|_{t=0}=O(1)$ still remains bounded, the above uniform estimates directly give the same stronge convergence results as in Corollary \ref{cor CMHDlimit3}. We do not repeat the statement of convergence theorems here. The result is also true for 2D case under the stability condition \eqref{syrov 2D com}.
\end{rmk}
\begin{rmk}[The smallness assumption on the density jump]\label{small rho}
The assumption $|\jump{\rho_0}|_{1.5}\leq C_0\eps$ on $\Sigma$ implies that $|\jump{\rho(t)}|_{1.5}\leq C_1\eps$ on $[0,T']\times \Sigma$ for some $C_1>0$. To achieve this, one has to assume the fluids are isentropic and that is why the entropy $S$ is deleted in $\EE$ and $\EEW$. Indeed, taking the incompressible limit yields $\jump{\rr}=0$ on $\Sigma$ for the density functions $\rr^\pm$. If the fluids are non-isentropic, then $\rr^\pm$ are not constants and only satisfy $(\p_t+\wb^\pm\cdot\cnab) \rr^\pm=0$ on $\Sigma$. Since vortex-sheet problems require $\jump{\wb}\neq 0$ on $\Sigma$, it is not possible to have $\rr^+(t)=\rr^-(t)$ on $\Sigma$ even if it holds at $t=0$.
\end{rmk}

\medskip

\noindent\textbf{List of Notations: }
In the rest of this paper, we sometimes write $\TT^k$ to represent a tangential derivative $\TT^{\alpha}$ in $\Om^\pm$ with order $\lee\alpha\ree =k$ when we do not need to specify what the derivative $\TT^{\alpha}$ contains. We also list all the notations used in this manuscript.
\begin{itemize}
\item $\Omega^\pm:=\T^{d-1}\times\{0<\pm x_d<H\}$, $\Sigma:=\T^{d-1}\times\{x_d=0\}$ and $\Sigma^\pm:=\T^{d-1}\times \{x_d=\pm H\}$, $d=2,3$.
\item $\|\cdot\|_{s,\pm}$:  We denote $\|f\|_{s,\pm}:= \|f(t,\cdot)\|_{H^s(\Omega^\pm)}$ for any function $f(t,x)\text{ on }[0,T]\times\Omega^\pm$.
\item $|\cdot|_{s}$:  We denote $|f|_{s}:= |f(t,\cdot)|_{H^s(\Sigma)}$ for any function $f(t,x')\text{ on }[0,T]\times\Sigma$.
\item $\|\cdot\|_{m,*}$: For any function $f(t,x)\text{ on }[0,T]\times\Omega$, $\|f\|_{m,*,\pm}^2:= \sum\limits_{\lee \alpha\ree\leq m}\|\p_*^\alpha f(t,\cdot)\|_{0,\pm}^2$ denotes the $m$-th order space-time anisotropic Sobolev norm of $f$.
\item $P(\cdots)$:  A generic polynomial with positive coefficients in its arguments;
%\item $\PP_0$:  $\PP_0=P(E(0))$;
\item $[T,f]g:=T(fg)-fT(g)$, and $[T,f,g]:=T(fg)-T(f)g-fT(g)$, where $T$ denotes a differential operator and $f,g$ are arbitrary functions.
\item $\TP$: $\TP=\p_1,\cdots,\p_{d-1}$ denotes the spatial tangential derivative.
\item $A\eql B$: $A$ is equal to $B$ plus some lower-order terms that are easily controlled.
\end{itemize}

\section{Strategy of the proof}\label{stat}
By the compactness argument, it suffices to establish the uniform-in-$(\es)$ estimates under the stability condition \eqref{syrov 3D com}. Let us briefly recall the analysis of $E(t)$, defined by \eqref{energy lwp}, in the first paper of this sequence \cite{Zhang2023CMHDVS1}. The first step is to reduce the normal derivatives via div-curl analysis. The divergence part is automatically converted to tangential estimates with suitable weights of Mach number thanks to the continuity equation. The control of curl part requires the anisotropic Sobolev norms because there is an anisotropic structure of Lorentz force, first discovered in the preparatoty work \cite{WZ2023CMHDlimit}, indicating that we shall trade one normal derivative (in the curl operator $\nab\times$) for $\eps^2$-weighted, second-order tangential derivative $\eps^2 D_t^2$. This fact exactly explains why we need 8-th order tangential regularity with $\eps^8$ weight (corresponding to the term $\EW_8(t)$ in the energy functional) to close the control of 4-th order full Sobolev norms ($\EW_4(t)$ in the energy). 

In \cite{Zhang2023CMHDVS1}, we adopted the div-curl inequality \eqref{divcurlTT} and the reduction of normal derivatives is also recorded in Section \ref{sect lwp energy} of this paper. It remains to establish the tangential estimates for $\eps^{2l}\TT^{\alpha}\TT^{\beta}\p_t^k$ where $\TT^{\alpha}=(\omega(x_3)\p_3)^{\alpha_4}\p_t^{\alpha_0}\p_1^{\alpha_1}\p_2^{\alpha_2}$ and $\alpha,\beta,k,l$ satisfy
\begin{equation}\label{TT index}\lee\alpha\ree=2l,~\lee\beta\ree=4-l-k,~0\leq k\leq 4-l,~0\leq l\leq4~\text{ and }~\beta_0=0.\end{equation} It should be noted that the $\eps^{2l}\TT^\alpha$-part appears due to the weighted anisotropic structure in the vorticity analysis for $E_{4+l}$ and the $\TT^{\beta}\p_t^k$-part comes from the interior tangential derivatives in div-curl inequality (this part is parallel to Euler equations). See the analysis in \cite[Section 2.1-2.2]{Zhang2023CMHDVS1} for more details.

\subsection{Zero-surface-tension limit under the stability conditions}\label{stat syrov}
Let us rewrite $\TT^\gamma=\TT^{\alpha}\TT^{\beta}\p_t^k$. Among the major terms arising in tangential estimates, the most difficult one is the following  VS term contributed by the tangential discontinuity of velocity fields
\begin{align}
\VS:=\eps^{4l}\is \TT^\gamma q^- (\jump{\vb}\cdot\cnab)\TT^\gamma\psi\dx',
\end{align} where $\TT^\gamma=\p_t^{k}\TP^{4+l-k}$, $0\leq k\leq 4+l,~0\leq l\leq 4$ (Note that $\TT^\gamma$ vanishes on $\Sigma$ when $\gamma_4>0$). This term must appear in current-vortex sheet problems due to $\jump{\vb}\neq 0$, whereas it vanished in the study of MHD contact discontinuity \cite{WangXinMHDCD} ($\jump{v}|_{\Sigma}=0,~B\cdot N|_{\Sigma}\neq 0$). In \cite{Zhang2023CMHDVS1}, we use the ellipticity of the mean curvature to get enhanced regularity of $\psi$ and $H^{\frac12}$-$H^{-\frac12}$ duality to control this term when $\TT^\gamma$ contains at least one spatial derivative, and present rather delicate analysis to control the full-time derivatives with the help of some symmetric structures and lots of technical modifications. In this paper, we want to discuss the zero-surface-tension limit under the stability condition \eqref{syrov 3D com} and the dependece on $\sigma^{-1}$ should be dropped. Thus, we must avoid using the ellipticity of the equation $\jump{q}=\sigma\h$ and find a new way to control or modify or eliminate VS term with the help of \eqref{syrov 3D com}. Roughly speaking, the stability condition \eqref{syrov 3D com} brings the following two benefits that Euler equations do not enjoy:
\begin{itemize}
\item [a.] Enhance the regularity of the free surface to $H^{s+\frac12}(\Sigma)$ in $\TP^s$-estimates (possibly with suitable $\eps$-weights). This gives 1/2-order higher regularity of $\psi$ than the Rayleigh-Taylor sign condition does.

\item [b.] \textbf{Completely eliminate the problematic term VS}.
\end{itemize}

The enhanced regularity in (a) is easy to prove, as the ``non-collinearity" allows us to resolve $\cnab\psi$ in terms of $b^\pm$ without any derivative. The benefit (b) is rather important. When taking the vanishing surface tension limit, the enhanced regularity obtained in (a) is still not enough to control VS. Instead, we completely eliminate the contribution of $\jump{\vb}$ in the term VS by inserting a suitable term involving the magnetic fields. Specifically, we want to insert a term $\mu^\pm \bc^\pm$ into VS to get
\begin{equation}\label{TT VS'}
\VS':=\is \TT^\gamma q^-\left(\jump{\vb-\mu\bc}\cdot\cnab\right)\TT^\gamma\psi\dx'
\end{equation} and find suitable functions $\mu^\pm$ such that $\jump{\vb-\mu \bc}=\mathbf{0}$ on $\Sigma$. Under the stability condition \eqref{syrov 3D com}, the functions $\mu^\pm$ uniquely exist. To construct the term$\VS'$ from MHD equations \eqref{CMHDVS0}, we shall replace the variable $v^\pm$ in the momentum equation by $v^\pm-\mu^\pm b^\pm$. However, this operation makes the MHD system not symmetric and consequently the energy estimates cannot be closed. To overcome this difficulty, we introduce the ``Friedrichs secondary symmetrization" \cite{Friedrichs}, which was first applied to compressible ideal MHD by Trakhinin \cite{Trakhinin2005CMHDVS}, to re-symmetrize the MHD system. %See also \cite{CMST2012MHDVS,MSTY2023CMHDVS}.

The final step is to determine the range for $\mu^\pm$ such that the energy estimates for the secondary-symmetrized MHD system can be closed. Using the idea of Alinhac good unknowns \cite{Alinhac1989good}, the energy for $\VV,\BB,\PP$ (defined in Section \ref{sect AGU}) becomes
\begin{align}
&\frac12\ddt\iopm \rho^\pm|\VV^\pm|^2+|\BB^\pm|^2+\ffpm|\PP^\pm|^2 -2\mu^\pm \rho^\pm\VV^\pm\cdot\BB^\pm-2\mu^\pm\rho^\pm\ffpm\PP^\pm(b^\pm\cdot\VV^\pm)\dvt,
\end{align} where $\ffpm=1/(\rho^\pm (c_s^\pm)^2)$ and $\dvt:=\p_3\vp\dx$. We must guarantee the above quadratic form of $(\VV,\BB,\PP)$ to be positive-definite, which is equivalent to guarantee the hyperbolicity. This requires $\mu^\pm$ to satisfy $(\mu^\pm)^2\rho^\pm(1+(c_A^\pm/c_s^\pm)^2)<1$. This inequality gives the range of $\mu^\pm$, which exactly coincides with the stability condition \eqref{syrov 3D com}.

\begin{rmk}[Stabilization effects on 2D subsonic vortex sheets]
In the 2D case, the non-collinearity property no longer holds because the interface is 1D. The functions $\mu^\pm$ still exist but are not unique. For a rectilinear piecewise-smooth background solution $(\pm\underline{v},0,\pm\underline{b},0,\underline{p}^\pm,\underline{S}^\pm)^\top$, condition \eqref{syrov 2D com} implies that $|\underline{v}|^2<c_s^2\frac{c_A^2}{c_A^2+c_s^2}<c_s^2$, that is, the background solution must be a subsonic flow, whereas the linear stability only holds for supersonic flow, that is, $|v|/c_s>\sqrt{2}$, for 2D vortex sheets of compressible Euler equations. Thus, sufficiently strong magnetic fields have stabilization effects on 2D subsonic vortex sheets. %However, this range is only a subset of the subsonic zone for the linear neutral stability obtained in Wang-Yu \cite{WangYu2013CMHDVS}. It still remains open to justify the nonlinear stability \textit{in the whole domain} for the linear stability obtained in Wang-Yu \cite{WangYu2013CMHDVS}. 
\end{rmk} 

\subsection{A paradifferential approach for the low Mach number limit of vortex sheets}\label{stat prepared}

\subsubsection{Difficulty caused by free-surface motion}
As stated in Section \ref{stat weaker E}, when $\eps$ is suitably small, the incompressible limit can be established under the assumption $\nabp\cdot v_0=O(\eps),~\p_t v|_{t=0}=O(1)$ without any redundant restrictions on higher-order time derivatives. This is not difficult under the fixed-domain setting, but for free-boundary problems, we need to use the energy defined in \eqref{energy total}-\eqref{energy es weak} and new essential difficulties caused by the free-interface motion. For example, in the $\TP^3\p_t$-estimate (without $\eps$-weight) that arises in $\EE_4(t)$, we have to control integrals in the following form
\[
\iopm (\TP^2\p\p_t v_i^\pm)(\TP\NN_i)(\TP^2\p \p_t q^\pm)\dx,\text{ which arises from }\iopm (\TP^3\p_t v_i^\pm) [\TP^3, \NN_i, \TP^3\p_t q]\dx.
\]  The simultaneous appearance of $\TP^2\p\p_t v$ and $\TP^2\p \p_t q$ causes a loss of $\eps$-weight.
The appearance of such loss of $\eps$-weight is actually necessary when $\TT^\gamma v$ ($\gamma_0<\len{\gamma}$) is assigned a different $\eps$-weight from that of $\TT^\gamma q$, because the normal vector $\NN$ may not necessarily absorb a time derivative when $\TT^\gamma$ contains both $\TP$ and $\p_t$. \textbf{This difficulty is completely caused by the free-interface motion} and there is no such difficulty for the fixed-domain problem \cite{WZ2023CMHDlimit}. It should be noted that for one-phase flow without surface tension, where the pressure function vanished on the free surface, this difficulty can be avoided by combining the elliptic estimates and the interior estimates of the wave equation for the pressure. This was proven in the author's previous work \cite{Zhang2021elasto}, but the proof heavily relies on $q|_{\Sigma}=0$.

\subsubsection{Improved estimates for double limits: paralinearization of the free-interface motion}\label{stat para}
To control $\EE_4$ and $\EEW_4$ in \eqref{energy weak}-\eqref{energy es weak}, we only need to re-consider the estimates of $\|v_t\|_3$ and $\|b_t\|_3$ because $\p_t^k v$ and $\p_t^kq$ share the same weights of Mach number for $k\geq 2$ and the control of $\geq 2$ time derivatives should be the same as $E_4(t)$ in \cite{Zhang2023CMHDVS1}. To avoid interior tangential estimates, we apply the div-curl inequality \eqref{divcurlNN} to $v_t,b_t$ and reduce the control to their normal traces $|v_t\cdot N|_{2.5}$ and $|b_t\cdot N|_{2.5}$. In view of the boundary conditions, we must seek for other ways to control $|\p_t^k\psi|_{4.5-k}$ for $0\leq k\leq 2$ and they must be $\sigma$-indepedent when taking the double limits $\es\to 0$. 

We now consider the time-differentiated kinematic boundary condition, which together with the momentum equation gives
\[
\rho^\pm\psi_{tt}=-N\cdot\nabp q^\pm+(\bc_i^\pm\bc_j^\pm-\rho^\pm\vb_i^\pm\vb_j^\pm)\TP_i\TP_j\psi+\cdots
\]Motivated by Shatah-Zeng \cite{SZ3}, we try to separate the boundary values of $q^\pm$ from the interior contribution of $q^\pm$. Specifically, $q^\pm$ satisfies a two-phase wave equation (we write $\ffpm=\eps^2$ for convenience)
\[
\eps^2(\Dtpm)^2 q^\pm - \lapp q^\pm =\eps^2(\Dtpm)^2(\frac12|b^\pm|^2)+(\pp_i v_j^\pm)(\pp_j v_i^\pm)-(\pp_i b_j^\pm)(\pp_j b_i^\pm)\text{ in }\Om^\pm,\quad \p_3 q^\pm|_{\Sigma^\pm}=0,~~\jump{q}|_{\Sigma}=\sigma\h(\psi).
\]and we introduce the decomposition $q^\pm=q_\psi^\pm+q_w^\pm$ with
\begin{align*}
-\lapp q_\psi^\pm=0\text{ in }\Om^\pm,~~q_\psi^\pm=q^\pm\text{ on }\Sigma,~~\p_3 q_\psi^\pm=0\text{ on }\Sigma^\pm,\\
- \lapp q_w^\pm =-\eps^2(\Dtpm)^2(q^\pm-\frac12|b^\pm|^2)+(\pp_i v_j^\pm)(\pp_j v_i^\pm)-(\pp_i b_j^\pm)(\pp_j b_i^\pm)\text{ in }\Om^\pm,~~q_w^\pm=0\text{ on }\Sigma,~~\p_3 q_w^\pm=0\text{ on }\Sigma^\pm.
\end{align*}Under this setting, we can write $-N\cdot \nabp q^\pm = \pm \dnpm (q^\pm|_{\Sigma}) - N\cdot \nabp q_w^\pm$ where $\dnpm$ represents the Dirichlet-to-Neumann (DtN) operators with respect to $\Om^\pm$ and $\psi$ (defined in Section \ref{sect DtNST}). The traces of $q^\pm$ on $\Sigma$ can be resolved by inverting the DtN operators, which then gives us the following evolution equation
\begin{align}
(\rho^++\rho^-)\p_t^2\psi=&~\frac{\sigma}{2}\left(\dnp+\dnm\right)(\h(\psi)) +\left(\bc_i^+\bc_j^+-\rho^+\vb_i^+\vb_j^+ + \bc_i^-\bc_j^--\rho^-\vb_i^-\vb_j^-\right)\TP_i\TP_j\psi  \no\\
& - N\cdot\nabp q_w^+ - N\cdot\nabp q_w^-+\underline{(\dnp-\dnm)(\dnp+\dnm)^{-1}(\jump{\rho}\p_t^2\psi)}+\cdots\label{psiequ0}
\end{align} Using the paralinearization tools in Alazard-Burq-Zuily \cite{ABZ2011IWWST,ABZ2014IWW} and Alazard-M\'etivier \cite{AMDtN}, the principal symbol of the major term $\left(\dnp+\dnm\right)(\h(\psi))$ is negative and of the third order. Besides, the stability condition \eqref{syrov 3D com} ensures the ellipticity of the second term on the right side. The contribution of $q_w$ is completely reduced to the source term of the wave equation thanks to $q_w|_{\Sigma}=0$. Thus, we can simultaneously obtain the estimates of $|\psi|_{4.5},~|\sqrt{\sigma}\psi|_5$ and $|\psi_t|_{3.5}$ by taking a suitable 3.5-th order paradifferential operator in \eqref{psiequ0}  and we refer to Section \ref{sect uniform psi eq} for details.

For the control of $|\psi_{tt}|_{2.5}$, it suffices to take $\p_t$ in \eqref{psiequ0} and take a suitable 2.5-th order  paradifferential operator. However, the underlined term in \eqref{psiequ0}, after taking $\p_t$, contains third-order time derivative, whose $H^{2.5}(\Sigma)$ norm is bounded by $|\jump{\rho}|_{1.5}|\p_t^3\psi|_{1.5}$. In general, we have a loss of $\eps$-weight, as the energy $\EE(t)$ only gives $\p_t^3\psi=O(\eps^{-1})$. So, if we additionally require $|\jump{\rho}|_{1.5}=O(\eps)$, which can be fulfilled for isentropic flows according to Remark \ref{small rho}, this extra $\eps$-weight could compensate the loss of $\eps$-weight arising in this term. Hence, under the extra assumption $|\jump{\rho}|_{1.5}=O(\eps)$, we can control the accelaration of the free interface uniformly in $\eps$ without assuming the boundedness of high-order time derivatives of $v$.
\begin{rmk}
As we can see, {\bf the framework above is also applicable to compressible vortex sheets with surface tension for Euler equations, not just MHD system itself}. In fact, when $b^\pm=0$, the second term on the right side of \eqref{psiequ0} is controlled by the energy contributed by surface tension. Since the vorticity analysis of Euler equations does not produce high-order terms. Then the estimates of $\EE_4(t)+E_5(t)$ (or $\EEW_4(t)+\EW_5(t)$ for the double limits) can be closed by applying the above framework. Do note that $E_5(t)$ is necessary to close the uniform estimates because of the term $\eps^2(\Dtp)^2 q$. This part vanishes as $\eps\to 0$.
\end{rmk}

\begin{rmk}[Comparison with the Syrovatski\u{\i} condition for incompressible MHD]
The original Syrovatski\u{\i} stability condition (cf. Syrovatski\u{\i} \cite{SyrovatskiiMHD} or Landau-Lifshitz-Pitaevski\u{\i} \cite[\S 71]{landau}) for $\rho^\pm=1$ is $|\bar{h}^+\times\jump{\wb}|^2+|\bar{h}^-\times\jump{\wb}|^2<2|\bar{h}^+\times\bar{h}^-|^2$  on $\Sigma$, which is less restrictive than the formal incompressible limit of \eqref{syrov 3D com} as $\rho^\pm\to 1$. We recall that, in the analysis of equation \eqref{psiequ0}, the compressibility introduces an extra term $\eps^2(\Dtp)^2 p$ in the source term of $q_w$ that should be controlled via interior estimates. Since $\Om^+,\Om^-$ are disconnected domains, it is reasonable to have restrictions for solutions in $\Om^+$ and $\Om^-$ respectively. Besides, the extra term $\eps^2(\Dtp)^2 p$ presents a loss of derivative in standard Sobolev spaces and again indicates that one should trade a normal derivative for two tangential derivatives with $\eps^2$ weights.

The difference between the two stability conditions is related to the singular nature of incompressible limit. One can further see such difference from the derivation of these stability conditions and we refer to Trakhinin \cite{Trakhinin2005CMHDVS} for details. Similar situation also occurs in the case of 2D. %That is, the formal incompressible limit of \eqref{syrov 2D com} is more restrictive than the stability condition (cf. \cite{Axford2D}) for 2D incompressible current-vortex sheets.
\end{rmk}

\subsubsection{Comparison with one-phase problems}\label{stat ill}
Finally, we briefly discuss the differences between one-phase problems and vortex-sheet problems. Without loss of generality, we assume everything in $\Om^-$ is vanishing for one-phase problems and thus the term VS is vanishing and other major boundary terms can be similarly analyzed under the Rayleigh-Taylor sign condition $\p_3 q^+\geq c_0>0$. When considering the incompressible limit without boundedness assumptions on high-order time derivatives, compressible vortex-sheet problems exhibit new essential difficulty because the Dirichlet-type boundary condition $q|_{\Sigma}=\sigma\h$ becomes a jump condition $\jump{q}=\sigma\h$. Thus, we have to invert the DtN operators to resolve the traces of $q^\pm$ on $\Sigma$. Since the fluids are compressible, we cannot directly apply $(\dnpm)^{-1}$ to $(1/\rho^\pm)\dnpm(q^\pm|_{\Sigma})$ as in the incompressible case \cite{LiuXin2023MHDVS,sb} because the product of a harmonic function and $1/\rho^{\pm}$ is no longer harmonic. Then the bad term $(\dnp-\dnm)(\dnp+\dnm)^{-1}(\jump{\rho}\p_t^2\psi)$ must appear on the right side of \eqref{psiequ0} which already contains a second-order time derivative. Without the assumption $\jump{\rho}=O(\eps)$ on $\Sigma$, which can be achieve by requiring the fluids to be isentropic and the density functions converge to the same constant, there exhibits a loss of $\eps$-weight in the control of $\psi_{tt}$ in general. That is to say, under the assumption $\nabpp\cdot v_0=O(\eps)$, the accelaration of the free interface may be still not uniformly bounded in $\eps$ for compressible vortex sheets. Such loss of $\eps$-weight never appears in the fixed-domain problems \cite{Alazard2005limit, WZ2023CMHDlimit}, one-phase problems \cite{Zhang2021elasto, LuoZhang2022CWWST} or incompressible (current-)vortex sheets \cite{SWZ2015MHDLWP,LiuXin2023MHDVS,sb}. %The above discussion also opens the possibility to study the low Mach limit for one-phase free-surface compressible fluids (possibly with surface tension) under ill-prepared initial data ($\dive v_0=O(1),~\p_t v|_{t=0}=O(1/\eps)$, also called ``fast singular limit" \cite{Schochet1994limit}), which has never been studies before, as neither extra time derivatives nor loss of $\eps$-weight in the interior estimates appear when using the method in Section \ref{stat para}.

\section{Reduction of normal derivatives}\label{sect lwp energy}

Compared with the analysis in \cite{Zhang2023CMHDVS1}, the control of $\EW(t)$ only needs modifications in tangential estimates, as the VS term must be controlled uniformly in $\sigma$. In this section, we record the reduction of normal derivatives that has been analyzed in \cite{Zhang2023CMHDVS1} without repeating the proof.

\subsubsection*{Div-Curl analysis and reduction of pressure}\label{sect divcurl}
Using the div-curl decomposition (Lemma \ref{hodgeTT}), we have for $0\leq l\leq 3,~0\leq k\leq 3-l,~\len{\alpha}=2l,~\alpha_3=0$, we have
\begin{align}
\ino{\eps^{2l}\p_t^k\TT^\alpha (v^\pm,b^\pm)}_{4-k-l,\pm}^2\leq C\bigg(&\ino{\eps^{2l}\p_t^k\TT^\alpha (v^\pm,b^\pm)}_{0,\pm}^2+\ino{\eps^{2l}\nabp\cdot \p_t^k\TT^\alpha (v^\pm,b^\pm)}_{3-k-l,\pm}^2 \no\\
&+\ino{\eps^{2l}\nabp\times\p_t^k\TT^\alpha (v^\pm,b^\pm)}_{3-k-l,\pm}^2+\ino{\eps^{2l}\TP^{4-k-l}\p_t^k\TT^\alpha (v^\pm,b^\pm)}_{0,\pm}^2\bigg)\label{divcurlE}
\end{align} with $$C=C\left(\sum_{j=0}^l\sum_{k=0}^{3+j}|\eps^{2j}\p_t^{j}\psi|_{4+l-j}^2,|\cnab\psi|_{W^{1,\infty}}\right)>0$$ a positive continuious function in its arguments.
The conclusion for the div-curl analysis is 
\begin{prop}[{\cite[Proposition 3.8]{Zhang2023CMHDVS1}}] \label{prop divcurl}
Fix $l\in\{0,1,2,3\}$. For any $0\leq k\leq l-1$, any multi-index $\alpha$ satisfying $\len{\alpha}=2l$ and any constant $\delta\in(0,1)$, we can prove the following \textit{uniform-in-}$(\es)$ estimates for the curl part
\begin{equation}
\begin{aligned}
&\ino{\eps^{2l}\nabp\times\p_t^k\TT^\alpha v^\pm}_{3-k-l,\pm}^2+\ino{\eps^{2l}\nabp\times\p_t^k\TT^\alpha b^\pm}_{3-k-l,\pm}^2\\
\lesssim &~\delta E_{4+l}(t)+P\left(\sum_{j=0}^{l} \EW_{4+j}(0)\right)+P(E_4(t))\int_0^t P\left(\sum_{j=0}^l \EW_{4+j}(\tau)\right) + E_{4+l+1}(\tau)\dtau,
\end{aligned}
\end{equation} and for the divergence part
\begin{equation}
\begin{aligned}
&\ino{\eps^{2l}\nabp\cdot \p_t^k\TT^\alpha v^\pm}_{3-k-l,\pm}^2 +\ino{\eps^{2l}\nabp\cdot \p_t^k\TT^\alpha v^\pm}_{3-k-l,\pm}^2\\
\lesssim &~\delta E_{4+l}(t)+P\left(\sum_{j=0}^{l} \EW_{4+j}(0)\right)+P(E_4(t))\int_0^t P\left(\sum_{j=0}^l \EW_{4+j}(\tau)\right) \dtau.
\end{aligned}
\end{equation}
\end{prop}
\begin{rmk}
The above estimates are not uniform in $\sigma$ in \cite[Proposition 3.8]{Zhang2023CMHDVS1} due to the lacking of stability condition. With the help of stability condition \eqref{syrov 3D com}, we can obtain the non-$\sigma$-weighted boundary regularity, namely the quantities in the constant $C$ in \eqref{divcurlE}. This will be proved in Section \ref{sect noncollinear}.
\end{rmk}

\subsubsection*{Reduction of pressure and divergence}
Let us start with $l=0$. The spatial derivative of $q$ is controlled by invoking the momentum equation:
\begin{align}
-\p_3 q=&~(\p_3 \vp)\left(\rho\Dtp v_3 - \bp b_3\right);\\
-\TP_i q=&-(\p_3\vp)^{-1}\TP_i\vp\,\p_3 q+\rho\Dtp v_i - \bp b_i,~~i=1,2.
\end{align} Let $\TT$ be $\p_t$ or $\TP$ or $\omega(x_3)\p_3$. Then we have
\begin{align}
\|\p_t^k \p_3 q\|_{3-k}\lesssim&~ \|\p_t^k(\rho\TT v_3)\|_{3-k}+\|\p_t^k(b\TT b_3)\|_{3-k}\\
\|\p_t^k \TP_i q\|_{3-k}\lesssim&~ \|\p_t^k (\TP_i\vp\p_3 q)\|_{3-k}+\|\p_t^k(\rho\TT v_i)\|_{3-k}+\|\p_t^k(b\TT b_i)\|_{3-k},
\end{align}in which the leading order terms are $\|\p_t^k\TT(v,b)\|_{3-k}$ and $|\p_t^k \psi|_{4-k}$. This shows that we can convert the control of spatial derivative of $q$ to \textit{tangential estimates} of $v$ and $b$. 

\subsubsection*{Control of the entropy}
The control of entropy is easy thanks to $\Dtpm S^\pm=0$. We now record the result as below.
\begin{prop}\label{prop S} For $0\leq l\leq 4,0\leq k\leq 4-l$, multi-index $\gamma$ satisfying $\len{\gamma}=4+l$ and any $\delta\in(0,1)$, we have
\begin{align}
\left\|\eps^{2l}\p_t^k\TT^{\gamma} S^\pm\right\|_{4-k-l,\pm}^2\lesssim  \delta \EW_{4+l}(t)+P\left(\sum_{j=0}^l \EW_{4+j}(0)\right)+\EW_4(t)\int_0^t P\left(\sum_{j=0}^l \EW_{4+j}(\tau)\right) \dtau, 
\end{align}
\end{prop}

\section{Tangential estimates and the zero-surface-tension limit}\label{sect 0STlimit}
In this section, we aim to prove the following uniform-in-$(\es)$ tangential estimates
\begin{prop}[Tangential estimates]\label{prop ETT}
For fixed $l\in\{0,1,2,3,4\}$ and any $\delta\in(0,1)$, the following uniform-in-$(\es)$ energy inequalities hold:
\begin{align}
&\sum_\pm\sum_{\lee\alpha\ree=2l}\sum_{\substack{0\leq k\leq 4-l \\ k+\alpha_0<4+l}}\left\|\left(\eps^{2l}\TP^{4-k-l}\TT^{\alpha}\p_t^{k}(v^\pm, b^\pm,p^\pm)\right)\right\|^2_{0,\pm}+\sum_{k=0}^{3+l}\left|\sqrt{\sigma}\eps^{2l}\p_t^{k}\psi\right|^2_{5-k-l}\dtau\no\\
\lesssim&~\delta \EW_{4+l}(t) + P\left(\sum_{j=0}^l\EW_{4+j}(0)\right)+P\left(\sum_{j=0}^l\EW_{4+j}(t)\right)\int_0^tP\left(\sum_{j=0}^l\EW_{4+j}(\tau)\right)\dtau
\end{align}and
\begin{align}
&\sum_\pm\sum_{k=0}^{4-l}\left\|\left(\eps^{2l}\p_t^{4+l}(v^\pm, b^\pm,(\ffp)^{\frac12}p^\pm)\right)\right\|^2_{4-k-l,\pm}+\left|\sqrt{\sigma}\eps^{2l}\p_t^{4+l}\psi\right|^2_{1}\no\\
\lesssim&~\delta \EW_{4+l}(t)+ P\left(\sum_{j=0}^l\EW_{4+j}(0)\right)+ P\left(\sum_{j=0}^l\EW_{4+j}(t)\right)\int_0^tP\left(\sum_{j=0}^l\EW_{4+j}(\tau)\right)\dtau.
\end{align} Here the first inequality represents the case when there are at least one spatial tangential derivatives and the second inequality represents the case of full time derivatives. 
\end{prop}

An analogue of Proposition \ref{prop ETT} has been proven in \cite[Proposition 3.3]{Zhang2023CMHDVS1} without assuming the stability condition \eqref{syrov 3D com}, and thus the energy bounds in \cite[Proposition 3.3]{Zhang2023CMHDVS1} rely on $\sigma^{-1}$. Now, we explain how to use the stability condition \eqref{syrov 3D com} to get rid of the dependence on $1/\sigma$ in the estimates recorded in Section \ref{sect lwp energy}. Let us recall what quantities in tangential estimates of $E(t)$ depend on $1/\sigma$ in \cite{Zhang2023CMHDVS1} without assuming the stability condition \eqref{syrov 3D com}. Following the analysis in \cite[Section 2.2, 3.3-3.5]{Zhang2023CMHDVS1}, we can obtain the energy inequality for a given multi-index $\alpha$ satisfying $\len{\alpha}=2l$ and for $k,l\in\N$ satisfying $0\leq k\leq 4-l,~0\leq l\leq 4$.
\begin{align}
\sum_{\pm}\left\|\left(\eps^{2l}\TP^{4-k-l}\TT^{\alpha}\p_t^{k}(v^\pm, b^\pm,(\ffp)^{\frac12}p^\pm)\right)\right\|^2_{4-k-l,\pm}\leq&~ \delta \EW(t) + P(\EW(0)) + P(\EW(t))\int_0^tP(\EW(\tau))\dtau \notag\\
&+\int_0^t\ST+\VS+\RT+\sum_{\pm}(ZB^\pm+Z^\pm)\dtau.
\end{align} Here these major terms are defined by
\begin{align}
\ST:=&~\eps^{4l}\is\TT^\gamma(\sigma\h)\p_t\TT^\gamma\psi\dx',~~\RT:=-\eps^{4l}\is \jump{\p_3q}\TT^\gamma\psi\TT^\gamma\p_t\psi\dx',\\
\VS:=&~\eps^{4l}\is \TT^\gamma q^- (\jump{\vb}\cdot\cnab)\TT^\gamma\psi\dx',\\
ZB^\pm:=&~\mp\eps^{4l}\is\TT^\gamma q^\pm[\TT^\gamma,N_i,v_i^\pm]\dx',\quad Z^\pm:=-\eps^{4l}\iopm \TT^\gamma q^\pm [\TT^\gamma,\NN_i,\p_3 v_i^\pm]\dvt.
\end{align}  where $\TT^\gamma=\p_t^{k}\TP^{4+l-k}$, $0\leq k\leq 4+l,~0\leq l\leq 4$ and $\dvt:=\p_3\vp\dx$. Do note that when $\alpha_4>0$, all boundary terms vanish and $Z^\pm$ can be directly controlled, so we only need to analyze the case when $\alpha$ has the same form as $\gamma$. To get rid of the dependence on $1/\sigma$, the above quantities should be controlled in the following way:
\begin{itemize}
\item The control of quantity ST does not depend on $1/\sigma$;
\item Quantity RT can be controlled if we have the estimates of $|\eps^{2l}\p_t^k\psi|_{4.5+l-k}$  (under time integral) for $0\leq k\leq 4-l,~0\leq l\leq 4$;
\item Quantity $ZB^\pm+Z^\pm$ can be controlled if we have the estimates of $|\eps^{2l}\p_t^k\psi|_{4+l-k}$ (NOT under time integral) for $0\leq k\leq 4-l,~0\leq l\leq 4$;
\item Quantity VS must be completely eliminated.
\end{itemize}Besides, the control of the commutators $\|\eps^{2l}[\nabp\cdot,\p_t^k\TT^\alpha](v^\pm,b^\pm)\|_{3-k-l,\pm}$ and $\|\eps^{2l}[\nabp\times,\p_t^k\TT^\alpha](v^\pm,b^\pm)\|_{3-k-l,\pm}$ in the proof of Proposition \ref{prop divcurl} (Note that this commutator appears without time integral!) also needs the estimates of $|\eps^{2l}\p_t^k\psi|_{4+l-k}$.

\subsection{Enhanced regularity of the interface: non-collinearity of magnetic fields}\label{sect noncollinear}
It is easy to see that merely invoking the kinematic boundary condition and using trace lemma does not solve any issues mentioned above. In order to seek for $\sigma$-independent estimates for $\psi$ and its time derivatives, we require the stability condition \eqref{syrov 3D com} when the space dimension is 3, that is, for some $\delta_0\in(0,\frac18)$,
\begin{align}\label{syrovc3}
0<\delta_0\leq a^\pm\left|\bc^\mp\times\jump{\vb}\right|\leq (1-\delta_0)\left|\bc^+\times\bc^-\right|~~\text{ on }[0,T]\times\Sigma,
\end{align}where we view $\bc^\pm=(b_1^\pm,b_2^\pm,0)^\top,\jump{\vb}=(\jump{v_1},\jump{v_2},0)^\top$ as vectors lying on the plane $\T^2\times\{x_3=0\}\subset \R^3$ to define the exterior product. The quantity $a^\pm$ is defined by
\begin{align*}
a^\pm:=\sqrt{\rho^\pm\left(1+\left(\frac{c_A^\pm}{c_s^\pm}\right)^2\right)}
\end{align*} and $c_A^\pm:=|b^\pm|/\sqrt{\rho^\pm}$ represents the Alfv\'en speed, $c_s^\pm:=\sqrt{\p p^\pm/\p\rho^\pm}$ represents the sound speed. This condition implies the following two important features:
\begin{enumerate}
\item Magnetics fields are not collinear on $\Sigma$, which allows us to gain 1/2-order regularity of the free interface.
\item Quantitative relation between $b^\pm$ and $\jump{\vb}$ on $\Sigma$ allows us to completely eliminate the problematic term VS.
\end{enumerate}

We aim to prove uniform-in-$(\eps,\sigma)$ estimates for the energy functional $\EW(t)$ for the compressible current-vortex sheet system \eqref{CMHDVS0}. Recall $\EW(t)$ is defined by
\begin{align}\label{energy EW}
\EW(t):=\sum_{l=0}^4\EW_{4+l}(t),~~\EW_{4+l}(t):=E_{4+l}(t)+\sum_{k=0}^{4+l}\left|\eps^{2l}\p_t^k \psi\right|_{4.5+l-k}^2,
\end{align}where the term added to $E_{4+l}(t)$ exactly gives the enhanced regularity of the free interface contributed by the non-collinearity stability condition \eqref{syrovc3}. 

Recall that the magnetic fields satisfy the constraint $b^\pm\cdot N=0$ on $\Sigma$, that is, $b_3^\pm=b_1^\pm\TP_1\psi+b_2^\pm \TP_2\psi$. So, we can solve $\TP\psi$ in terms of $b^\pm$ without any derivatives thanks to $\bc^+\times\bc^-\neq \bd{0}$. However, due to the anisotropy of the function spaces, we have to take derivatives on the constraint before we use trace lemma. We have the following estimates for the interface function $\psi$.
\begin{lem}\label{lem psi}
For $s\geq 3$, one has 
\begin{align}\label{psi plus}
|\TP\psi|_{s-\frac12}^2\leq  P(\|b^\pm\|_{s-1,*,\pm},\|b^\pm\|_{3,\pm},|\TP\psi|_{W^{1,\infty}})\left(\|\p_3\jp^{s-2}b^\pm\|_{0,\pm}\|\jp^s b^\pm\|_{0,\pm}+|\TP\psi|_{s-1}^2+\|\jp^s b^\pm\|_{0,\pm}^2\right).
\end{align} Here $\jp:=\sqrt{1-\TL}$, that is, $\widehat{\jp f}(\xi)=\sqrt{1+|\xi|^2}\hat{f}(\xi)$ on $\T^2$.
\end{lem} 
\begin{proof} Taking $\TP^{s-\frac12}$ in the constraint $b^\pm\cdot N=0$ for $s\geq 3$, we get
\begin{align*}
\begin{cases}
b_1^+\TP_1\jp^{s-\frac12}\psi + b_2^+\TP_2\jp^{s-\frac12}\psi=f_b^+\\
b_1^-\TP_1\jp^{s-\frac12}\psi +  b_2^-\TP_2\jp^{s-\frac12}\psi=f_b^-
\end{cases}\xRightarrow {~b^+\nparallel b^-}
\begin{cases}
\TP_1\jp^{s-\frac12}\psi=\dfrac{-b_2^+f_b^- + b_2^-f_b^+}{b_1^+b_2^--b_1^-b_2^+}\\
~\\
\TP_2\jp^{s-\frac12}\psi=\dfrac{b_1^+f_b^- - b_1^-f_b^+}{b_1^+b_2^--b_1^-b_2^+}
\end{cases}
\end{align*}with $f_b^\pm:=\jp^{\frac12}(\jp^{s-1}b^\pm\cdot N)+\jp^{\frac12}([\jp^{s-1},b^\pm,N])+[\jp^{\frac12},b^\pm\cdot]\jp^{s-1}N$. The $L^2(\Sigma)$ norms of the last two terms in $f_b^\pm$ can be directly controlled via Kato-Ponce type inequality (Lemma \ref{KatoPonce}) and interpolations
\begin{align*}
\bno{[\jp^{s-1},b^\pm,N]}_{\frac12}\lesssim&~ |b^\pm|_{s-\frac32}|\TP\psi|_{W^{1,\infty}}+|b^\pm|_{W^{1,\infty}}|\TP\psi|_{s-\frac32}\lesssim\|\p_3\jp^{s-2}b^\pm\|_{0,\pm}^{\frac12}\|\jp^{s-1}b^\pm\|_{0,\pm}^{\frac12}|\TP\psi|_{W^{1,\infty}}+ \|b^\pm\|_{3,\pm}|\TP\psi|_{s-\frac32}\\
%\lesssim&~\|b^\pm\|_{H_*^s(\Om^\pm)}^{\frac12}\|b^\pm\|_{H_*^{s-1}(\Om^\pm)}^{\frac12}|\TP\psi|_{L^{\infty}}+ \|b^\pm\|_{3,\pm}|\TP\psi|_{s-\frac32}\\
\bno{[\jp^{\frac12},b^\pm\cdot] \jp^{s-1} N}_0\lesssim&~ |\jp^{\frac12} b^\pm|_{L^{\infty}}|\TP\psi|_{s-1}\lesssim \|b^\pm\|_{2.5,\pm}|\TP\psi|_{s-1}
\end{align*}Then the regularity of the free interface is given by
\[
|\TP\psi|_{s-\frac12}\leq P(|\bc^\pm|_{L^{\infty}})|f_b^\pm|_{0}\leq P(\|b^\pm\|_{3,\pm},|\TP\psi|_{W^{1,\infty}})\left( \|b^\pm\|_{H_*^s(\Om^\pm)}^{\frac12}\|b^\pm\|_{H_*^{s-1}(\Om^\pm)}^{\frac12}+|\TP\psi|_{s-1}+\left|\jp^{\frac12}\left(\TP^{s-1}b^\pm\cdot N\right)\right|_0\right).
\]To control the boundary norm of $\jp^{s-\frac12}b^\pm\cdot N$, we again convert it to an interior integral and use the divergence constraint $0=\nabp\cdot b^\pm=\cnab\cdot\bc^\pm+\pp_3b^\pm\cdot\NN$ in $\Om^\pm$. Note that we may not directly use the trace lemma due to the anisotropy.
\begin{align*}
\left|\jp^{\frac12}\left(\jp^{s-1}b^\pm\cdot N\right)\right|_0^2=&~\mp 2\iopm\jp^{\frac12}\p_3\left(\jp^{s-1}b^\pm\cdot N\right)\jp^{\frac12}\left(\jp^{s-1}b^\pm\cdot N\right)\dx\\
\overset{\jp^{\frac12}}{==}&~\mp 2\iopm\p_3\left(\jp^{s-1}b^\pm\cdot \NN\right)\,\jp\left(\jp^{s-\frac12}b^\pm\cdot \NN\right)\dx\\
=&~\mp 2\iopm \jp^{s-1}(\p_3 b^\pm\cdot \NN)\,\jp\left(\jp^{s-\frac12}b^\pm\cdot \NN\right)\dx\\
&~\mp 2\iopm \left(\jp^{s-1}b^\pm\cdot \p_3\NN-\left[\jp^{s-1},\NN\cdot\right]\p_3 b^\pm\right)\,\jp\left(\jp^{s-\frac12}b^\pm\cdot \NN\right)\dx\\
=&~\pm 2\iopm\jp^{s-1}\left(\p_3 \vp(\cnab\cdot\bc^\pm)\right)\,\jp\left(\jp^{s-\frac12}b^\pm\cdot \NN\right)\dx \\
&~\mp 2\iopm \left(\jp^{s-1}b^\pm\cdot \p_3\NN-\left[\jp^{s-1},\NN\cdot\right]\p_3 b^\pm\right)\,\jp\left(\jp^{s-\frac12}b^\pm\cdot \NN\right)\dx\\
\lesssim& ~P(|\TP\psi|_{W^{1,\infty}})\left(\|\jp^{s} b^\pm\|_{0,\pm}+\|\jp^{s-1} b^\pm\|_{0,\pm}+\|\jp^{s-2}\p_3b^\pm\|_{0,\pm}\right)\|\jp^s b^\pm\|_{0,\pm}.
\end{align*}
\end{proof}

Lemma \ref{lem psi} shows that the $H^{s+\frac12}(\Sigma)$ norm of the free interface $\psi$ can be converted to lower-order terms and $H_*^s(\Om^\pm)$ norms of $b^\pm$. Thus, the ``non-collinearity" of $b^+$ and $b^-$ on $\Sigma$ brings a gain of $1/2$-order regularity for the interface. Given $l\in\{0,1,2,3,4\}$, the definition of $\EW_{4+l}(t)$ suggests that $\eps^{2l}\jp^{2l}b^\pm \in H^{4-l}(\Om^\pm) $. Thus, letting $s=4+l$ in Lemma \ref{lem psi}, we can get 
\begin{align*}
\bno{\eps^{2l}\TP\psi}_{3.5+l}^2\leq P(\|b^\pm\|_{3,\pm},|\TP\psi|_{W^{1,\infty}})\left(\|\eps^{2l}b^\pm\|_{H_*^{4+l}(\Om^\pm)}\|\eps^{2l}b^\pm\|_{H_*^{3+l}(\Om^\pm)}+|\eps^{2l}\TP\psi|_{3+l}^2+\|\eps^{2l}b^\pm\|_{H_*^{4+l}(\Om^\pm)}^2\right).
\end{align*} 
Similarly, we can show the enhanced regularity for time derivatives of $\psi$ after replacing $\jp^{s-\frac12}\psi$ by $\jp^{s-k-\frac12}\p_t^k\psi$. We conclude the following proposition.
\begin{lem}\label{lem psit}
For $3\leq s\in\N^*$ and $1\leq k\leq s-1,~k\in\N^*$, one has 
\begin{align}
|\TP\p_t^k\psi|_{s-k-\frac12}^2\leq  P\left(\|b^\pm\|_{s-1,*,\pm},|\cnab\psi|_{W^{1,\infty}}\right)&\bigg(\sum_{j=1}^{k-1}|\p_t^j\TP\psi|_{s-j-\frac32}^2+|\TP\p_t^k\psi|_{s-k-1}^2\no\\
&+\|\jp^{s-k}\p_t^kb^\pm\|_{0,\pm}\|\p_3\jp^{s-k-2}\p_t^kb^\pm\|_{0,\pm}+\|\jp^{s-k}\p_t^kb^\pm\|_{0,\pm}^2\bigg),\label{psit plus}
\end{align} where the number of time derivatives in $\|b^\pm\|_{s,*,\pm}$ appearing on the right side does not exceed $k$. The term $\p_3\jp^{s-k-2}\p_t^kb^\pm$ does not appear when $s-k=1$.
\end{lem}
\begin{rmk}
In the control of $\EW_{4+l}(t)$ for $0\leq l\leq 4$, we have $\eps^{2l}\p_t^k\TT^\alpha b\in L^2(\Om^\pm)$ for $0\leq k\leq l,~\len{\alpha}=2l,~\alpha_3=0$. It should also be noted that there is no loss of Mach number in the estimates of $\p_t^k\psi$ because the number of time derivatives appearing on the right side of \eqref{psit plus} does not exceed that on the left side. Thus, the above estimates directly help us to control the quantities ST, RT and $Z^\pm+ZB^\pm$ by $P(\EW(0))+P(\EW(t))\int_0^tP(\EW(\tau))\dtau$  uniformly in $\sigma$ under the stability condition \eqref{syrovc3}. 
\end{rmk}

Apart from the term VS, we still need to control the commutators $\ino{\eps^{2l}[\p_t^k\TT^\alpha, \nabp\cdot]f}_{3-k-l}^2$ and $\ino{\eps^{2l}[\p_t^k\TT^\alpha, \nabp\times]f}_{3-k-l}^2$ for $f=v,b$ and $0\leq l\leq 3,~1\leq k+l\leq 3,~\len{\alpha}=2l,~\alpha_3=0$, in which the highest order terms have the form $\eps^{2l}(\p_3\vp)^{-1}(\TP\p_t^k\TT^\alpha\vp)(\p_3 f)$ whose estimate requires the bound for $|\eps^{2l}\TP\p_t^k\TT^\alpha\psi|_{3-k-l}^2$. We can also assume $\alpha_4=0$ because $\vp$ has $C^{\infty}$-regularity in $x_3$-direction. Such terms appear without time integral, so we control them by $P(\EW(0))+P(\EW(t))\int_0^tP(\EW(\tau))\dtau$. 

Letting $s=3.5+l$ in Lemma \ref{lem psi} for $0\leq l\leq 4$, using interpolation and Young's inequality, we get
\begin{align*}
|\TP\psi|_{3+l}^2\leq&~  P(\|b^\pm\|_{2+l,*,\pm},\|b^\pm\|_{3,\pm},|\TP\psi|_{W^{1,\infty}})\left(\|\p_3\jp^{1.5+l}b^\pm\|_{0,\pm}\|\jp^{3.5+l} b^\pm\|_{0,\pm}+|\TP\psi|_{2.5+l}^2+\|\jp^{3.5+l} b^\pm\|_{0,\pm}^2\right)\\
\lesssim &~P(\|b^\pm\|_{2+l,*,\pm},\|b^\pm\|_{3,\pm},|\TP\psi|_{W^{1,\infty}})\\
&\left(\|\p_3\jp^{1+l}b^\pm\|_{0,\pm}^{\frac12}\|\p_3\jp^{2+l}b^\pm\|_{0,\pm}^{\frac12}\|\jp^{3+l} b^\pm\|_{0,\pm}^{\frac12}\|\jp^{4+l} b^\pm\|_{0,\pm}^{\frac12}+|\TP\psi|_{2.5+l}^2+\|\jp^{3+l} b^\pm\|_{0,\pm}\|\jp^{4+l} b^\pm\|_{0,\pm}\right)\\
\lesssim&~\delta\left(\|\p_3\jp^{2+l}b^\pm\|_{0,\pm}^2+\|\jp^{4+l} b^\pm\|_{0,\pm}^2\right)\\
&+P(\|b^\pm\|_{2+l,*,\pm},\|b^\pm\|_{3,\pm},|\TP\psi|_{W^{1,\infty}},\delta^{-1})\left(\|\jp^{3+l} b^\pm\|_{0,\pm}^2+\|\jp^{3+l} b^\pm\|_{0,\pm}+|\TP\psi|_{2.5+l}^2\right)\\
\lesssim&~\delta \|b^\pm\|_{4+l,*,\pm}^2 + P(\|b_0^\pm\|_{3+l,*,\pm},|\TP\psi_0|_{2.5+l})+ \int_0^tP\left(\|\p_tb^\pm(\tau)\|_{3+l,*,\pm},|\p_t\TP\psi(\tau)|_{2.5+l}\right)\dtau.
\end{align*} Similarly, by replacing $\TP\psi$ with $\TP\p_t^k\TT^\alpha \psi$, we can get the following inequality, where $0\leq l\leq 3,~1\leq k+l\leq 3,~\len{\alpha}=\alpha_0+\alpha_1+\alpha_2=2l ~(\alpha_3=\alpha_4=0)$:
\begin{align*}
\bno{\TP\p_t^k\TT^\alpha\psi}_{3-k-l}^2=\bno{\TP^{1+2l-\alpha_0}\p_t^{k+\alpha_0}\psi}_{3-k-l}^2\lesssim&~\delta \|b^\pm\|_{4+l,*,\pm}^2+\sum_{j=0}^k P\left(\|b^\pm\|_{3+l,*,\pm},|\TP\p_t^{j+\alpha_0}\psi|_{2.5+l-j-\alpha_0}\right)\big|_{t=0} \\
&+ \int_0^tP\left(\|\p_tb^\pm(\tau)\|_{3+l,*,\pm},|\TP\p_t^{j+\alpha_0+1}\psi(\tau)|_{2.5+l-j-\alpha_0}\right)\dtau.
\end{align*}

Since there is no loss of weights of Mach number when applying Lemma \ref{lem psi}-Lemma \ref{lem psit} to the estimates of compressible current-vortex sheet system \eqref{CMHDVS0}, we can conclude the enhanced regularity, which is uniform in $(\eps,\sigma)$, of the free interface by the following proposition.
\begin{prop}\label{prop psi}
For $l\in\{0,1,2,3,4\}$ and $k\leq 4+l,~k\in\N$, we have
\begin{enumerate}
\item When $0\leq k\leq 3+l$:
\begin{align}\label{psi+0.5}
\bno{\eps^{2l}\TP\p_t^k\psi(t)}_{3.5+l-k}^2\leq P\left(\sum_{j=0}^{(l-1)_+}\EW_{4+j}(0)\right) + P\left(\sum_{j=0}^{(l-1)_+}\EW_{4+j}(t)\right)\left(\int_0^t P\left(\sum_{j=0}^{l}\EW_{4+j}(\tau)\right)\dtau + \ino{b^\pm(t)}_{4+l,*,\pm}^2 \right).
\end{align}
\item When $k=4+l$:
\begin{align}
\bno{\eps^{2l}\p_t^{4+l}\psi(t)}_{0.5}^2\leq&~ P(|\cnab\psi|_{W^{1,\infty}})\left(\ino{\eps^{2l}\jp\p_t^{3+l}v^\pm}_{0,\pm}^2+\ino{\eps^{2l+2}\jp\p_t^{3+l}\Dtpm p^\pm}_{0,\pm}^2\right)+|\vb|_{L^{\infty}}\bno{\eps^{2l}\TP\p_t^{3+l}\psi(t)}_{0.5}\no\\
&+P\left(\sum_{j=0}^{(l-1)_+}\EW_{4+j}(0)\right) + \int_0^t P\left(\sum_{j=0}^{l}\EW_{4+j}(\tau)\right)\dtau. \label{psit+0.5}
\end{align}
\end{enumerate}
\end{prop}
\begin{proof}
When $k\leq 3+l$, the inequality \eqref{psi+0.5} is a direct consequence of Lemma \ref{lem psi} and Lemma \ref{lem psit}. Indeed, we just need to control the $\psi$ term on the right side of \eqref{psi+0.5}-\eqref{psit+0.5} to be $P(\EW(0))+P(\EW(t))\int_0^tP(\EW(\tau))\dtau$. This can be done by applying again \ref{lem psi} and Lemma \ref{lem psit} to the $\psi$ term appearing on the right side of \eqref{psi plus} and \eqref{psit plus} by replacing $s=4+l-k$ with $s=3.5+l-k$. When $k=4+l$, we just differentiate the kinematic boundary condition $\p_t\psi=v^\pm\cdot N$ to get
\[
\p_t^{4+l}\psi=\p_t^{3+l}v^\pm\cdot N+\vb^\pm\cdot\cnab\p_t^{3+l}\psi + [\p_t^{3+l},v^\pm\cdot,N],
\]where the second term contributes to the second term on the right side of  \eqref{psit+0.5} and the last term contributes to the second line of \eqref{psit+0.5}. For $\p_t^{3+l}v^\pm\cdot N$, we again use Gauss-Green formula to covert its boundary norm to an interior integral and use $\pp_3 v^\pm\cdot \NN=\nabp\cdot v^\pm-\cnab\cdot\vb^\pm=-\eps^2\Dtpm p^\pm-\cnab\cdot \vb^\pm$ to replace the normal derivative by tangential derivative. The proof is parallel to the control of $|\jp^{\frac12}\left(\jp^{s-1}b^\pm\cdot N\right)|_0$ as in the proof of Lemma \ref{lem psi} and we no longer repeat the details.
\end{proof}

\subsection{Elimination of VS term: Friedrichs secondary symmetrization}\label{sect symmetrize}
The quantity ST  produces the $\sqrt{\sigma}$-weighted boundary energy in $\EW(t)$ as in \cite[Prop. 3.3]{Zhang2023CMHDVS1}. With the new energy functional \eqref{energy EW}, remainder terms in ST, quantities RT and $ZB^\pm+Z^\pm$ mentioned at the beginning of Section \ref{sect 0STlimit} are all controlled by $P(\EW(0))+P(\EW(t))\int_0^tP(\EW(\tau))\dtau$ thanks to Proposition \ref{prop psi}. The terms that appear without time integral on the right side of \eqref{psi+0.5} and \eqref{psit+0.5} can also be controlled by $P(\EW(0))+P(\EW(t))\int_0^tP(\EW(\tau))\dtau$ via div-curl analysis or tangential estimates. In other words, we have reached the following energy inequality for $\EW(t)$
\begin{align}
\EW(t)\lesssim \delta \EW(t) + P(\EW(0)) + P(\EW(t))\int_0^tP(\EW(\tau))\dtau + \int_0^t\VS\dtau,
\end{align}
so it remains to control or eliminate the term VS arising from the estimates of $\EW(t)$, such that we can close the energy estimates for $\EW(t)$ and also get rid of the dependence on $1/\sigma$. 

\subsubsection{Motivation for Friedrichs secondary symmetrization}
The regularity for the free interface needed in the control of$\VS$ is higher than the one we obtain in Proposition \ref{prop psi}. So, we alternatively try to completely eliminate the term VS by utilizing the jump of tangential magnetic field. Recall that the term VS is generated due to the discontinuity in tangential velocity 
\[ 
\VS=\is \TT^\gamma q^-\,(\jump{\vb}\cdot\cnab)\TT^\gamma \psi\dx',\q\q \TT^\gamma=\p_t^{\gamma_0}\TP_1^{\gamma_1}\TP_2^{\gamma_2}~(\gamma_3=\gamma_4=0),
\] in which we may try to insert a term $\jump{\bmu \bc}$ into $\jump{\vb}$ such that $\jump{\vb-\bmu \bc}=\mathbf{0}$ on $\Sigma$ for some function $\bmu^\pm$. Such functions $\bmu^\pm$ do exist and are unique thanks to the non-collinearity $\bc^+\nparallel \bc^-$ on $\Sigma$:
\begin{equation}\label{choice mu 3D}
\begin{cases}
\jump{v_1}=\bar{\mu}^+b_1^+-\bar{\mu}^-b_1^-\\
\jump{v_2}=\bar{\mu}^+b_2^+-\bar{\mu}^-b_2^-
\end{cases}\xRightarrow{~b^+\nparallel b^-}~ \bar{\mu}^\pm=\dfrac{b_1^\mp\jump{v_2}-b_2^\mp\jump{v_1}}{b_1^+b_2^--b_1^-b_2^+}=\dfrac{(\bc^\mp\times\jump{\vb})_3}{(\bc^+\times \bc^-)_3}.
\end{equation}

Next, a natural question is how to produce such $\jump{\bar{\mu}\bc}$-terms in the tangential estimates. Recall that the discontinuity term $(\jump{\vb}\cdot\cnab)\TT^\gamma \psi$ is produced by taking substraction bewteen the equations of $\TT^\gamma v^\pm \cdot N$ which originates from $\frac12\iopm\rho^\pm|\TT^\gamma v^\pm |^2\dvt$ (see \cite[(3.33)]{Zhang2023CMHDVS1} for example). This suggests us to replace the variable $v^\pm$ by $v^\pm-\bar{\mu}^\pm b^\pm$ in the momentum equation in order to create the elimination $\jump{\vb-\bmu \bc}=\mathbf{0}$. However, such replacement in the momentum equation will make the compressible ideal MHD system \eqref{CMHDVS0} no longer symmetric, which will further lead to the failure of $L^2$ energy conservation. Hence, we must re-symmetrize the hyperbolic system after replacing $v$ by $v-\mu b$.

The technique we use is the so-called Friedrichs secondary symmetrization \cite{Friedrichs}. For compressible ideal MHD system, the symmetrizer was explicitly calculated in Trakhinin \cite{Trakhinin2005CMHDVS}. Let $\mu(t,x)=\bar{\mu}(t,x')\eta(x_3)$ where $\eta(x_3) \in C_c^{\infty}(\R)$ is a smooth, non-negative, even function satisfying $\eta(0)=1$ and $\eta(x_3)=0$ when $|x_3|>\delta_1$ for some sufficiently small constant $\delta_1>0$. Inserting this $\eta$ is to localise the function $\mu^\pm$ near the interface $\Sigma$. The new system takes the form
\begin{equation}\label{CMHDVS0mu}
\begin{cases}
\rho\Dtp v -\bp b+\nabp (p+\frac12|b|^2) -\mu\rho \left( \Dtp b - \bp v+b(\nabp\cdot v)\right)=0,\\
\ffp\Dtp p+\nabp\cdot v +\mu\ffp\left( \rho\Dtp v\cdot b+ b\cdot\nabp p\right)=0\\
\Dtp b-\bp v+b(\nabp\cdot v)-\mu\left(\rho \Dtp v -\bp b +\nabp (p+\frac12|b|^2)\right)=0,
\end{cases}
\end{equation}where these equations are obtained by doing the following linear transform on \eqref{CMHDVS0}
\begin{align*}
\text{new momentum equation}=&~\text{momentum equation} -\mu\rho (\text{evolution equation of }b),\\
\text{new continuity equation}=&~\text{continuity equation} + \mu\ffp(\text{momentum equation})\cdot b,\\
\text{new evolution equation of }b=&~\text{momentum equation} -\mu(\text{momentum equation}).
\end{align*}Note that in the second equation we use the fact that $\left(\nabp(1/2|b|^2)-\bp b\right)\cdot b=(b\times(\nabp\times b))\cdot b=0$.

\subsubsection{Reformulations in Alinhac good unknowns}\label{sect AGU}
Given a tangential derivative $\TT^\gamma:=(\omega(x_3)\p_3)^{\gamma_4}\p_t^{\gamma_0}\p_1^{\gamma_1}\p_2^{\gamma_2}$ with $\lee\gamma\ree=\gamma_0+\gamma_1+\gamma_2+\gamma_4$., we need to re-consider the tangential estimates in order to avoid the appearance of the term VS when $\gamma_4=0$ (When $\gamma_4\neq 0$, no boundary term appears thanks to $\omega(x_3)=0$, so we can just mimic the proof in \cite[Section 3.3-3.4]{Zhang2023CMHDVS1}).   We define the Alinhac good unknown of a given function $f$ with respect to $\TT^\gamma$ by $\FF^\gamma:=\TT^\gamma f-\TT^\gamma\varphi\p_3^\vp f$. The good unknown $\FF$ satisfies
\begin{equation}\label{AGU good}
\TT^{\gamma}\nab_i^\vp f=\nabp_i\FF^\gamma+\cc_i^\gamma(f),~~\TT^\gamma \Dtp f=\Dtp\FF^\gamma+\dd^\gamma(f),
\end{equation}where the commutators $\cc^\gamma_i(f)$ and $\dd^\gamma(f)$ are defined by
\begin{align}\label{AGU comm Ci}
\mathfrak{C}_i^\gamma(f) =&~(\pp_3\pp_i f)\TT^\gamma\vp
+\left[ \TT^\gamma, \frac{\NN_i}{\p_3 \vp}, \p_3 f\right]+\p_3 f \left[ \TT^\gamma, \NN_i, \frac{1}{\p_3 \vp}\right] +\NN_i \p_3 f\left[\TT^{\gamma-\gamma'}, \frac{1}{(\p_3 \vp)^2}\right] \TT^{\gamma'} \p_3  \vp \nonumber\\
&+\frac{\NN_i}{\p_3  \vp} [\TT^\gamma, \p_3] f - \frac{\NN_i}{(\p_3 \vp)^2}\p_3f [\TT^\gamma, \p_3] \vp,\quad i=1,2,3,
\end{align} and 
\begin{align} \label{AGU comm D}
\mathfrak{D}^\gamma(f) =&~ (\Dtp \pp_3 f)\TT^\gamma\vp+[\TT^\gamma, \vb]\cdot \TP f + \left[\TT^\gamma, \frac{1}{\p_3\vp}(v\cdot \NN-\p_t\varphi), \p_3 f\right]+\left[\TT^\gamma, v\cdot \NN-\p_t\varphi, \frac{1}{\p_3\vp}\right]\p_3 f\nonumber\\ 
&+\frac{1}{\p_3\vp} [\TT^\gamma, v]\cdot \NN \p_3 f-(v\cdot \NN-\p_t\varphi)\p_3 f\left[ \TT^{\gamma-\gamma'}, \frac{1}{(\p_3 \vp)^2}\right]\TT^{\gamma'} \p_3 \vp\nonumber\\
&+\frac{1}{\p_3\vp}(v\cdot \NN-\p_t\varphi) [\TT^\gamma, \p_3] f+ (v\cdot \NN-\p_t\varphi) \frac{\p_3 f}{(\p_3\vp)^2}[\TT^\gamma, \p_3] \vp
\end{align}
with $\len{\gamma'}=1$. Here $\NN:=(-\TP_1\vp,-\TP_2\vp,1)^\top$ is the extension of normal vector $N$ in $\Om^\pm$. The third term on the right side of \eqref{AGU comm Ci} is zero when $i=3$ because $\NN_3=1$ is a constant. One can follow \cite[Section 3.3-3.5]{Zhang2023CMHDVS1} to verify that the $L^2(\Om^\pm)$ norms of commutators $\cc_i^\gamma(f)$ and $\dd^\gamma(f)$ can be controlled by $P(\EW(t))$ for $f=v_i^\pm$ or $q^\pm$ by directly counting the number of derivatives.

Under the above setting and dropping the superscript $\gamma$ for convenience, the $\TT^\gamma$-differentiated current-vortex sheet system, after doing Friedrichs secondary symmetrization, is reformulated in the corresponding Alinhac good unknowns $(\VV^{\pm},\BB^{\pm},\PP^{\pm},\QQ^{\pm},\BS^{\pm})$ as follows
\begin{align}
\label{goodvmu}\rho^\pm\Dtpm \VV^{\pm}-\bpm \BB^{\pm}+\nabp \QQ^{\pm}-\mu^\pm \rho^\pm(\Dtpm \BB^{\pm}-\bp\VV^\pm+b^\pm(\nabp\cdot\VV^\pm))=&~\RR^{\pm,\mu}_{v}-\cc(q^\pm)+\mu^\pm\rho^\pm b^\pm \cc_i(v_i^\pm),\\
\label{goodpmu}\ffpm \Dtpm \PP^\pm+\nabp\cdot\VV^\pm+\mu^\pm\ffpm(\rho^\pm\Dtpm \VV^\pm\cdot b^\pm +b^\pm\cdot\nabp\PP^\pm)=&~\RR_{p}^{\pm,\mu}-\cc_i(v_i^\pm)-\mu^\pm \ffpm b^\pm\cdot\cc(q^\pm),\\
\label{goodbmu}\Dtpm\BB^\pm-\bpm\VV^\pm+b^\pm(\nabp\cdot\VV^\pm)-\mu(\rho^\pm\Dtpm \VV^{\pm}-\bpm \BB^{\pm}+\nabp \QQ^{\pm})=&~\RR_b^{\pm,\mu}-b^\pm\cc_i(v_i^\pm)+\mu^\pm\cc(q^\pm)\\
\Dtpm\BS^\pm =&~\dd(S^\pm),
\end{align}where $\RR_v^\mu,\RR_p^\mu,\RR_b^\mu$ terms consist of the following commutators
\begin{align}
\label{goodrvmu} \RR_v^{\pm,\mu}:=&~\RR_{v}^\pm-\mu^\pm\rho^\pm\RR_b^\pm+[\TT^\gamma,\mu^\pm\rho^\pm]\left(\Dtpm b^\pm-\bpm v^\pm+b^\pm(\nabp\cdot v^\pm)\right)\\
\label{goodrpmu} \RR_p^{\pm,\mu}:=&~\RR_p^\pm+\mu^\pm \ffpm b^\pm\cdot \RR_v^{\pm}-[\TT^\gamma,\mu^\pm\ffpm]\left(\rho^\pm\Dtpm v^\pm\cdot b^\pm - \bpm p^\pm\right)\\
\label{goodrbmu} \RR_b^{\pm,\mu}:=&~\RR_b^{\pm}-\mu^\pm\rho^\pm\RR_v^\pm+[\TT^\gamma,\mu^\pm]\left(\Dtpm v^\pm-\bpm b^\pm+\nabp q^\pm\right),
\end{align} with $\RR_v^{\pm},\RR_b^{\pm},\RR_b^\pm$ defined by
\begin{align}
\label{goodrv} \RR_v^{\pm}:=&~[\TT^\gamma,b^\pm]\cdot\nabp b^\pm-[\TT^\gamma,\rho^\pm]\Dtpm v^\pm-\rho^\pm\dd^\gamma(v^\pm),\\
\label{goodrp} \RR_p^{\pm}:=&-[\TT^\gamma,\ffpm]\Dtpm p^\pm -\ffpm \dd^\gamma(p^\pm),\\
\label{goodrb} \RR_b^{\pm}:=&~[\TT^\gamma,b^\pm]\cdot\nabp v^\pm-\dd^\gamma(b^\pm).
\end{align}
The boundary conditions on the  interface $\Sigma$ are
\begin{align}
\label{goodbdcmu}\jump{\QQ}=\sigma\TT^{\gamma}\h(\psi)-\jump{\p_3 q}\TT^\gamma\psi~~&\text{ on }[0,T]\times\Sigma,\\
\label{goodkbcmu}\VV^\pm\cdot N=\p_t\TT^{\gamma}\psi+\vb^\pm\cdot\cnab\TT^\gamma\psi-\WW_v^\pm~&\text{ on }[0,T]\times\Sigma,\\
\label{goodbncmu}b^\pm\cdot N=0\Rightarrow \BB^\pm\cdot N=\bc^\pm\cdot\cnab\TT^\gamma\psi-\WW_b^\pm~&\text{ on }[0,T]\times\Sigma,
\end{align}and the boundary term $\WW^\gapm$ is 
\begin{align}\label{goodwwmu}
\WW_f^\pm:=(\p_3 f^\pm\cdot N)\TT^\gamma\psi+[\TT^\gamma,N_i,f_i^\pm],~~f=v,b.
\end{align} 

Note that the good unknowns $\VV^\pm,\BB^\pm,\QQ^\pm,\PP^\pm,\BS^\pm$ also satisfy the above system with mathematically setting $\mu^\pm=0$ because the secondary-symmetrized system above is obtained by doing a linear transform on the original system \eqref{CMHDVS0}. Also note that we only need to re-consider the $\TT^\gamma$-estimates for $\gamma_4=0$, we can rewrite $\TT^\gamma$ to be $\p_t^k\TP^{4+l-k}$ for $0\leq k\leq 4+l$ and $0\leq l\leq 4$. In this case, the second line in \eqref{AGU comm Ci} and the third line in \eqref{AGU comm D} are both zero because $\p_t^k\TP^{4+l-k}$ commutes with $\p_3$.

\subsubsection{Tangential estimates (I): Analysis in the interior}
Recall that the term VS originates from the tangential estimates. After doing Friedrichs secondary symmtrisation, we shall consider the tangential estimates for the following functional to establish the $\eps^{2l}\TT^\gamma$-estimates $(0\leq l\leq 4,\len{\gamma}=4+l)$:
\begin{align}
\GG^{\pm,\mu}(t):=\frac{\eps^{4l}}{2}\iopm \rho^\pm|\VV^\pm|^2+|\BB^\pm|^2+\ffpm(\PP^\pm)^2-2\mu^\pm\rho^\pm\VV^\pm\cdot\BB^\pm +2\mu^\pm\rho^\pm\ffpm \PP^\pm(b^\pm\cdot\VV^\pm)\dvt
\end{align}instead of the following one used in \cite[Section 3.3-3.5]{Zhang2023CMHDVS1}
\[
\GG^\pm(t):=\frac{\eps^{4l}}{2}\iopm \rho^\pm|\VV^\pm|^2+|\BB^\pm|^2+\ffpm(\PP^\pm)^2\dvt.
\]
Using Reynolds' transport theorem (Theorem \ref{transport thm nonlinear}), we have
\begin{equation}
\begin{aligned}
\ddt\GG^{\pm,\mu}(t)=&~\eps^{4l}\iopm \rho^\pm \Dtpm\VV^\pm\cdot(\VV^\pm-\mu^\pm\BB^\pm+\eps^{4l}\mu^\pm\ffpm\PP^\pm b^\pm)\dvt+\eps^{4l}\iopm\Dtpm\BB^\pm\cdot(\BB^\pm-\mu^\pm\rho^\pm\VV^\pm)\dvt\\
&+\eps^{4l}\iopm \ffpm\Dtpm\PP^\pm(\PP^\pm+\mu^\pm \rho^\pm b^\pm\cdot\VV^\pm)\dvt+\eps^{4l}R_{1}^{\pm,\mu}\\
=:&~\eps^{4l}(G_1^{\pm,\mu}+G_2^{\pm,\mu}+G_3^{\pm,\mu}+R_{1}^{\pm,\mu})
\end{aligned}
\end{equation}where
\begin{align}
R_{1}^{\pm,\mu}:=&~\frac12\iopm (\nabp\cdot v^\pm)\left(|\BB^\pm|^2+ (\sqrt{\ffpm}\PP^\pm)^2\right)\dvt+\iopm \mu^\pm\ffpm\PP^\pm(\Dtpm( \rho^\pm b^\pm)\cdot \VV^\pm)\dvt\no\\
&-\iopm \Dtpm\mu \left(\rho^\pm\VV^\pm\cdot\BB^\pm + \rho^\pm\ffpm\PP^\pm(b^\pm\cdot\VV^\pm)\right)\dvt.
\end{align} Invoking the evolution equations of good unknowns with all $\mu^\pm$-terms dropped, we get
\begin{align}
G_1^{\pm,\mu}=&\iopm \bpm \BB^\pm\cdot\left(\VV^\pm-\mu^\pm\BB^\pm\right)\dvt-\iopm\nabp\QQ^\pm\cdot\left(\VV^\pm-\mu^\pm\BB^\pm\right)\dvt\no\\
&+\iopm\ffpm\mu^\pm(\rho^\pm\Dtpm\VV^\pm\cdot b^\pm) \PP^\pm\dvt-\iopm \cc(q^\pm)\cdot\left(\VV^\pm-\mu^\pm\BB^\pm\right)\dvt+\iopm\RR_v^\pm\cdot\left(\VV^\pm-\mu^\pm\BB^\pm\right)\dvt\no\\
=:&~G_{11}^{\pm,\mu}+G_{12}^{\pm,\mu}+G_{13}^{\pm,\mu}+R_{2}^{\pm,\mu}+R_{3}^{\pm,\mu}.
\end{align}
In $G_{11}^{\pm,\mu}$ and $G_{12}^{\pm,\mu}$, we integrate by parts to get
\begin{align}
G_{11}^{\pm,\mu}=&-\iopm \BB^\pm\cdot\bpm\VV^\pm \dvt-\frac12 \iopm \nabp\cdot(\mu^\pm b^\pm)|\BB^\pm|^2\dvt\no\\
=:&~G_{111}^{\pm,\mu}+R_{4}^{\pm,\mu}
\end{align}and use $\nabp\cdot\BB^\pm=-\cc_i(b_i^\pm)$ to get
\begin{align}
G_{12}^{\pm,\mu}=&\pm\is \QQ^\pm (\VV^\pm-\mu^\pm\BB^\pm)\cdot N\dx'+\iopm\QQ^\pm(\nabp\cdot\VV^\pm) \dvt+\iopm \mu^\pm \QQ^\pm\,\cc_i(b_i^\pm)\dvt\no\\
=:&~G_{0}^{\pm,\mu}+G_{121}^{\pm,\mu}+G_{122}^{\pm,\mu}.
\end{align}
In $G_{13}^{\pm,\mu}$, we notice that 
\[
(-\bpm\BB^\pm+\nabp\BB^\pm_jb_j^\pm)\cdot b^\pm=-b_j^\pm(\pp_j\BB_i^\pm) b_i^\pm+(\pp_i\BB^\pm_j)b_j^\pm b_i^\pm=0,
\]so it becomes the following controllable quantities by using symmetry
\begin{align}
G_{13}^{\pm,\mu}=&-\iopm\ffpm\mu^\pm\rho^\pm (\bpm \PP^\pm)\,\PP^\pm\dvt+\iopm \ffpm\mu^\pm(\RR_v^\pm-\cc(q^\pm))\cdot  b^\pm \PP^\pm\dvt\no\\
=&-\frac12\iopm\nabp\cdot(\ffpm\mu^\pm\rho^\pm b^\pm)(\PP^\pm)^2\dvt +\iopm \ffpm\mu^\pm(\RR_v^\pm-\cc(q^\pm))\cdot b^\pm \PP^\pm\dvt=:R_{5}^{\pm,\mu}+R_{6}^{\pm,\mu}.
\end{align} Note that the terms $G_{111}^{\pm,\mu}$ and $G_{122}^{\pm,\mu}$ already appear in the previous analysis for \eqref{CMHDVS0} in \cite[Section 3.3.1]{Zhang2023CMHDVS1}, so we no longer need to put extra effort on it. 

Next we analyze $G_2^{\pm,\mu}\sim G_3^{\pm,\mu}$. Invoking equations of $\PP^\pm$ and $\BB^\pm$ with all $\mu^\pm$-terms dropped, we get
\begin{align}
G_2^{\pm,\mu}=&~\underbrace{\iopm (\bpm\VV^\pm)\cdot\BB^\pm\dvt}_{=-G_{111}^{\pm,\mu}} - \iopm b^\pm(\nabp\cdot \VV^\pm)\cdot\BB^\pm\dvt -\iopm \BB^\pm\cdot b^\pm\cc_i(v_i)\dvt\no\\
&+\iopm \mu^\pm\rho^\pm \VV^\pm\cdot b^\pm(\nabp\cdot \VV^\pm)\dvt\no + \iopm \mu^\pm\rho^\pm\VV^\pm\cdot b^\pm \cc_i(v_i^\pm)\dvt-\iopm\mu^\pm\rho^\pm\bpm\VV^\pm\cdot\VV^\pm\dvt\no\\
=:&-G_{111}^{\pm,\mu}+G_{21}^{\pm,\mu}+G_{22}^{\pm,\mu}+G_{23}^{\pm,\mu}+R_{7}^{\pm,\mu}+R_{8}^{\pm,\mu},
\end{align}and
\begin{align}
G_3^{\pm,\mu}=&-\iopm (\nabp\cdot\VV^\pm)\PP^\pm\dvt-\iopm \PP^\pm\cc_i(v_i^\pm)\dvt\underbrace{-\iopm \rho^\pm\mu^\pm b^\pm(\nabp\cdot\VV^\pm)\cdot\VV^{\pm} \dvt}_{=-G_{23}^{\pm,\mu}} \no\\
&+\iopm\PP^\pm\RR_p^\pm +\mu^\pm(\RR_p^\pm-\cc_i(v_i^\pm))(\rho^\pm b^\pm\cdot\VV^\pm)\dvt\no\\
=:&~G_{31}^{\pm,\mu}+G_{32}^{\pm,\mu}-G_{23}^{\pm,\mu}+R_{9}^{\pm,\mu}.
\end{align}
Now we can see a lot of cancellation structures among these interior integrals. First, using $$\QQ^\pm= \PP^\pm+b^\pm\cdot\BB^\pm+\RR_q^\pm,\quad\RR_q^\pm=\sum\limits_{1\leq\len{\gamma'}\leq\len{\gamma}-1}\TT^\gamma b_j^\pm\cdot\TT^{\gamma-\gamma'} b^\pm_j,$$ we have
\begin{align}
G_{121}^{\pm,\mu}+G_{21}^{\pm,\mu}+G_{31}^{\pm,\mu} =& \iopm (\nabp\cdot\VV^\pm)\RR_q^\pm\dvt=:R_{10}^{\pm,\mu},\\
G_{122}^{\pm,\mu}+G_{22}^{\pm,\mu}+G_{32}^{\pm,\mu} =& -\iopm \QQ^\pm (\cc_i(v_i^\pm)-\mu\cc_i(b_i^\pm))\dvt + \iopm \RR_q^\pm (\cc_i(v_i^\pm)-\mu\cc_i(b_i^\pm))\dvt \no\\
=: \eps^{-4l}Z^{\pm,\mu} + R_{11}^{\pm,\mu}.
\end{align}
Here we add $\eps^{-4l}$ to the first term just for the consistency of notations. The terms $R_j^{\pm,\mu}~(1\leq j\leq 11)$ can be directly controlled using the same method in \cite[Section 3.3-3.4]{Zhang2023CMHDVS1}, so we omit the details. 
\begin{align}
\eps^{4l}\int_0^t\sum_{j=1}^{11}R_j^{\pm,\mu} \leq \delta P(\EW(t)) + P(\EW(0))+\int_0^t P(\EW(\tau)) \dtau,\quad\forall \delta\in(0,1).
\end{align} Thus, it again remains to analyze the boundary integral $G_{0}^{\pm,\mu}$ and the commutator term $Z^{\pm,\mu}$.

\subsubsection{Tangential estimates (II): Elimination of the term VS}
Invoking the boundary conditions, the boundary integral $G_{0}^{\pm,\mu}$ can be decomposed as follows
\begin{align}
\eps^{4l}(G_{0}^{+,\mu}+G_0^{-,\mu})=&\eps^{4l}\is \QQ^+ (\VV^+-\mu^+\BB^+)\cdot N\dx'-\eps^{4l}\is \QQ^- (\VV^--\mu^-\BB^-)\cdot N\dx'\no\\
=&\ST^\mu+{\ST^\mu}'+\VS^\mu+\RT^\mu+\RT^{\pm,\mu}+ZB^{\pm,\mu}
\end{align} where
\begin{align}
\label{def STmu} \ST^\mu:=&~\eps^{4l}\is\TT^\gamma\jump{q}\,\p_t\TT^\gamma\psi\dx',\\
\label{def STmu'} {\ST^\mu}':=&~\eps^{4l}\is\TT^\gamma\jump{q}\,\left((\vb^+-\bar{\mu}^+\bc^+)\cdot\cnab\right)\TT^\gamma\psi\dx',\\
\label{def VSmu}\VS^\mu:=&~\eps^{4l}\is\TT^\gamma q^-\,\left(\jump{\vb-\bar{\mu} \bc}\cdot\cnab\right)\TT^\gamma\psi\dx',\\
\label{def RTmu}\RT^\mu:=&-\eps^{4l}\is\jump{\p_3 q}\TT^\gamma\psi\,\p_t\TT^\gamma\psi\dx',\\
\label{def RTmu'}\RT^{\pm,\mu}:=&\mp\eps^{4l}\is \p_3 q^\pm\,\TT^\gamma\psi\,\left((\vb^\pm-\bar{\mu}^\pm\bc^\pm)\cdot\cnab\right)\TT^\gamma\psi\dx',\\
\label{def ZBmu}ZB^{\pm,\mu}:=&\mp\eps^{4l}\is \QQ^{\pm}(\WW_v^{\pm}-\mu^\pm\WW_b^{\pm})\dx'
\end{align}

Among the terms \eqref{def STmu}-\eqref{def ZBmu},$\ST^\mu,{\ST^\mu}',\RT^\mu,\RT^{\pm,\mu}$ can be analyzed in the same way as in \cite[Section 3.3-3.4]{Zhang2023CMHDVS1}, so we no longer repeat the proof.
\begin{align*}
\frac{\sigma\eps^{4l}}{2}\is\frac{|\cnab\TT^\gamma\psi|^2}{{\sqrt{1+|\cnab\psi|^2}}^3}\dx'+\int_0^t\ST^\mu+{\ST^\mu}'+\RT^\mu+\RT^{\pm,\mu}\dtau\lesssim  \delta\EW(t)+P(\EW(0))+P(\EW(t))\int_0^tP(\EW(\tau))\dtau.
\end{align*} Also, the control of these terms do not depend on $1/\sigma$ thanks to the enhanced regularity of $\psi$ obtained in Section \ref{sect noncollinear}. 

{\bf With the unique choice of $\mu^\pm$ in \eqref{choice mu 3D}, the quantity $\vb-\mu \bc$ has NO jump across the interface $\Sigma$ and so$\VS^\mu=0$.} Finally, there is a cancellation structure in $ZB^{\pm,\mu}+Z^{\pm,\mu}$ which is similar to the one observed in step 4 of Section 3.3.1 and step 4 of Section 3.4.3 in \cite{Zhang2023CMHDVS1}. It suffices to replace $v^\pm$ by $v^\pm-\mu^\pm b^\pm$ and use $\nabp\cdot b^\pm=0$ in $\Om^\pm$ in order for the same result. 
\begin{align*}
\int_0^t ZB^{\pm,\mu}+Z^{\pm,\mu}\dtau\lesssim \delta\EW(t)+ P(\EW(0))+P(\EW(t))\int_0^tP(\EW(\tau))\dtau.
\end{align*}

\subsubsection{The stability condition ensures the hyperbolicity}
So far, we have obtained the following estimates for $\GG^{\pm,\mu}(t)$ in tangential estimates:
\begin{equation}\label{EmuTT1}
\begin{aligned}
&\sum_{\pm}\frac{\eps^{4l}}{2}\iopm \rho^\pm|\VV^\pm|^2+|\BB^\pm|^2+\ffpm(\PP^\pm)^2-2\mu^\pm\rho^\pm\VV^\pm\cdot\BB^\pm+2\mu^\pm\rho^\pm\ffpm \PP^\pm(b^\pm\cdot\VV^\pm)\dvt+\frac{\sigma\eps^{4l}}{2}\is \frac{|\TT^\gamma\cnab\psi|^2}{{\sqrt{1+|\cnab\psi|^2}}^3}\dx'\\
\lesssim&~\delta \EW(t)+ P(\EW(0)) + P(\EW(t))\int_0^t P(\EW(\tau))\dtau.
\end{aligned}
\end{equation}

Compared to the tangential estimates in \cite[Section 3.3-3.4]{Zhang2023CMHDVS1}, we must guarantee the \textit{positive-definiteness} of 
\begin{align*}
&|\sqrt{\rho^\pm}\VV^\pm|^2+|\BB^\pm|^2+\ffpm(\PP^\pm)^2-2\mu^\pm\rho^\pm\VV^\pm\cdot\BB^\pm +2\mu^\pm\rho^\pm\ffpm \PP^\pm(b^\pm\cdot\VV^\pm)
\end{align*} as a quadratic form of $(\sqrt{\rho^\pm}\VV^\pm,\BB^\pm,\sqrt{\ffpm}\PP^\pm)$ in $\Om^\pm$, respectively. We have
\begin{align*}
	&|\sqrt{\rho^\pm}\VV^\pm|^2+|\BB^\pm|^2+\ffpm(\PP^\pm)^2-2\mu^\pm\rho^\pm\VV^\pm\cdot\BB^\pm +2\mu^\pm\rho^\pm\ffpm \PP^\pm(b^\pm\cdot\VV^\pm)\\
	\geq &~|\sqrt{\rho^\pm}\VV^\pm|^2+|\BB^\pm|^2+\ffpm(\PP^\pm)^2-2\bno{\mu^\pm\sqrt{\rho^\pm}}\bno{\sqrt{\rho^\pm}\VV^\pm}\bno{\BB^\pm}-2\bno{\mu^\pm\sqrt{\rho^\pm}}\frac{ c_A^\pm}{c_s^\pm}\bno{\sqrt{\ffpm}\PP^\pm}\bno{\sqrt{\rho^\pm}\VV^\pm}
\end{align*} Here we use the fact that $\ffp=1/(\rho c_s^2)$ and $c_A:=|b|/\sqrt{\rho}$. Therefore, we woule require that the following matrix only has strictly positive eigenvalues
\[
\begin{pmatrix}
1&-|\mu^\pm|\sqrt{\rho^\pm}&-|\mu^\pm|\sqrt{\rho^\pm}\dfrac{c_A^\pm}{c_s^\pm}\\
-|\mu^\pm|\sqrt{\rho^\pm}&1&0\\
-|\mu^\pm|\sqrt{\rho^\pm}\dfrac{c_A^\pm}{c_s^\pm}&0&1
\end{pmatrix},
\]which further requires the following inequalities
\[
(\mu^\pm)^2\rho^\pm\left(1+\left(\frac{c_A^\pm}{c_s^\pm}\right)^2\right)< 1\text{ in }\Om^\pm \Rightarrow  |\bc^+\times\bc^-|> |\bc^\mp\times\jump{\vb}|\sqrt{\rho^\pm\left(1+\left(\frac{c_A^\pm}{c_s^\pm}\right)^2\right)}\text{ on }\Sigma.
\] Thus, we find that the stability condition \eqref{syrov 3D com}, namely $$\exists \delta_0\in(0,\frac18),\text{ such that } (1-\delta_0) |\bc^+\times\bc^-|\geq |\bc^\mp\times\jump{\vb}|\sqrt{\rho^\pm\left(1+\left(\frac{c_A^\pm}{c_s^\pm}\right)^2\right)}>0\text{ holds on }\Sigma,$$ exactly ensures the positive-definiteness of $\GG^{\pm,\mu}(t)$. Plugging this into the energy inequality \eqref{EmuTT1}, we find that there exists some constant $\delta_0\in(0,\frac18)$, such that the following estimate holds for the Alinhac good unknowns with respect to the tangential derivative $\eps^{4l}\TT^\gamma~(0\leq l\leq 4,\len{\gamma}=4+l)$.
\begin{equation}\label{EmuTT1}
\begin{aligned}
&\sum_{\pm}\frac{\delta_0}{2}\eps^{4l}\iopm \rho^\pm|\VV^\pm|^2+|\BB^\pm|^2+\ffpm(\PP^\pm)^2\dvt+\frac{\sigma\eps^{4l}}{2}\is \frac{|\TT^\gamma\cnab\psi|^2}{{\sqrt{1+|\cnab\psi|^2}}^3}\dx'\\
\lesssim&~\delta \EW(t)+ P(\EW(0)) + P(\EW(t))\int_0^t P(\EW(\tau))\dtau,\quad\forall\delta\in(0,\delta_0/100).
\end{aligned}
\end{equation} Finally, we just use the definition of Alinhac good unknowns to recover the $\eps^{2l}$-weighted $\TT^\gamma$-tangential estimates as recorded in Proposition \ref{prop ETT} by following the argument in \cite[Section 3.5]{Zhang2023CMHDVS1}.

\subsection{Incompressible and zero-surface-tension-limits under the stability condition}\label{sect limit 0ST}
Combining the tangential estimates \eqref{EmuTT1}, the enhanced regularity for $\psi$ obtained in Section \ref{sect noncollinear} and the div-curl analysis recorded in Section \ref{sect lwp energy}, we conclude the uniform-in-$(\eps,\sigma)$ estimates of $\EW(t)$ by
\begin{align}
\forall\delta\in(0,\frac{\delta_0}{100}),\quad\EW(t)\lesssim\delta \EW(t)+ P(\EW(0)) + P(\EW(t))\int_0^t P(\EW(\tau))\dtau.
\end{align}Using Gronwall-type argument, we know there exists some $T>0$ independent of $(\es)$ such that
\begin{equation}
\sup_{0\leq t\leq T} \EW(t)\leq P(\EW(0)).
\end{equation} 
With the uniform-in-($\es$) estimates, we can pass the limit $\es\to 0_+$ to the incompressible current-vortex sheets problem without surface tension under the non-collinearity condition. Given $\es>0$, let $(\psi^\es, v^{\pm,\es}, b^{\pm,\es}, \rho^{\pm,\es}, S^{\pm,\es})$ be the solution to \eqref{CMHDVS0} with initial data $(\psi_0^\es, v_0^{\pm,\es}, b_0^{\pm,\es}, \rho_0^{\pm,\es}, S_0^{\pm,\es})$ and let $(\xi^0, w^{\pm,0}, h^{\pm,0},\mathfrak{S}^{\pm,0})$ be the solution to \eqref{IMHDs} with $\sigma=0$ with initial data $(\xi_0^0, w_0^{\pm,0}, h_0^{\pm,0}, \mathfrak{S}_0^{\pm,0})$. We assume 
\begin{enumerate}
\item [a.] $(\psi_0^\es, v_0^{\pm,\es}, b_0^{\pm,\es}, S_0^{\pm,\es}) \in H^{9.5}(\Sigma)\times H_*^8(\Omega^\pm)\times H_*^8(\Omega^\pm)\times H_*^8(\Omega^\pm)$ satisfies the compatibility conditions \eqref{comp cond} up to 7-th order, the stability condition \eqref{syrov 3D com 1} and {$|\psi_0^\es|_{\infty} \leq 1$}. 
\item [b.] $(\psi_0^\es, v_0^{\pm,\es}, b_0^{\pm,\es}, S_0^{\pm,\es}) \to (\xi_0^0, w_0^{\pm,0}, h_0^{\pm,0}, \mathfrak{S}_0^{\pm,0})$ in $H^{4.5}(\Sigma) \times H^4(\Omega^\pm)\times H^4(\Omega^\pm)\times H^4(\Omega^\pm)$ as $\eps, \sigma\to 0$. 
\item [c.] The incompressible initial data satisfies the constraints $\nab^{\xi_0}\cdot h_0^{\pm}=0$ in $\Om^\pm$, $h^{\pm}\cdot N^0|_{\{t=0\}\times\Sigma}=0$, the stability condition \begin{align}
 2\delta_0\leq \sqrt{\rr_0^\pm}\left|\bar{h}_0^\mp\times[\bar{w}_0]\right|\leq (1-2\delta_0)|\bar{h}_0^+\times \bar{h}_0^-|~~ \text{ on }\Sigma,
\end{align}where $\delta_0>0$ is the same constant as in \eqref{syrov 3D com 1}.
\end{enumerate}
Then, by the Aubin-Lions compactness lemma, it holds that 
\begin{align}
(\psi^\es, v^{\pm,\es}, b^{\pm,\es}, S^{\pm,\es})\to(\xi^0, w^{\pm,0}, h^{\pm,0},\mathfrak{S}^{\pm,0}),
\end{align} weakly-* in $L^\infty([0,T]; H^{4.5}(\Sigma)\times(H^{4}(\Om^\pm))^3)$ and strongly in  $C([0,T]; H_{\text{loc}}^{4.5-\delta}(\Sigma)\times(H_{\text{loc}}^{4-\delta}(\Om^\pm))^3)$ after possibly passing to a subsequence. Theorem \ref{thm CMHDlimit2} is proven.

\subsection{Double limits in 2D: a subsonic zone}\label{sect 2D syrov}
When the space dimension $d=2$, the substantial part of the proof for well-posedness, uniform estimates and limit process remains unchanged. In fact, we shall only re-consider the following aspects 
\begin{itemize}
\item The curl operator now becomes $\nabper\cdot:=(-\pp_2,\pp_1)\cdot$, so we need to check the special structure given by Lorentz force in the vorticity analysis. This has been analyzed in \cite[Section 3.6.4]{Zhang2023CMHDVS1}.
\item The interface is now a 1D curve instead of a 2D surface, thus it is impossible to have ``non-parallel" magnetic fields $b^\pm$ on $\Sigma$. The functions $\mu^\pm$ are no longer uniquely determined by $b_1^\pm$. 
\end{itemize}

 We prove the following lemma about the choice of Friedrichs symmetrizer.
\begin{lem}
There exist functions $\bar{\mu}^\pm(t,x_1)$ satisfying $$\jump{v_1-\bar{\mu}b_1}=0~\text{ and }~|\bar{\mu}^\pm|<1/a^\pm~~~\text{ on } [0,T]\times\Sigma,~~~a^\pm:=\sqrt{\rho^\pm\left(1+(c_A^\pm/c_s^\pm)^2\right)}$$ if and only if the following inequality holds \begin{align}\label{syrov2Dcom}|\jump{v_1}|<\frac{|b_1^+|}{a^+}+\frac{|b_1^-|}{a^-}~~~\text{ on } [0,T]\times\Sigma.\end{align} Under \eqref{syrov2Dcom}, the functions $\bar{\mu}^\pm$ are chosen to be
\begin{align}\label{choice mu 2D}
\bar{\mu}^\pm=\pm\frac{\sgn(b_1^\pm)\,a^\mp\jump{v_1}}{a^-|b_1^+|+a^+|b_1^-|}.
\end{align}
\end{lem}
\begin{proof}
First, we shall exclude the possibility for $b_1^\pm=0$ at some point $(x_1,0)\in\Sigma$ because of $\jump{v_1}\neq 0$ everywhere on $\Sigma$. 

\textbf{Case 1:} \underline{One of the two magnetic fields is vanishing on $\Sigma$}, e.g., we assume $|b_1^+|>0=|b_1^-|$ on $\Sigma$, then $\jump{v_1-\bmu b_1}=0$ directly gives us $\bmu^+=\jump{v_1}/b_1^+$ and $\bmu^-$ can be any function satisfying the hyperbolicity constraint $|\bmu^-|<1/a^-$. For simplicity, we may choose $\bmu^-=0$. Solving the constraint $|\bmu^+|<1/a^+$ gives us $|\jump{v_1}|<|b_1^+|/a^+$ as a special case of \eqref{syrov2Dcom}. Similarly, when $|b_1^-|>0=|b_1^+|$ on $\Sigma$, we can choose $\bmu^-=-\jump{v_1}/b_1^-$ and $\bmu^+$ can be any function satisfying the hyperbolicity constraint $|\bmu^+|<1/a^+$ and we may choose $\bmu^+=0$ for simplicity.
\begin{rmk}
One can verify that if $b_0=0$ on $\Sigma$, then $b$ must be identically zero on $\Sigma$. In fact, restricting the equation of $b$ onto $\Sigma$ and doing $L^2$ estimate shows that $\ddt|b^\pm|_{0}^2\leq C|\p v|_{L^{\infty}}|b^\pm|_{0}^2$. Using Gronwall's inequality and $b_0|_{\Sigma}=0$ yields the result.
\end{rmk}

\textbf{Case 2:}  \underline{$b_1^\pm$ are not identically zero on $\Sigma$.} In this case, we may assume $|b_1^\pm|>0$ on $\Sigma$ as well. In fact, if $b_1^-$ vanishes at some point $(x_1,0)\in\Sigma$, then we can follow the choice of $\bmu^\pm$ as in case 1 to determine the function $\bmu^\pm$ at this point. Note that the functions $\bmu^+=\jump{v_1}/b_1^+,~\bmu^-=0$ still satisfies \eqref{choice mu 2D}, so they do not break the continuity and differentiability of \eqref{choice mu 2D} at the points where one of $b_1^\pm$ vanishes. Thus, making the assumption  $|b_1^\pm|>0$ on $\Sigma$ is reasonable. 

The ``if" part is easy to prove. Indeed, when the stability condition \eqref{syrov2Dcom} holds on $\Sigma$, we can set $\bmu^\pm$ as in \eqref{choice mu 2D}. The direct computation shows that such $\bmu^\pm$ satisfy $\jump{v_1-\bmu b_1}=0$ on $\Sigma$ and $|\bmu^\pm|<1/a^\pm$. Let us prove the ``only if" part. When  $|b_1^\pm|>0$ on $\Sigma$, we can write 
\[
\bmu^+=\frac{\jump{v_1}+\bmu^-b_1^-}{b_1^+}.
\]Using $|\bmu^+|<1/a^+$, we can solve the inequality by
\[
-\frac{1}{a^+}-\frac{\jump{v_1}}{b_1^+}<\frac{b_1^-}{b_1^+}\bmu^-<\frac{1}{a^+}-\frac{\jump{v_1}}{b_1^+},
\]where $\bmu^-$ should also satisfy $|\bmu^-|<1/a^-$. Assume $b_1^\pm<0$ for simplicity (that is, the horizontal directions of $b^\pm$ on $\Sigma$ are the same). Combining these two requirements, we find that the following inequality is necessary
\begin{align}
|\jump{v_1}|<-\frac{b_1^+}{a^+}-\frac{b_1^-}{a^-}=\frac{|b_1^+|}{a^+}+\frac{|b_1^-|}{a^-}.
\end{align}Similar calculation for the case $b_1^\pm >0$ and the case $b_1^+b_1^-<0$ also leads to the same inequality as above.
\end{proof}

\section{Incompressible limit for vortex-sheet problems: A paradifferential approach}\label{sect limit well}
Recall that the uniform-in-$\es$ control of $\EW(t)$ requires $\p_t^k v|_{t=0}=O(1)$ for $0\leq k\leq 4$, while the usual definition of ``well-prepared initial data'' only requires $\p_t v|_{t=0}=O(1)$ and $\nab^{\vp_0}\cdot v_0=O(\eps)$. In this final section, we aim to drop the uniform boundedness (in $\eps$) assumption for $\geq 2$ time derivatives. Compared with the energy $E(t)$ that we use to prove the local existence in \cite{Zhang2023CMHDVS1}, a new difficulty arises in the control of the ``weaker" energy $\EE(t)$: There exhibits a loss of weight of Mach number in $\TP^3\p_t$-tangential estimates when analyzing $\EE_4(t)$. In particular, we have to control the following quantity in the cancellation structure in $Z^\pm+ZB^\pm$, 
\[
\io (\TP^2\p_3\p_t v_i)(\TP\NN_i)(\TP^2\p_3\p_t q)\dx,\text{ which arises from }\iopm (\TP^3\p_t v_i^\pm) [\TP^3, \NN_i, \TP^3\p_t q]\dx.
\]In this term, $\p_t q$ has to be uniformly bounded with respect to Mach number. However, now we only have $\nabp\cdot v=O(\eps)$ and $\p_t q=O(1/\eps)$, which leads to a loss of $\eps$-weight. Besides, similar difficulty also appears in the control of $-\iopm \VV^\pm\cdot\cc(q^\pm)\dvt$. Indeed, such loss of $\eps$-weight necessarily happens in  $\TP^3\p_t$-tangential estimates because of the following two reasons
\begin{enumerate}
\item $\TP^3\p_t q$ needs one more $\eps$-weight than $\TP^3\p_t v$;
\item The (extension of) normal vector $\NN$, which arises from the commutator $[\TP^3\p_t,\NN_i/\p_3\vp, \p_3 f]$ in $\cc_i(f)$, may NOT absorb a time derivative. 
\end{enumerate}
%Moreover, when $1\leq l\leq 2$, we consider the control $\ino{\eps^{(k-1)_++2l}\p_t^k\TT^\alpha(v,b)}_{4-k-l}^2$ with $1\leq k\leq 3-l,~\len{\alpha}=2l$ in $\EE_{4+l}(t)$. When $\TT^\alpha=\p_t^{2l}$,  the loss of $\eps$-weight also appears in $\TP^{4-k-l}\TT^{\alpha}$-tangential estimates due to the same reason as above.

As we can see, \textit{this type of difficulty never appears in the fixed-domain setting because the commutator terms in $\cc(f)$ are contributed by the free-interface motion.} Unless the fluid is one-phase and the surface tension is neglected \cite{Zhang2021elasto}, such essential difficulty cannot be avoided as long as we apply the div-curl inequality \eqref{divcurlTT} to $v_t, b_t$. To get rid of the loss of Mach number, we have to find a new way to control $v_t$, $b_t$ and also avoid the appearance of $|\sqrt{\sigma}\cnab\p_t\psi|_3$ without $\eps$-weight.

\subsection{The weaker energy for the improved incompressible limit}\label{sect reduce EE4}
As in \eqref{energy total}-\eqref{energy es total}, we now assume the fluid to be isentropic and consider the new energy functional  
\begin{align*}
 \EE(t):=\EE_{4}(t)+E_{5}(t)+E_{6}(t)+E_{7}(t)+E_{8}(t)\\
\EEW(t):=\EEW_{4}(t)+\EW_{5}(t)+\EW_{6}(t)+\EW_{7}(t)+\EW_{8}(t)
\end{align*} where 
\begin{align}
\EE_4(t)=\sum_\pm&\ino{(v^\pm,b^\pm,p^\pm)}_{4,\pm}^2+\bno{\sqrt{\sigma}\psi}_{5}^2+\ino{\p_t(v^\pm,b^\pm,(\ffpm)^{\frac12}p^\pm)}_{3,\pm}^2+\bno{\sqrt{\sigma}\p_t\psi}_{4}^2 \no\\
&+\sum_{k=2}^{4}\ino{\eps\p_t^k(v^\pm,b^\pm,(\ffpm)^{\frac{(k-3)_+}{2}}p^\pm)}_{4-k,\pm}^2+\bno{\sqrt{\sigma}\eps\p_t^k\psi}_{5-k}^2
\end{align}and 
\begin{align}
\EEW_4(t)=\EE_4(t)+\bno{\psi}_{4.5}^2+\bno{\p_t\psi}_{3.5}^2+\bno{\p_t^2\psi}_{2.5}^2+\bno{\eps\p_t^3\psi}_{1.5}^2+\bno{\eps\p_t^4\psi}_{0.5}^2.
\end{align}
We aim to prove uniform-in-$\eps$ estimates for $\EE(t)$ for each $\sigma>0$ and prove uniform-in-$(\es)$ estimates for $\EEW(t)$ under the stability condition \eqref{syrov 3D com} (replaced with \eqref{syrov 2D com} in the 2D case). Let us analyze the estimates for different $k$ in $\EE_4(t)$ and $\EEW_4(t)$.

\subsubsection*{The case $k=0$} When $k=0$, the reduction of $v,b$ is the same as in Section \ref{sect lwp energy}. That is, we use the div-curl analysis to convert normal derivatives to tangential derivatives
\begin{align}
\|(v^\pm,b^\pm)\|_{4,\pm}^2\leq&~C(|\psi|_4,|\cnab\psi|_{W^{1,\infty}})\left(\|v^\pm,b^\pm\|_{0,\pm}^2+\ino{\nabp\cdot v^\pm,\nabp\times(v^\pm,b^\pm)}_{3,\pm}^2+\ino{\TP^4(v^\pm,b^\pm)}_{0,\pm}^2\right),\\
\ino{\nabp\cdot v^\pm}_{3,\pm}^2\lesssim&~\ino{\ffpm\Dtpm p^\pm}_{3,\pm}^2,\quad \ino{\nabp\times(v^\pm,b^\pm)}_{3,\pm}^2\lesssim~ \delta \EE_4(t)+\int_0^t P(\EE_4(\tau))+E_5(\tau)\dtau,\\
\ino{\nab q}_{3,\pm}^2 \lesssim&~\ino{\rho\Dtpm v^\pm}_{3,\pm}^2+\ino{\bpm b^\pm}_{3,\pm}^2.
\end{align} The $\TP^4$-control is proved in almost the same way as in \cite[Section 3.3]{Zhang2023CMHDVS1} which gives the control of $\bno{\sqrt{\sigma}\cnab\psi}_4^2$. The only difference is the treatment of$\RT$ because we need to avoid using $\sqrt{\sigma}$-weight enenrgy when taking the limit $\sigma\to 0$. Using Kato-Ponce type product estimate \eqref{product} in Lemma \ref{KatoPonce}, we have 
\begin{align}
\RT=\is\jump{\p_3 q} \TP^4\psi\TP^4\psi_t \leq  \bno{\jump{\p_3 q}\TP^4 \psi}_{\frac12} |\psi_t|_{3.5} \leq \|q^\pm\|_{3,\pm}|\psi|_{4.5}|\psi_t|_{3.5}.
\end{align}So, we need to find $(\es)$-independent control of $\bno{\psi}_{4.5}^2$ and $\bno{\p_t\psi}_{3.5}^2$.

\subsubsection*{The case $k=1$} When $k=1$, we cannot use the above div-curl inequality because we must avoid $\TP^3\p_t$-estimate. Instead, we use the div-curl inequality \eqref{divcurlNN} to get
\begin{align}
\|(\p_tv^\pm,\p_tb^\pm)\|_{3,\pm}^2\leq C(|\psi|_{3.5},|\cnab\psi|_{W^{1,\infty}})\bigg(&\|\p_tv^\pm,\p_tb^\pm\|_{0,\pm}^2+\ino{\nabp\cdot (\p_tv^\pm,\p_tb^\pm)}_{2,\pm}^2+\ino{\nabp\times(\p_tv^\pm,\p_tb^\pm)}_{2,\pm}^2\no \\
&+\bno{\p_t v^\pm\cdot N,\p_tb^\pm\cdot N}_{2.5}^2\bigg).
\end{align}The divergence part and the curl part are controlled in the same way as \cite[Section 3.6.1]{Zhang2023CMHDVS1}, so we do not repeat the analysis here. The boundary normal trace for $b_t$ is easy to control. Using $b^\pm\cdot N=0$, we have $b_t\cdot N=\bc\cdot\cnab\psi_t$ and thus
\begin{align}
\bno{\p_t b^\pm\cdot N}_{2.5}^2=\bno{\bc^\pm\cdot\cnab\psi_t}_{2.5}^2\lesssim \|b^\pm\|_{3,\pm}^2 |\psi_t|_{3.5}^2.
\end{align}
For the normal trace $\bno{\p_t v^\pm\cdot N}_{2.5}^2$, we invoke the kinematic boundary condition $\p_t\psi=v^\pm\cdot N$ to get
\begin{align}
\bno{\p_t v^\pm\cdot N}_{2.5}^2\leq \bno{\p_t^2\psi}_{2.5}^2+\bno{\vb^\pm\cdot\cnab\psi_t}_{2.5}^2\lesssim  \bno{\p_t^2\psi}_{2.5}^2+ \|v^\pm\|_{3,\pm}^2 |\psi_t|_{3.5}^2.
\end{align}
Since we avoid $\TP^3\p_t$-tangential estimates, we must seek for another way to find $\eps$-independent estimates for $\bno{\p_t^2\psi}_{2.5}^2$ and $\bno{\sqrt{\sigma}\p_t\psi}_4^2$. Also, under the stability condition \eqref{syrov 3D com}, we need to find $(\es)$-independent control of $\bno{\psi}_{4.5}^2,\bno{\p_t\psi}_{3.5}^2,\bno{\p_t^2\psi}_{2.5}^2$ and $\bno{\sqrt{\sigma}\p_t\psi}_4^2$.

\subsubsection*{The case $2\leq k\leq 4$}
When $k=2,3,4$, the reduction stays the same as in \cite[Section 3.3-3.5]{Zhang2023CMHDVS1} (for the incompressible limit for fixed $\sigma>0$) and the analysis in Section \ref{sect symmetrize}. The reason is that $\p_t^k q$ share the same weight of Mach number as $\p_t^k v$ which helps us avoid the loss of $\eps$-weight in $\cc(q)$.

\subsection{The evolution equation of the free interface and its paralinearization}\label{sect DtNST}
To prove the uniform-in-$\eps$ estimates for $\EE(t)$  and the uniform-in-$(\es)$ estimates for $\EEW(t)$ under the stability condition \eqref{syrov 3D com}, it remains to prove the $\eps$-independent control of $|\psi|_{4.5},~|\psi_{t}|_{3.5},~|\psi_{tt}|_{2.5}$ and $|\sqrt{\sigma}\psi_t|_4$ by $P(\EE(0))+P(\EE(t))\int_0^t P(\EE(\tau))\dtau$ and $(\es)$-independent control of them by $P(\EEW(0))+P(\EEW(t))\int_0^t P(\EEW(\tau))\dtau$. Since we must avoid $\TP^3\p_t$-tangential estimates, we shall further analyze the evolution equation of the free interface. 

\subsubsection{Derivation of the equation}\label{sect psi eq}
We take $\p_t$ in the kinematic boundary condition to get $\p_t^2\psi=\p_t v^\pm\cdot N-\vb^\pm\cdot\cnab\p_t\psi$. Plugging the momentum equation of \eqref{CMHDVS0} into the term $\p_t v^\pm\cdot N$, we get
\begin{align*}
\p_t v^\pm\cdot N=-\frac{1}{\rho^\pm}N\cdot\nabp q^\pm - (\vb^\pm\cdot\cnab) v^\pm\cdot N +\frac{1}{\rho^\pm} (\bc^\pm\cdot\cnab) b^\pm\cdot N \quad \text{ on }\Sigma.
\end{align*} Using $\p_t \psi=v^\pm\cdot N$ and $b^\pm\cdot N=0$ on $\Sigma$, we have
\begin{align*}
- (\vb^\pm\cdot\cnab) v^\pm\cdot N =&-(\vb^\pm\cdot\cnab)\p_t\psi + v^\pm\cdot (\vb^\pm\cdot\cnab) N=-\vb_j^\pm\TP_j\p_t\psi - \vb_i^\pm\vb_j^\pm \TP_i\TP_j\psi,\\
(\bc^\pm\cdot\cnab) b^\pm\cdot N =&~ \bc_i^\pm\bc_j^\pm \TP_i\TP_j\psi,
\end{align*}and thus 
\begin{align}
\p_t^2\psi = -\frac{1}{\rho^\pm}N\cdot\nabp q^\pm +\left(\frac{1}{\rho^\pm}\bc_i^\pm\bc_j^\pm-\vb_i^\pm\vb_j^\pm\right)\TP_i\TP_j\psi- 2(\vb^\pm\cdot\cnab)\p_t\psi.
\end{align}
Next we want to separate the boundary value of $q^\pm$ from its interior contribution in order to create an energy term involving the surface tension. First, taking $\nabp\cdot$ in the momentum equation and invoking the continuity equation in \eqref{CMHDVS0}, we derive a wave-type equation
\[
\ffpm (\Dtpm)^2 p^\pm -\lapp q^\pm = (\pp_i v_j^\pm)(\pp_j v_i^\pm) - (\pp_i b_j^\pm)(\pp_j b_i^\pm),
\]which can be written as a wave equation of $q^\pm$ thanks to $q^\pm=p^\pm+\frac12|b^\pm|^2$
\begin{align}\label{MHDwave}
\ffpm(\Dtpm)^2 q^\pm -\lapp q^\pm = \eps^2(\Dtpm)^2\left(\frac12|b^\pm|^2\right) + (\pp_i v_j^\pm)(\pp_j v_i^\pm) - (\pp_i b_j^\pm)(\pp_j b_i^\pm)\quad \text{ in }[0,T]\times\Om^\pm,
\end{align}with a jump condition $\jump{q}=\sigma\h(\psi)$ on $\Sigma$ and a Neumann-type boundary condition $\p_3 q^\pm=0$ on $\Sigma^\pm$ (got by restricting the momentum equations on $\Sigma^\pm$), where we omit the terms in which $\Dtpm$ falls on $\ffpm$.

\begin{defn}[\textbf{Dirichlet-to-Neumann operator}] \label{defn DtN}
For a function $f:\Sigma\to\R$, we now define the Dirichlet-to-Neumann operator with respect to $(\psi,\Om^\pm)$ by 
\begin{align}
\dnpm f:=\mp N\cdot\nabp(\HE_\psi^\pm f),
\end{align}where $\HE_\psi^\pm f$ is defined to be the harmonic extension of $f$ into $\Om^\pm$, namely
\begin{equation}
-\lapp (\HE_\psi^\pm f)= 0 \quad \text{ in }\Omega^\pm,\quad \HE_\psi^\pm f=f \quad \text{ on }\Sigma,\quad \p_3(\HE_\psi^\pm f)=0 \quad \text{ on }\Sigma^\pm.
\end{equation} 
\end{defn}

Thus, we can define a decomposition $q^\pm=q_\psi^\pm+q_w^\pm$ satisfying
\begin{equation}
q_\psi^\pm:=\HE_\psi^\pm(q^\pm|_{\Sigma})\text{ in }\Omega^\pm
\end{equation}and 
\begin{equation}\label{qw equ}
-\lapp q_w^\pm = -\ffpm(\Dtpm)^2 q^\pm+\ffpm(\Dtpm)^2\left(\frac12|b^\pm|^2\right) + (\pp_i v_j^\pm)(\pp_j v_i^\pm) - (\pp_i b_j^\pm)(\pp_j b_i^\pm)\text{ in }\Omega^\pm
\end{equation}  
with boundary conditions $q_{w}^\pm=0 \text{ on }\Sigma$ and $N\cdot\nabp q_w^\pm = \p_3 q_w^\pm = 0$ on $\Sigma^\pm$. The second boundary conditions holds thanks to the slip condition for $v_3^\pm,b_3^\pm$ on $\Sigma^\pm$. Thus, the evolution equation of $\psi$ can be written as
\begin{align}
\rho^\pm\p_t^2\psi =& \pm\dnpm (q^\pm|_{\Sigma} )-N\cdot\nabp q_w^\pm +\left(\bc_i^\pm\bc_j^\pm-\rho^\pm\vb_i^\pm\vb_j^\pm\right)\TP_i\TP_j\psi- 2(\rho^\pm\vb^\pm\cdot\cnab)\p_t\psi.
\end{align} We now want to resolve $q^\pm|_{\Sigma}$ in terms of $\rho^\pm$ and $F_\psi^\pm$ by inverting the Dirichlet-to-Neumann operators $\dnpm$. However, we no longer have $\is \rho\p_t\psi\dx'=\is\rho\p_t^2\psi\dx'=0$ due to the compressibility of fluids. Thus, we have to eliminate the zero-frequency part in $\rho^\pm\p_t^2\psi$ before inverting the Dirichlet-to-Neumann operators. In other words, we cannot directly apply $(\dnpm)^{-1}$ to $(1/\rho)\dnpm (q^\pm|_{\Sigma})$. 

For a function $f:\Sigma=\T^2\to \R$, we define the Littlewood-Paley projection $$\LP_{\neq 0} f:=f - (f)_{\Sigma},~~(f)_\Sigma:=\int_{\T^2}f\dx'.$$ Under this setting, we have
\begin{equation}
\rho^\pm\p_t^2\psi=\LP_{\neq 0}(\rho^\pm\p_t^2\psi) + (\rho^\pm\p_t^2\psi)_{\Sigma}%=\rho^\pm P_{\neq 0}(\p_t^2\psi) + [\LP_{\neq 0},\rho^+]\p_t^2\psi + (\rho^\pm\p_t^2\psi)_{\Sigma},
\end{equation}and we insert it back to the evolution equation to get
\begin{align}
\LP_{\neq 0}(\rho^\pm\p_t^2\psi)=&~ \pm\dnpm (q^\pm|_{\Sigma} ) -(\rho^\pm\p_t^2\psi)_\Sigma - N\cdot\nabp q_w^\pm +\left(\bc_i^\pm\bc_j^\pm-\rho^\pm\vb_i^\pm\vb_j^\pm\right)\TP_i\TP_j\psi- 2(\rho^\pm\vb^\pm\cdot\cnab)\p_t\psi\\
=&: \pm\dnpm (q^\pm|_{\Sigma} ) + F_\psi^\pm.
\end{align}
Note that the zero-frequency modes of both $\LP_{\neq 0}(\rho^\pm\p_t^2\psi)$ and $\pm\dnpm (q^\pm|_{\Sigma} )$ are vanishing on the interface $\Sigma$, so we deduce that $\is F_\psi^\pm=0$ and then $(\dnpm)^{-1}(F_\psi^\pm)$ is well-defined. Now we can resolve the traces $q^\pm|_{\Sigma}$ from the evolution equations of $\psi$. We have
\begin{align}
&\dnp (q^+|_{\Sigma} ) + \dnm (q^-|_{\Sigma} ) = \LP_{\neq 0}(\jump{\rho}\p_t^2\psi) -F_\psi^+ + F_\psi^- \no\\
\Rightarrow&\mp\dnmp (\jump{q}|_{\Sigma} ) + \left(\dnp+\dnm\right)(q^\pm|_{\Sigma}) =\LP_{\neq 0}(\jump{\rho}\p_t^2\psi)  -F_\psi^+ + F_\psi^-\no\\
\Rightarrow&~q^\pm|_{\Sigma} = \dnw^{-1}\left(\pm\dnmp (\sigma\h(\psi))+\LP_{\neq 0}(\jump{\rho}\p_t^2\psi) -\jump{F_\psi}\right), \label{resolve trace q}
\end{align}where $\dnw:=\dnp+\dnm$ (and equivalently we have $\dnpm=\frac12(\dnw\pm(\dnp-\dnm))$ represents the mixed Dirichlet-to-Neumann operator and $\jump{F_\psi}:=F_\psi^+-F_\psi^-$. %With this notation, we know that for any $f^\pm:\Sigma\to\R$ with $\is f^\pm\dx'=0$, it holds that
%\[
%f^\pm-\dnp\dnw^{-1}f^\pm=\dnw\dnw^{-1}f-\dnp \dnw^{-1} f= \dnm\dnw^{-1} f
%\]

Plugging \eqref{resolve trace q} back into the evolution equation of the free interface, we get
\begin{align}
\LP_{\neq 0}(\rho^+\p_t^2\psi) =&~ \dnp(q^+|_{\Sigma}) + F_\psi^+ =\dnp\dnw^{-1}\left(\dnm(\sigma\h(\psi))+\LP_{\neq 0}(\jump{\rho}\p_t^2\psi)-F_\psi^+ + F_\psi^- \right) + F_\psi^+ \no\\
=&~\sigma\dnp\dnw^{-1}\dnm(\h(\psi)) +\dnp\dnw^{-1} F_\psi^- - \dnp\dnw^{-1} F_\psi^+ + F_\psi^+ + \dnp\dnw^{-1}\left(\LP_{\neq 0}(\jump{\rho}\p_t^2\psi)\right)\no\\
=&~\sigma\dnp\dnw^{-1}\dnm(\h(\psi)) + \dnm\dnw^{-1} F_\psi^+ +\dnp\dnw^{-1} F_\psi^- + \dnp\dnw^{-1}\left(\LP_{\neq 0}(\jump{\rho}\p_t^2\psi)\right). \label{psi equ +}
\end{align}
Similarly, we have
\begin{align}
\LP_{\neq 0}(\rho^-\p_t^2\psi) =\sigma\dnm\dnw^{-1}\dnp(\h(\psi)) + \dnm\dnw^{-1} F_\psi^+ +\dnp\dnw^{-1} F_\psi^- - \dnm\dnw^{-1}\left(\LP_{\neq 0}(\jump{\rho}\p_t^2\psi)\right).\label{psi equ -}
\end{align}

Now, using the expressions of $\dnpm$ in terms of $\dnw$ and $\dnp-\dnm$, we have
\begin{align}
\label{DtNST+}\dnp\dnw^{-1}\dnm f =&~\frac12\dnp\dnw^{-1}(\dnw f+(\dnp-\dnm) f) =~\frac12\dnp f +\frac12\dnp\dnw^{-1}(\dnp-\dnm) f,\\
\label{DtNST-}\dnm\dnw^{-1}\dnp f =&~\frac12\dnm\dnw^{-1}(\dnw f-(\dnp-\dnm) f) =~\frac12\dnm f -\frac12\dnm\dnw^{-1}(\dnp-\dnm) f,
\end{align}and also for $g^\pm:\Sigma\to\R$ with $\is g^\pm\dx'=0$, we have
\begin{align}
\label{DtNF} \dnm\dnw^{-1} g^+ +\dnp\dnw^{-1} g^- =& \frac12(\dnw - (\dnp-\dnm))\dnw^{-1}g^+ + \frac12(\dnw + (\dnp-\dnm))\dnw^{-1} g^- \no\\
=&~\frac{g^+ + g^-}{2} -\frac12(\dnp-\dnm)\dnw^{-1}\jump{g}.
\end{align} Let $f=\h(\psi)$ and $g^\pm=F_\psi^\pm$ in \eqref{DtNST+}-\eqref{DtNF}. We find that $\eqref{psi equ +}+\eqref{psi equ -}$ can be written as
\begin{align}
\LP_{\neq 0}(\rho^+\p_t^2\psi)+\LP_{\neq 0}(\rho^-\p_t^2\psi) =&~ \frac{\sigma}{2}\left(\dnp+\dnm\right)(\h(\psi)) +F_\psi^+ + F_\psi^- \no\\
& + \frac{\sigma}{2}(\dnp-\dnm)\dnw^{-1}(\dnp-\dnm)(\h(\psi))\no\\
& -(\dnp-\dnm)\dnw^{-1}(F_\psi^+ - F_\psi^-)+(\dnp-\dnm)\dnw^{-1}\left(\LP_{\neq 0}(\jump{\rho}\p_t^2\psi)\right).
\end{align} Recall that $F_\psi^\pm = \FP^\pm - (\rho^\pm\p_t^2\psi)_{\Sigma}$ and $\rho^\pm\p_t^2\psi=\LP_{\neq 0}(\rho^\pm\p_t^2\psi)+(\rho^\pm\p_t^2\psi)_{\Sigma}$ where 
$$\FP^\pm:=-N\cdot\nabp q_w^\pm +\left(\bc_i^\pm\bc_j^\pm-\rho^\pm\vb_i^\pm\vb_j^\pm\right)\TP_i\TP_j\psi- 2(\rho^\pm\vb^\pm\cdot\cnab)\p_t\psi.$$ 
Thus, the evolution equation of the free interface becomes
\begin{align}
(\rho^++\rho^-)\p_t^2\psi=&~\frac{\sigma}{2}\left(\dnp+\dnm\right)(\h(\psi)) +\left(\bc_i^+\bc_j^+-\rho^+\vb_i^+\vb_j^+ + \bc_i^-\bc_j^--\rho^-\vb_i^-\vb_j^-\right)\TP_i\TP_j\psi-2(\rho^+\vb_i^++\rho^-\vb_i^-)\TP_i\p_t\psi \no\\
& - N\cdot\nabp q_w^+ - N\cdot\nabp q_w^-\no\\
& +\frac{\sigma}{2}(\dnp-\dnm)\dnw^{-1}(\dnp-\dnm)(\h(\psi)) -(\dnp-\dnm)\dnw^{-1}\left(\jump{\FP-\rho\p_t^2\psi}\right)\label{psi equ},
\end{align}where the first line is expected to give the $\sqrt{\sigma}$-weighted regularity (contributed by surface tension) and the non-weighted regularity (provided that stability condition \eqref{syrov 3D com}) for the free interface, the second line will be converted to the interior estimate of the right side of \eqref{qw equ}, and the last line consists of remainder terms that can be directly controlled by using paradifferential calculus.

\subsubsection{Preliminaries on paradifferential calculus}\label{sect para pre}
In the equation \eqref{psi equ}, the term $\left(\dnp+\dnm\right)(\h(\psi))$ is a fully nonlinear term. Although it is well-known that the Dirichlet-to-Neumann operator is a first-order elliptic operator and the mean-curvature operator is a second-order elliptic operator, it is still necessary for us to find out their concrete forms and ``symmetrize" the paradifferential formulations in order for an explicit energy estimate. In this subsection, we introduce several preliminary lemmas about paradifferential calculus that have been proven in Alazard-Burq-Zuily \cite{ABZ2011IWWST}. Following the notations in M\'etivier \cite{MetivierPara}, we first introduce the basic definition of a paradifferential operator. Note that the dimension $d$ below is not the same as the one in Section 1.
\begin{defn}[Symbols]
Given $r\geq 0,~m\in\R$, we denote $\Gamma_r^m(\T^d)$ to be the space of locally bounded functions $a(x',\xi)$ on $\T^d\times(\R^d\backslash\{0\})$, which are $C^\infty$ with respect to $\xi (\xi\neq \bd{0})$, such that for any $\alpha\in\N^d,\xi\neq \bd{0}$, the function $x'\mapsto \p_\xi^\alpha a(x',\xi)$ belongs to $W^{r,\infty}(\T^d)$ and there exists a constant $C_\alpha$ such that
\[
\bno{\p_\xi^\alpha a(\cdot,\xi)}_{W^{r,\infty}(\T^d)}\leq C_\alpha (1+|\xi|)^{m-|\alpha|},~~\forall|\xi|\geq 1/2.
\]
\end{defn} 

\begin{defn}[Paradifferential operator]
Given a symbol $a$, we shall define the \textbf{paradifferential operator }$T_a$ by
\begin{align}
\widehat{T_a u}(\xi):=(2\pi)^{-d}\int_{\R^d} \tilde{\chi}(\xi-\eta,\eta)\hat{a}(\xi-\eta,\eta)\phi(\eta)\hat{u}(\eta)\deta
\end{align}where $\hat{a}(\theta,\xi)=\int_{\T^d}\exp(-ix'\cdot\theta) a(x',\xi)\dx'$ is the Fourier transform of $a$ in variable $x'$. Here $\tilde{\chi}$ and $\phi$ are two given cut-off functions such that
\[
\phi(\eta)=0~~\text{for }|\eta|\leq 1,\quad \phi(\eta)=1~~\text{for }|\eta|\geq 2,
\]and $ \tilde{\chi}(\theta,\eta)$ is homogeneous of degree 0 and satisfies that for $0<\eps_1<\eps_2\ll 1$, $ \tilde{\chi}(\theta,\eta)=1$ if $|\theta|\leq \eps_1|\eta|$ and $ \tilde{\chi}(\theta,\eta)=0$ if $|\theta|\geq \eps_2|\eta|$. We also introduce the semi-norm
\begin{align}
M_r^a(a):=\sup_{|\alpha|\leq \frac{d}{2}+1+r}\sup_{|\xi|\geq 1/2}\bno{(1+|\xi|)^{|\alpha|-m}\p_\xi^\alpha a(\cdot,\xi)}_{W^{r,\infty}(\T^d)}.
\end{align}
\end{defn}

For $m\in\R$, we say $T$ is of order $m$ if for all $s\in\R$, $T$ is bounded from $H^s$ to $H^{s-m}$.
\begin{prop}
Let $m\in\R$. If $a\in \Gamma_0^m(\T^d)$, then $T_a$ is of order $m$. Moreover, for any $s\in\R$, there exists a constant $K$ such that $\|T_a\|_{H^s\to H^{s-m}}\leq K M_0^m(a)$.
\end{prop}

\begin{prop}[Composition, {\cite[Theorem 3.7]{ABZ2011IWWST}}]
Let $m\in\R$ and $r>0$. If $a\in\Gamma_r^m(\T^d),~b\in\Gamma_r^{m'}(\T^d)$, then $T_aT_b-T_{a\# b}$ is of order $m+m'-r$ where \[
a\# b:= \sum_{|\alpha|<r}\frac{1}{i^{|\alpha|}\alpha!}\p_\xi^\alpha a \p_{x'}^\alpha b.
\]Moreover, for all $s\in\R$, there exists a constant $K$ such that
\begin{align}
\|T_aT_b - T_{a\# b}\|_{H^s\to H^{s-m-m'+r}}\leq K M_r^m(a) M_r^{m'}(b).
\end{align}
\end{prop}

\begin{prop}[Adjoint, {\cite[Theorem 3.10]{ABZ2011IWWST}}]
Let $m\in\R$, $r>0$ and $a\in\Gamma_r^m(\T^d)$. We denote by $(T_a)^*$ the adjoint operator of $T_a$. Then $(T_a)^*-T_{a^*}$ is of order $m-r$ where $$a^*:=\sum_{|\alpha|<r}\frac{1}{i^{\alpha}\alpha!}\p_\xi^\alpha \p_{x'}^\alpha \bar{a}.$$ Moreover, for any $s\in\R$, there exists a constant $K$ such that $\|(T_a)^*-T_{a^*}\|_{H^s\to H^{s-m+r}}\leq KM_r^m(a)$.
\end{prop}

The symbolic calculus adopted in this paper is not of $C^{\infty}$-regularity. We shall introduce the following class of symbols. Here and thereafter in this section, $\psi\in C([0,T];H^{s+\frac12}(\T^d))$ is a given function with $s>2+\frac{d}{2}$.
\begin{defn}
Given $m\in\R$, we denote $\Sigma^m$ to be the class of symbols $a$ of the form $a=a^{(m)}+a^{(m-1)}$ with
\[
a^{(m)}(t,x',\xi)=F(\cnab_{x'}\psi(t,x'),\xi),\quad a^{(m-1)}(t,x',\xi)=\sum_{|\alpha|_2}G_\alpha(\cnab_{x'}\psi(t,x'),\xi)\p_{x'}^\alpha\psi(t,x')
\]such that 
\begin{itemize}
\item [i.] $T_a$ maps real-valued functions to real-valued functions;
\item [ii.] $F$ is a $C^\infty$ real-valued functions of $(\zeta,\xi)\in\R^d\times(\R^d\backslash\{0\})$, homogeneous of degree $m$ in $\xi$, such that there exists a continuous function $K=K(\zeta)>0$ such that $F(\zeta,\xi)\geq K(\zeta)|\xi|^m$ for all $(\zeta,\xi)\in\R^d\times(\R^d\backslash\{0\})$;
\item [iii.] $G_\alpha$ is a $C^\infty$ complex-valued function of $(\zeta,\xi)\in\R^d\times(\R^d\backslash\{0\})$, homogeneous of degree $m-1$ in $\xi$.
\end{itemize}
\end{defn} 

\begin{defn}[``Equivalence" of operators]
Given $m\in\R$ and consider two families of operators of order $m$: $\{A(t):t\in[0,T]\}$ and  $\{B(t):t\in[0,T]\}$, We say $A\sim B$ if $A-B$ is of order $m-1.5$ and satisfies the estimate: for all $r\in\R$ there exists a continuous function $C(\cdot)$ such that
\[
\forall t\in[0,T], ~~\|A(t)-B(t)\|_{H^r\to H^{r-(m-1.5)}}\leq C(|\psi(t)|_{s+\frac12}).
\]
\end{defn}  From now on, we use the notation $|\cdot|_{s_1\to s_2}$ to represent the operator norm $\|\cdot\|_{H^{s_1}\to H^{s_2}}$, and use the notation $|\cdot|_{s}$ to represent $\|\cdot\|_{H^s(\T^d)}$, as we only apply paradifferential calculus on the free interface $\Sigma$. With this definition, we have
\begin{prop}[{\cite[Prop. 4.3]{ABZ2011IWWST}}]
Let $m,m'\in\R$. Then
\begin{enumerate}
\item If $a\in\Sigma^m,~b\in\Sigma^{m'}$, then $T_aT_b\sim T_{a\# b}$ where $a\# b$ is given by 
\[
a\# b = a^{(m)}b^{(m')} + a^{(m-1)} b^{(m')} + a^{(m)}b^{(m'-1)} + \frac{1}{i} \p_\xi a^{(m)}\cdot\p_{x'} b^{(m')}.
\]
\item  If $a\in\Sigma^m$, then $(T_a)^*\sim T_b$ where $b\in \Sigma^m$ is given by
\[
b=a^{(m)}+\overline{a^{(m-1)}}+\frac{1}{i}(\p_{x'}\cdot\p_{\xi})a^{(m)}.
\]
\end{enumerate}
\end{prop}As a corollary, we have
\begin{cor}[{\cite[Prop. 4.3(2)]{ABZ2011IWWST}}]
If $a\in \Sigma^m$ satisfies $\Im a^{(m-1)}=-0.5 (\p_\xi \cdot\p_{x'})a^{(m)}$, then $(T_a)^*\sim T_a$.
\end{cor}

The next proposition is significant for the estimate of Sobolev norms via paradifferential calculus. 
\begin{prop}[{\cite[Prop. 4.4 and 4.6]{ABZ2011IWWST}}]\label{prop para Hs}
Let $m\in\R,~r\in\R$. Then for all symbol $a\in\Sigma^m$ and $t\in[0,T]$, the following estimate holds.
\begin{align}
|T_{a(t)} u|_{r-m} \leq &~C(|\psi(t)|_{s-1})|u|_{r},\\
|u|_{r+m}\leq &~ C(|\psi(t)|_{s-1})\left(|T_{a(t)} u|_{r} + |u|_0\right).
\end{align}
\end{prop}

\subsubsection{Paralinearization of the nonlinear terms}\label{sect para DtN}
Now we can start to paralinearize the term $\left(\dnp+\dnm\right)(\h(\psi))$ in \eqref{psi equ}.
\begin{lem}[Paralinearization of the Dirichlet-to-Neumann operator, {\cite[Section 4.4]{AMDtN}}]\label{prop para DtN}
For $f,\psi\in H^{s+\frac12}(\T^d)$, we have
\begin{align}\label{DtNsymbol}
\dnpm f = T_{\Lam^\pm} \psi + R_{\Lam,1}^\pm(\psi, f) + R_{\Lam,2}^\pm(\psi, f),
\end{align}with the symbols $\lam^\pm=\lam^{(1),\pm}+\lam^{(0),\pm}$ give by
\begin{align}
\Lam^{(1),\pm}=&\sqrt{(1+|\cnab\psi|^2)|\xi|^2 - (\cnab\psi\cdot\xi)^2},\\
\Lam^{(0),-}=-\overline{\Lam^{(0),+}}=&\frac{1+|\cnab\psi|^2|}{2\Lam^{(1),-}}\left(\cnab\cdot(\alpha^{(1)}\cnab\psi)+i\p_\xi \Lam^{(1),-}\cdot\cnab\alpha^{(1)}\right),
\end{align}and $\alpha^{(1)}:=(\Lam^{(1),-}+i\cnab\psi\cdot \xi)/(1+|\cnab\psi|^2)$. The remainder terms satisfy the following estimates
\begin{align}
| R_{\Lam,1}^\pm (\psi, f)|_{s-\frac12} \leq C(|\psi|_{C^2},|f|_3)|f|_{s+\frac12},\quad |R_{\Lam,2}^\pm (\psi, f)|_{s-\frac12} \leq C(|\psi|_{s-\frac12})|\TP f|_{s-2}.
\end{align}
\end{lem}

\begin{lem}[Paralinearization of the mean curvature operator, {\cite[Lemma 3.25]{ABZ2011IWWST}}]\label{prop para ST}
There holds $\h(\psi)=-T_\hh f+R_\hh$ where $\hh=\hh^{(2)}+\hh^{(1)}$ is defined by
\begin{align}
\hh^{(2)}=&\frac{1}{\sqrt{1+|\cnab\psi|^2}}\left(|\xi|^2-\frac{(\cnab\psi\cdot\xi)^2}{1+|\cnab\psi|^2}\right),\\
\hh^{(1)}=&-\frac{i}{2}(\cnab_{x'}\cdot\p_\xi)\hh^{(2)},
\end{align}and the remainder term $R_\hh$ satisfies
\begin{align}
|R_\hh|_{2s-3}\leq C(|\psi|_{s+\frac12}).
\end{align}
\end{lem}

With the paralinearization of operators $\dnpm$ and $\h(\psi)$, the term $\left(\dnp+\dnm\right)(\h(\psi))$ in \eqref{psi equ} becomes 
\begin{align}\label{psi equ main}
\sigma\left(\dnp+\dnm\right)(\h(\psi))=-\sigma T_{\Lam}T_\hh \psi +\sigma \RR_\psi^{\sigma},
\end{align}where $-\overline{\Lam^{(0),+}}=\Lam^{(0),-}$ shows that $\Re(\Lam^{(0),+})+\Re(\Lam^{(0),-})=0$,~$\Im(\Lam^{(0),+})=\Im(\Lam^{(0),+})$ and thus
\begin{align}
\Lam :=&~\underbrace{(\Lam^{(1),+} + \Lam^{(1),-})}_{=:\Lam^{(1)}} + \underbrace{(\Lam^{(0),+} + \Lam^{(0),-})}_{=:\Lam^{(0)}}=2\Lam^{(1),-}+2i\Im(\Lam^{(0),-})\\
\label{para RR}\RR_\psi^{\sigma}:=&\sum_\pm T_{\Lam^\pm} R_\hh+ R_{\Lam,1}^\pm(\psi,\h(\psi)) + R_{\Lam,2}^\pm(\psi, \h(\psi)) \text{ and }|\RR_\psi^{\sigma}|_{s-\frac12} \leq C(|\psi|_{s+\frac12})|\psi|_{s+1}.
\end{align}

In order for an explicit energy estimate for $\psi$ and $\psi_t$, we shall symmetrize the 3-rd order paradifferential operator $T_{\Lam}T_\hh$. That is, find suitable symbols $\fm \in \Sigma^{1.5}$ and $\fn\in \Sigma^0$ such that $T_\fn T_\Lam T_\hh \sim T_\fm T_\fm T_\fn$ and $T_\fm \sim (T_\fm)^*$.

\begin{prop}[Symmetrisation of the composition]\label{prop para symm}
Let $\fn\in\Sigma^0$ and $\fm\in \Sigma^{1.5}$ be defined by 
\begin{align}
\label{symbol n} \fn:=&~\frac{1}{\sqrt[3]{2}\sqrt[4]{1+|\cnab\psi|^2}}=2^{-\frac13}|N|^{-\frac12},\\
\label{symbol m} \fm :=&~ \underbrace{\sqrt{\hh^{(2)}\Lam^{(1)}}}_{=:\fm^{(1.5)}}+\underbrace{\frac{1}{2i}(\p_\xi\cdot\p_{x'})\sqrt{\hh^{(2)}\Lam^{(1)}}}_{=:\fm^{(0.5)}}.
\end{align} Then $T_\fn T_\lam T_\hh \sim T_\fm T_\fm T_\fn$ and $T_\fm \sim (T_\fm)^*$ are both fulfilled.
\end{prop}
\begin{proof}
Given the symbol $\Lam$ and $\hh$, we shall find suitable symbols $\fn\in\Sigma^0,\fm\in\Sigma^{1.5}$ such that $\fn(x',\xi)$ is  independent of $\xi$ and $\fn\#(\Lam\# \hh)=(\fm\#\fm)\# \fn$, i.e.,
\begin{align*}
&~\fn^{(0)}(\Lam\# \hh)+\fn^{(-1)}\Lam^{(1)}\hh^{(2)}+\frac{1}{i}\p_\xi \fn^{(0)}\cdot\p_{x'}(\Lam^{(1)}\hh^{(2)})\\
=&~(\fm\#\fm)\fn^{(0)} + (\fm^{(1.5)})^2\fn^{(-1)} +\frac{1}{i}\p_\xi((\fm^{(1.5)})^2)\cdot\p_{x'}\fn^{(0)}.
\end{align*}Recall that
\begin{align*}
(\Lam\# \hh)=&~\Lam^{(1)}\hh^{(2)}+\Lam^{(0)}\hh^{(2)}+\Lam^{(1)}\hh^{(1)}+\frac{1}{i}\p_\xi \hh^{(2)}\cdot\p_{x'}\Lam^{(1)},\\
(\fm\#\fm)=&~(\fm^{(1.5)})^2+2(\fm^{(1.5)})(\fm^{(0.5)})+\frac{1}{i}\p_\xi(\fm^{(1.5)})\cdot\p_{x'}(\fm^{(1.5)}).
\end{align*}We choose the principal symbol $\fm^{(1.5)}:=\sqrt{\Lam^{(1)}\hh^{(2)}}$ in order for cancelling the leading-order symbols. Since we require $(T_{\fm})^*\sim T_{\fm^*}$, we must have $\Im (\fm^{(0.5)})=-0.5(\p_{x'}\cdot\p_\xi)\fm^{(1.5)}$ (cf. \cite[Prop. 4.3]{ABZ2011IWWST}). With this choice for $\fm$, it remains to solve the symbolic equation
\begin{equation}
\fn^{(0)}(\Lam\# \hh - \fm\#\fm)=\frac{1}{i}\p_\xi((\Lam^{(1)}\hh^{(2)})\cdot\p_{x'}\fn^{(0)}-\frac{1}{i}\p_{x'}(\Lam^{(1)}\hh^{(2)})\cdot\p_{\xi}\fn^{(0)},
\end{equation}with
\begin{align*}
\Lam\# \hh - \fm\#\fm=\Lam^{(0)}\hh^{(2)}+\Lam^{(1)}\hh^{(1)} - 2(\fm^{(1.5)})(\fm^{(0.5)}) + \frac{1}{i}\p_\xi \hh^{(2)}\cdot\p_{x'}\Lam^{(1)} - \frac{1}{i}\p_\xi(\fm^{(1.5)})\cdot\p_{x'}(\fm^{(1.5)}).
\end{align*} The sub-principle $\fn^{(-1)}$ does not appear, so we can choose $\fn^{(-1)}=0$. Since the principal symbols of $\Lam$ and $\hh$ are real-valued, we now just need to solve
\[
\Re(\Lam\# \hh - \fm\#\fm)=0,~~\fn^{(0)}\Im(\Lam\# \hh - \fm\#\fm)=-\p_\xi((\Lam^{(1)}\hh^{(2)})\cdot\p_{x'}\fn^{(0)}+\p_{x'}(\Lam^{(1)}\hh^{(2)})\cdot\p_{\xi}\fn^{(0)}.
\] The condition for the real part is fulfilled if we have
\[
\underbrace{\Re(\Lam^{(0)})}_{=0}\hh^{(2)}=2\fm^{(1.5)}\Re(\fm^{(0.5)})\Rightarrow \Re (\fm^{(0.5)}) = 0.
\]For the imaginary part, inserting the symbols $\hh^{(1)},~\Im(\Lam^{(0)}),~\fm^{(1.5)}$ and $\Im(\fm^{(0.5)})$, we get
\begin{align*}
\Im(\Lam\# \hh - \fm\#\fm)=\frac12\p_{\xi}\hh^{(2)}\cdot\p_{x'} \Lam^{(1)} -\frac12\p_{x'}\hh^{(2)}\cdot\p_{\xi} \Lam^{(1)} ,
\end{align*}and thus we need to solve
\begin{align}
\fn^{(0)}\left(\frac12\p_{\xi}\hh^{(2)}\cdot\p_{x'} \Lam^{(1)} -\frac12\p_{x'}\hh^{(2)}\cdot\p_{\xi} \Lam^{(1)}\right) = -\p_\xi((\Lam^{(1)}\hh^{(2)})\cdot\p_{x'}\fn^{(0)}+\p_{x'}(\Lam^{(1)}\hh^{(2)})\cdot\p_{\xi}\fn^{(0)}.
\end{align} Notice that $\hh^{(2)}=(c\Lam^{(1)})^2$ with $c=\frac12 (1+|\cnab\psi|^2)^{-\frac34}$. Plugging it to the above equation, after a long and tedious calculation, we get the following relation
\[
\fn^{(0)}\left(c^2\p_{x'}\Lam^{(1)}(\Lam^{(1)}) - (\p_{x'}c) c (\Lam^{(1)})^2 - c^2\p_{x'}(\Lam^{(1)})\Lam^{(1)}\right) = -3c^2(\Lam^{(1)})^2\p_{x'}\fn^{(0)},
\]that is,
\[
\frac{\p_{x'} \fn^{(0)}}{\fn^{(0)}} = \frac{-(\p_{x'} c) c(\Lam^{(1)})^2}{-3c^2(\Lam^{(1)})^2}=\frac13\frac{\p_{x'} c}{c}\Rightarrow \fn^{(0)}=c^{\frac13}=2^{-\frac13}(1+|\cnab\psi|^2)^{-\frac14}.
\]
\end{proof}

We expect to take $(s-\frac12)$-th order derivatives in \eqref{psi equ}. In view of the paradifferential formulation, we shall alternatively take $T_{\MM}$ with 
\begin{align}
\label{symbol M}\MM:=(\fm^{(1.5)})^{\frac{2s-1}{3}}=2^{\frac{2s-1}{6}}|\xi|^{s-\frac12}\left(1-\left|\frac{N}{|N|}\cdot\frac{\xi}{|\xi|}\right|^2\right)^{\frac{s}{2}-\frac14}\in\Sigma^{s-\frac12}
\end{align} for sake of simplicity. Below, we list several commutator estimates for the paradifferential operators. 
\begin{lem}\label{lem paracomm}
For any $r\in\R,~s>2+\frac{d}{2}$, any functions $a$ and $f$, the following commutator estimates hold
\begin{align*}
\bno{\left[T_\MM,T_\fm\right]}_{r+s-1\to r} \leq&~ C\left(|\psi|_{s+0.5}\right),\\
\bno{[T_\fn,T_a]f}_{s-\frac12}\leq&~ C(|\cnab\psi|_{W^{1,\infty}})|a|_{W^{1,\infty}}|f|_{s-1.5},\\
\bno{\left[T_\MM,a\right]T_\fn f}_{0}+\bno{T_\MM\left[T_\fn,a\right]f}_{0}\leq&~ C(|\cnab\psi|_{W^{1,\infty}})|a|_{s-0.5}|f|_{s-1.5},\\
\bno{[T_\fm,a]f}_{0}\leq&~ C(|\cnab\psi|_{W^{1,\infty}})|a|_{s-0.5}|f|_{0.5}.
\end{align*}
\end{lem}

We also need to commute $\p_t$ with paradifferential operators. These steps will generate paradifferential operators whose symbols are spatial or time derivatives.
\begin{lem}\label{lem paracomm t}
For any $r\in\R$, the following estimates hold
\begin{align*}
|T_{\p_t\fn}|_{r\to r}+|T_{\p_t\MM}|_{r\to r-(s-0.5)}\leq&~C(|\psi|_{W^{1,\infty}},|\p_t\psi|_{W^{1,\infty}}),\\
|T_{\p_t^2\fn}|_{r\to r} + |T_{\p_t^2\MM}|_{r \to r-(s-0.5)} \leq&~C\left(\left|\cnab\psi,\p_t\psi,\p_t^2\psi\right|_{W^{1,\infty}}\right),\\
|T_{\TP \fn}|_{r\to r} + |T_{\TP \MM}|_{r\to r-(s-0.5)}  \leq&~ C(|\cnab\psi|_{W^{1,\infty}}),\\
|T_{\p_t\fm}|_{r\to r-1.5}+ |T_{\TP \fm}|_{r\to r-1.5}\leq&~C(|\cnab\psi|_{W^{1,\infty}},|\cnab\p_t\psi|_{W^{1,\infty}}).
\end{align*}
\end{lem}

\subsection{Uniform estimates for the free interface}\label{sect uniform psi eq}
With the symmetrized paralinearization of $(\dnp+\dnm)\h(\psi)$ derived in Section \ref{sect para pre}, we can now prove the uniform-in-$(\es)$ estimates of $\psi$ under the stability condition \eqref{syrov 3D com}. Our desired result is listed in the following proposition.
\begin{prop}
For a fixed $\sigma>0$, there is a time $T_\sigma>0$ (independent of $\eps$) such that the following uniform-in-$\eps$ estimates hold true
\begin{align}\label{psitt 2.5-1}
\ddt\left(\bno{\sqrt{\sigma}\psi}_5^2+\bno{\psi_{t}}_{3.5}^2+\bno{\sqrt{\sigma}\psi_t}_4^2+\bno{\psi_{tt}}_{2.5}^2\right)\lesssim  \sigma^{-1} P(\EE_4(t))E_5(t).
\end{align}Moreover, if the stability condition \eqref{syrov 3D com} is fulfilled, then there exists a time $T>0$ (independent of $\es$) such that the following uniform-in-$(\es)$ estimates hold true
\begin{align}\label{psitt 2.5-2}
\ddt\left(\bno{\sqrt{\sigma}\psi}_5^2+\bno{\psi_{t}}_{3.5}^2+\bno{\sqrt{\sigma}\psi_t}_4^2+\bno{\psi_{tt}}_{2.5}^2\right)\lesssim P(\EEW_4(t))\EW_5(t).
\end{align}
\end{prop}

\subsubsection*{Step 1: Find equivalent energy functionals}
The equation \eqref{psi equ} can be written as
\begin{align}
(\rho^++\rho^-)\p_t^2\psi=&~-\frac{\sigma}{2}T_\Lam T_\hh \psi -(\rho^+ + \rho^-)(\fwc_i\fwc_j)\TP_i\TP_j\psi -2(\rho^++\rho^-)\fwc_i\TP_i\p_t\psi  \no\\
&+\left(\rho^+(\fbc_i^+\fbc_j^+ - \fuc_i\fuc_j) + \rho^-(\fbc_i^-\fbc_j^- - \fuc_i\fuc_j)\right)\TP_i\TP_j\psi\no\\
& -( N\cdot\nabp q_w^+ + N\cdot\nabp q_w^-) + \Psi^R\label{psi equ 1},
\end{align}where $T_\Lam,T_\hh$ are the paradifferential operators defined in Proposition \ref{prop para DtN} and Proposition \ref{prop para ST}, the quantities $\fw,\fu,\fb$ are defined by
\begin{align}
\fw:=\frac{\rho^+ v^+ + \rho^- v^-}{\rho^+ + \rho^-},\quad \fu:=\frac{\sqrt{\rho^+\rho^- }}{\rho^+ + \rho^-}\jump{v},\quad \fb^\pm:=\frac{b^\pm}{\sqrt{\rho^\pm}},
\end{align}and $\Psi^R$ is defined by
\begin{align}\label{PSIR}
\Psi^R:=\frac{\sigma}{2}(\dnp-\dnm)\dnw^{-1}(\dnp-\dnm)(\h(\psi)) -(\dnp-\dnm)\dnw^{-1}\left(\jump{\FP-\rho\p_t^2\psi}\right)+\frac{\sigma}{2}\RR_\psi^\sigma.
\end{align}Note that $\RR_\psi^\sigma$ has been defined in \eqref{para RR}.

We pick $s=4$ in the paradifferential operator $T_\MM$, that is, $\MM=(\fm^{(1.5)})^{\frac{7}{3}}\in\Sigma^{3.5}$ and then consider the energy functionals
\begin{align}
\ee(t):=&~\frac12\is(\rho^++\rho^-)\bno{(\p_t+\fwc\cdot\cnab)T_\MM T_\fn \psi}^2\dx' + \frac14\is\bno{\sqrt{\sigma}T_\fm T_\MM T_\fn \psi}_0^2\dx', \\
\eew(t):=&~\frac12\is \rho^+\left(\bno{\fbc^+\cdot\cnab T_\MM T_\fn \psi}^2 - \bno{\fuc\cdot\cnab T_\MM T_\fn \psi}^2\right) +\rho^-\left(\bno{\fbc^-\cdot\cnab T_\MM T_\fn \psi}^2 - \bno{\fuc\cdot\cnab T_\MM T_\fn \psi}^2\right)\dx'.
\end{align}

\begin{lem}[Comparison between $\ee,\eew$ and Sobolev norms]\label{lem para norm}
For any fixed $\sigma>0$, we have the following relations between $\ee,\eew$ and standard Sobolev norms.
\begin{align*}
\bno{\sqrt{\sigma}\psi}_5^2\leq&~ C(|\cnab\psi|_{W^{1,\infty}})\left( \ee(t)+|\sqrt{\sigma}\psi|_0^2\right),\\
|\psi|_{4.5}^2\lesssim&~  |\sqrt{\sigma}\psi|_5^2 + \sigma^{-1}|\psi|_0^2,\\
|\psi_t|_{3.5}^2\leq &~C(|\cnab\psi,\cnab\psi_t,\vb^\pm,\rho^\pm|_{W^{1,\infty}})\left(\ee(t)+|\psi|_{4.5}^2 + |\psi_t|_0^2\right)
\end{align*}where $C(\cdot)$ represents a generic positive continuous function in its arguments. Moreover, when the stability condition \eqref{syrov 3D com} holds, there exist positive continuous functions $C_1,C_1',C_1''$ depending on $|\cnab\psi,\vb^\pm,\bc^\pm,\rho^\pm|_{W^{1,\infty}}$ and independent of $\sigma$, such that
\begin{align*}
C_1\left(|\cnab\psi,\vb^\pm,\bc^\pm,\rho^\pm|_{W^{1,\infty}}\right)\bno{\psi}_{4.5}^2\leq \eew(t) + C_1'|\psi|_0^2,~~ \eew(t)\leq  C_1''\left(|\cnab\psi,\vb^\pm,\bc^\pm,\rho^\pm|_{W^{1,\infty}}\right)\bno{\psi}_{4.5}^2.
\end{align*}
\end{lem}
\begin{proof}
Recall that $T_\MM$ and $T_\fm$ are paradifferential operators of order 3.5 and 1.5 respectively, thus the first inequality is a direct consequence of Proposition \ref{prop para Hs}. The second inequality is a directly consequence of Sobolev interpolation and Young's inequality
\[
|\psi|_{4.5}^2\leq |\sqrt{\sigma}\psi|_5^{1.8}|\sigma^{-\frac12}\psi|_0^{0.2}\leq \frac{|\sqrt{\sigma}\psi|_5^2}{10/9} + \frac{|\psi|_0^2}{10\sigma}.
\]
To prove the third inequality, we again use Proposition \ref{prop para Hs} to get
\begin{align*}
|\psi_t|_{3.5}^2\leq&~ C(|\cnab\psi|_{W^{1,\infty}})\left(|T_\MM T_\fn\psi_t|_0^2 + |\psi_t|_0^2\right)\\
\leq&~C(|\cnab\psi,\vb^\pm,\rho^\pm|_{W^{1,\infty}})\left(\ee(t) + |T_\MM T_{\p_t\fn} \psi|_0^2 + |T_{\p_t\MM} T_\fn \psi|_0^2 + |\psi_t|_0^2\right)\\
\leq&~C(|\cnab\psi,\cnab\psi_t,\vb^\pm,\rho^\pm|_{W^{1,\infty}})\left(\ee(t) + |\psi|_{3.5}^2 + |\psi_t|_0^2\right).
\end{align*}

For the last inequality in this lemma, the right side is trivial. When the stability condition \eqref{syrov 3D com} holds, it suffices to prove that $\eew(t)$ is a positive-definite energy, then the left side automatically holds. Multiplying $(\rho^+\rho^-)^{-\frac12}$ in \eqref{syrov 3D com}, the stability condition becomes
\begin{align}
\exists \delta_0\in(0,\frac18),\quad |\fbc^+\times\fbc^-|\geq (1-\delta_0)^{-1}|\fbc^\pm\times \jump{\vb}|\sqrt{1+(c_A^\pm/c_s^\pm)^2}>(1-\delta_0)^{-1}|\fbc^\pm\times \jump{\vb}|.
\end{align}Since $\jump{v}\cdot N=0$ and $\fb^\pm$ are nonzero and not collinear, we may assume $\jump{v}=c_1\fb^++c_2\fb^-$. Plugging this into the stability condition, we get $c_1,c_2\leq 1-\delta_0$. Using Cauchy-Schwarz inequality, we derive that
\begin{align*}
\inf_{\substack{\bd{z}\in\R^2\\|\bd{z}|=1}} (1-\delta_0)^2\left((\fbc^+\cdot\bd{z})^2 + 2(\fbc^+\cdot\bd{z})(\fbc^-\cdot\bd{z}) + (\fbc^-\cdot\bd{z})^2 \right) - \left(\jump{\vb}\cdot\bd{z}\right)^2\geq 0
\end{align*} Invoking $\fuc=\frac{\sqrt{\rho^+\rho^-}}{\rho^++\rho^-}\jump{\vb}$ and using the non-collinearity, the above inequality implies that
\begin{align*}
&~\inf_{\substack{\bd{z}\in\R^2\\|\bd{z}|=1}} \left(\frac{\rho^+\rho^-}{\rho^++\rho^-}(\fbc^+\cdot\bd{z})^2 + 2\frac{\rho^+\rho^-}{\rho^++\rho^-}(\fbc^+\cdot\bd{z})(\fbc^-\cdot\bd{z}) +\frac{\rho^+\rho^-}{\rho^++\rho^-} (\fbc^-\cdot\bd{z})^2  - (\rho^++\rho^-) \left(\fuc\cdot\bd{z}\right)^2\right) >0.
\end{align*}Notice that
\begin{align*}
&\rho^+(\fbc^+\cdot\bd{z})^2+\rho^-(\fbc^-\cdot\bd{z})^2 - \left(\frac{\rho^+\rho^-}{\rho^++\rho^-}(\fbc^+\cdot\bd{z})^2 + 2\frac{\rho^+\rho^-}{\rho^++\rho^-}(\fbc^+\cdot\bd{z})(\fbc^-\cdot\bd{z}) +\frac{\rho^+\rho^-}{\rho^++\rho^-} (\fbc^-\cdot\bd{z})^2\right) \\
=&~\frac{1}{\rho^++\rho^-}\left((\rho^+)^2(\fbc^+\cdot\bd{z})^2 - 2\rho^+\rho^-(\fbc^+\cdot\bd{z})(\fbc^-\cdot\bd{z}) +(\rho^-)^2(\fbc^-\cdot\bd{z})^2\right)\geq 0.
\end{align*}Thus, it implies that
\begin{align}
&~\inf_{\substack{\bd{z}\in\R^2\\|\bd{z}|=1}} \left(\rho^+(\fbc^+\cdot\bd{z})^2+\rho^-(\fbc^-\cdot\bd{z})^2 - (\rho^++\rho^-) \left(\fuc\cdot\bd{z}\right)^2\right) >0,
\end{align}or equivalently, there exists some $\delta_0'>0$ such that
\begin{align}
&~\inf_{\bd{z}\in\R^2} \left(\rho^+(\fbc^+\cdot\bd{z})^2+\rho^-(\fbc^-\cdot\bd{z})^2 - (\rho^++\rho^-) \left(\fuc\cdot\bd{z}\right)^2\right)\geq 2\delta_0'|\bd{z}|^2.
\end{align}Now let $\bd{z}=\cnab T_\MM T_\fn \psi$, the above inequality shows that $\eew(t)\geq \delta_0'|\cnab T_\MM T_\fn \psi|_0^2\gtrsim |\psi|_{4.5}^2 - |\psi|_0^2$.
\end{proof}

\begin{rmk}[The 2D case]
When the space dimension $d=2$, we no longer have the non-collinearity, but the stability condition \eqref{syrov 2D com} still guarantees the ellipticity of the corresponding second-order differential operator, i.e., 
$$\exists \delta_0'>0,\quad \rho^+\left((\bd{b}_1^+)^2-\bd{u}_1^2\right)+\rho^-\left((\bd{b}_1^-)^2-\bd{u}_1^2\right)\geq \delta_0'.$$ In fact, the stability condition \eqref{syrov 2D com} implies $|\bd{b}_1^+|+|\bd{b}_1^-|\geq (1+\delta_0)|\jump{v_1}|$. Taking square and invoking $\fu:=\frac{\sqrt{\rho^+\rho^- }}{\rho^+ + \rho^-}\jump{v}$, we get
\[
\frac{\rho^+\rho^-}{\rho^++\rho^-}\left((\bd{b}_1^+)^2+2\bd{b}_1^+\bd{b}_1^-+(\bd{b}_1^-)^2\right)\geq (1+\delta_0)(\rho^++\rho^-)\bd{u}_1^2,
\]in which we find that the left side does not exceed $(\rho^+ + \rho^-)^{-1}\left(\rho^+(\bd{b}_1^+)^2+\rho^-(\bd{b}_1^-)^2\right)$ by direct calculation. The desired result immediately follows thanks to $|\bd{u}_1|>0$ (otherwise the interface is not a vortex sheet).
\end{rmk}

\subsubsection*{Step 2: Control of $|\psi_t|_{3.5}$ and $|\sqrt{\sigma}\psi|_5$}
In view of Lemma \ref{lem para norm}, it suffices to prove energy estimates for $\ee(t)$ and $\eew(t)$ under the stability condition \eqref{syrov 3D com}. We start with the estimate of $|\psi_t|_{3.5}$.
\begin{align}
&\ddt\frac12\is(\rho^++\rho^-)\bno{(\p_t+\fwc\cdot\cnab)T_\MM T_\fn \psi}^2\dx'\no\\
=&\is(\rho^++\rho^-)(\p_t^2 T_\MM T_\fn \psi)\left((\p_t+\fwc\cdot\cnab)T_\MM T_\fn \psi\right) \dx' + \is (\rho^++\rho^-) (\fwc\cdot\cnab\p_t T_\MM T_\fn \psi)\left((\p_t+\fwc\cdot\cnab)T_\MM T_\fn \psi\right) \dx'\no\\
&+\is(\rho^++\rho^-) (\p_t\fwc\cdot\cnab T_\MM T_\fn \psi)\left((\p_t+\fwc\cdot\cnab)T_\MM T_\fn \psi\right) \dx' +\frac12\is\p_t(\rho^++\rho^-) \bno{(\p_t+\fwc\cdot\cnab)T_\MM T_\fn \psi}^2\dx'\no\\
=:&~I_0+I_1+I_1^R+I_2^R.
\end{align}
The remainder terms are easy to control. Using Proposition \ref{prop para Hs}, we have
\begin{align}
I_1^R+I_2^R\leq C\left(\left|\rho^\pm,\p_t\rho^\pm,\vb^\pm,\p_t\vb^\pm\right|_{L^{\infty}},|\psi|_3\right)\left(|\psi_t|_{3.5}^2+|\psi|_{4.5}^2\right).
\end{align}
For the main term $I_0$, we first commute $(\rho^+ + \rho^-)\p_t^2$ with $T_\MM T_\fn$
\begin{align*}
(\rho^++\rho^-)\p_t^2 T_\MM T_\fn \psi = &~(\rho^++\rho^-) \left(T_\MM T_\fn \p_t^2\psi+ (T_{\p_t^2\MM}T_\fn + T_{\MM}T_{\p_t^2\fn})\psi + 2(T_{\p_t\MM}T_{\fn}+T_\MM T_{\p_t\fn})\p_t\psi\right)\\
=&~T_\MM T_\fn\left((\rho^++\rho^-)\p_t^2\psi\right) - \left([T_\MM,\rho^+ + \rho^-]T_\fn \p_t^2\psi + T_\MM([T_\fn,\rho^+ + \rho^-]\p_t^2\psi)\right)\\
&+(\rho^++\rho^-) \left((T_{\p_t^2\MM}T_\fn + T_{\MM}T_{\p_t^2\fn})\psi + 2(T_{\p_t\MM}T_{\fn}+T_\MM T_{\p_t\fn})\p_t\psi\right).
\end{align*} The commutators can be controlled straightforwardly thanks to Lemma \ref{lem paracomm} and Lemma \ref{lem paracomm t}:
\begin{align*}
\bno{[T_\MM,\rho^+ + \rho^-]T_\fn \p_t^2\psi}_0+\bno{T_\MM([T_\fn,\rho^+ + \rho^-]\p_t^2\psi)}_0\leq C(|\cnab\psi|_{W^{1,\infty}})\left(|\rho^+ + \rho^-|_{3.5}|\p_t^2\psi|_{2.5}\right),\\
\bno{(T_{\p_t^2\MM}T_\fn + T_{\MM}T_{\p_t^2\fn})\psi}_0+\bno{(T_{\p_t\MM}T_{\fn}+T_\MM T_{\p_t\fn})\p_t\psi}_0\leq C(|\psi_{tt},\psi_t,\cnab\psi|_{W^{1,\infty}})\left(|\psi|_{3.5}+|\psi_t|_{3.5}\right).
\end{align*} In the remaining of this section, we no longer explicitly write the commutators between the paradifferential operators and functions or $\p_t,\TP_i$, as they can be controlled in the same way as above. Instead, we will again use the notation $\eql$ to skip these terms and analyze the main terms.

Then we can invoke \eqref{psi equ 1} to get
\begin{align}
I_{00}:=&\is\left(T_\MM T_\fn ((\rho^++\rho^-)\p_t^2 \psi)\right)\left((\p_t+\fwc\cdot\cnab)T_\MM T_\fn \psi\right) \dx' \no\\
=&-\frac{\sigma}{2}\is \left(T_\MM T_\fn T_\Lam T_\hh \psi \right)\left((\p_t+\fwc\cdot\cnab)T_\MM T_\fn \psi\right) \dx' \no\\
&-2\is\left(T_\MM T_\fn((\rho^++\rho^-)\fwc_i\TP_i\p_t\psi) \right)\left((\p_t+\fwc\cdot\cnab)T_\MM T_\fn \psi\right) \dx'\no\\
&-\is \left(T_\MM T_\fn((\rho^++\rho^-)\fwc_i\fwc_j\TP_i\TP_j\psi) \right)\left((\p_t+\fwc\cdot\cnab)T_\MM T_\fn \psi\right) \dx'\no\\
&+\is\left(T_\MM T_\fn\left(\rho^+(\fbc_i^+\fbc_j^+ - \fuc_i\fuc_j) + \rho^-(\fbc_i^-\fbc_j^- - \fuc_i\fuc_j)\right)\TP_i\TP_j\psi \right)\left((\p_t+\fwc\cdot\cnab)T_\MM T_\fn \psi\right) \dx'\no\\
&-\is\left(T_\MM T_\fn\left(N\cdot\nabp q_w^+ + N\cdot q_w^-\right)\right)\left((\p_t+\fwc\cdot\cnab)T_\MM T_\fn \psi\right) \dx'\no\\
&+\is\left(T_\MM T_\fn\Psi^R\right)\left((\p_t+\fwc\cdot\cnab)T_\MM T_\fn \psi\right) \dx'=:I_{000}+I_{001}+I_{002}+I_{003}+I_{0}^w+I_0^R.
\end{align} Since Proposition \ref{prop para symm} indicates that $T_\fn T_\Lam T_\hh\sim T_\fm T_\fm T_\fn$ and $(T_\fm)^*\sim T_\fm$, we have
\begin{align*}
&~T_\MM T_\fn T_\Lam T_\hh\psi \eql T_\MM T_\fm T_\fm T_\fn\psi\\
= &~(T_\fm)^*T_\fm T_\MM T_\fn \psi + \big(((T_\fm)^*-T_\fm)T_\MM +T_\fm[T_\MM,T_\fm] + [T_\MM,T_\fm]T_\fm\big)T_\fn \psi \eql (T_\fm)^*T_\fm T_\MM T_\fn \psi 
\end{align*} and using the duality, we get
\begin{align}
I_{000}\eql& -\frac{\sigma}{2}\is (T_\fm T_{\MM} T_\fn \psi)\, T_\fm(\p_t+\fwc\cdot\cnab)T_\MM T_\fn \psi\dx'\no\\
=&-\frac{\sigma}{4}\ddt \is \bno{T_\fm T_{\MM} T_\fn \psi}^2\dx'\no\\
&+\frac{\sigma}{2}\is  (T_\fm T_{\MM} T_\fn \psi)\, \left(T_{\p_t\fm} + \fwc_i T_{\TP_i\fm} +\frac12(\cnab\cdot\fwc) -[T_{\fm},\fwc_i]\TP_i\right)T_\MM T_\fn\psi\dx',
\end{align}where the second term can be directly controlled (uniformly in $\sigma$) by $\ee(t)C\left(|\cnab\psi,\psi_t,\vb^\pm,\rho^\pm|_{W^{1,\infty}}\right)$ thanks to Lemma \ref{lem paracomm}. Next we analyze $I_1+I_{001}$ and $I_{002},I_{003}$. For a generic function $\bd{a}\in H^{3.5}(\Sigma\to\R^2)$ and a generic $\bm{\rho}\in H^{3.5}(\Sigma\to\R_+)$, we have
\begin{align*}
&\is \left(T_\MM T_\fn (\bm{\rho} \bd{a}_i\bd{a}_j\TP_i\TP_j \psi)\right)\,\left((\p_t+\fwc\cdot\cnab)T_\MM T_\fn \psi\right) \dx'\\
\eql&\is\bm{\rho}\bd{a}_i\bd{a}_j\TP_i\TP_j T_\MM T_\fn \psi\,\left((\p_t+\fwc\cdot\cnab)T_\MM T_\fn \psi\right) \dx'\eql-\is \bm{\rho}(\bd{a}_i\TP_iT_\MM T_\fn \psi)\,(\bd{a}_j\TP_j(\p_t+\fwc\cdot\cnab)T_\MM T_\fn \psi)\dx'\\
\eql&-\frac12\ddt\is \bm{\rho}\bno{\bd{a}_i\TP_iT_\MM T_\fn \psi}^2\dx' .
\end{align*}Setting $\bd{a}=\fwc,\fuc,\fbc$ and $\bm{\rho}=\rho^\pm$ or $\rho^++\rho^-$, we immediately get $$I_{002}\eql +\frac12\ddt \is (\rho^++\rho^-)\bno{(\fwc\cdot\cnab)T_\MM T_\fn\psi}^2\dx,\quad I_{003}\eql -\ddt\eew(t).$$ For $I_1+I_{001}$, we have
\begin{align}
I_{001}\eql &-2\is (\rho^++\rho^-)(\fwc\cdot\cnab)\p_t T_{\fm}T_\fn\psi\,\left((\p_t+\fwc\cdot\cnab)T_\MM T_\fn \psi\right)\dx'\no\\
\Rightarrow I_1+I_{001}\eql &-\is (\rho^++\rho^-)(\fwc\cdot\cnab)\p_t T_{\fm}T_\fn\psi\,\left((\p_t+\fwc\cdot\cnab)T_\MM T_\fn \psi\right)\dx'\no\\
\eql&-\frac12\ddt \is (\rho^++\rho^-)\bno{(\fwc\cdot\cnab)T_\MM T_\fn\psi}^2\dx,
\end{align}which cancels with the main term in $I_{002}$. When $\sigma>0$ is given and the stability condition \eqref{syrov 3D com} is not assumed, the quantity $\eew(t)$ is not necessarily positive, but its contribution, namely the term $I_{003}$, can be controlled by integrating by parts for 1/2-derivative
\[
I_{003}\lesssim P(\| v^\pm,\rho^\pm,b^\pm\|_{4,\pm})(|\psi|_5^2+|\psi|_5|\psi_t|_{4})\lesssim \sigma^{-1}P(\EE_4(t)).
\] 

Now, it remains to control $I_{0}^w$ and $I_0^R$. In view of Proposition \ref{prop para Hs}, it suffices to control the $H^{3.5}(\Sigma)$ norms of $N\cdot\nabp q_w^\pm$ and $\Psi^R$. For the term $q_w^\pm$, we use trace lemma and div-curl inequality with tangential trace (see \eqref{divcurltt}) to get
\begin{align}
&|N\cdot \nabp q_w^\pm|_{3.5}^2\leq |\psi|_{4.5}^2\|\nabp q_w^\pm\|_{4,\pm}^2 \no\\
\leq&~ C(|\psi|_{4.5})\left(\|\nabp q_w^\pm\|_{0,\pm}^2+\|\lapp q_w^\pm\|_{3,\pm}^2 +\|\nabp\times\nabp q_w^\pm\|_{3,\pm}^2+\bno{N\times\nabp q_w^\pm}_{3.5}^2+\bno{N\cdot\nabp q_w^\pm}_{H^{3,5}(\Sigma^\pm)}^2\right)
\end{align}where the last three terms are zero because of 
\[
\nabp\times\nabp q_w^\pm=\vec{0},\quad N\cdot\nabp q_w^\pm|_{\Sigma^\pm}=0,\quad q_w^\pm|_{\Sigma}=0\Rightarrow N\times\nabp q_w^\pm|_{\Sigma}=(-\TP_2 q_w^\pm,\TP_1 q_w^\pm,\TP_2\psi\TP_1 q_w^\pm-\TP_1\psi\TP_2 q_w^+)^\top|_{\Sigma}=\vec{0}.
\] Then invoking the definition of $q_w^\pm$ and using $\ffpm\lesssim\eps^2$, we find
\begin{align}
\|\lapp q_w^\pm\|_{3,\pm}^2\leq C(|\psi|_4,|\psi_t|_3)\left(\ino{\eps^2\TT^2 (b^\pm,p^\pm)}_{3,\pm}^2+P\left(\ino{v^\pm,b^\pm}_{4,\pm}\right)\right)\leq P(\EEW_4(t))\EW_5(t).
\end{align}
\begin{rmk}[Necessity of anisotropic Sobolev spaces]
The term bounded by $\EW_5(t)$ is contributed exactly by the extra $\frac12|b^\pm|^2$ in the total pressure.  From this, we can also see the necessity of the anisotropic Sobolev spaces when studying ideal compressible MHD. For Euler equations, the source terms for the wave equation only contain quadratic first-order terms, and one can use the trick in \cite{LL2018priori,Zhang2021elasto} to close the energy bound by $\EEW_4(t)$. For incompressible MHD, the second-order time derivative term vanishes because $q^\pm$ satisfies an elliptic equation. 
\end{rmk}

The term $|\Psi^R|_{3.5}$ can also be directly controlled. Recall that
\[
\Psi^R:=\frac{\sigma}{2}(\dnp-\dnm)\dnw^{-1}(\dnp-\dnm)(\h(\psi)) -(\dnp-\dnm)\dnw^{-1}\left(\jump{\FP-\rho\p_t^2\psi}\right)+\frac{\sigma}{2}\RR_\psi^\sigma.
\]Using the Sobolev estimates for the Dirichlet-to-Neumann operators, we have
\begin{align}
&\bno{\frac{\sigma}{2}(\dnp-\dnm)\dnw^{-1}(\dnp-\dnm)(\h(\psi))}_{3.5} + \bno{(\dnp-\dnm)\dnw^{-1}\left(\jump{\FP-\rho\p_t^2\psi}\right)}_{3.5}\no\\
\lesssim&~|\sigma \hh(\psi)|_{2.5} + \bno{\FP^\pm}_{2.5}+\bno{\jump{\rho}\p_t^2\psi}_{2.5}\no\\
\lesssim&~\left((1+\sigma)|\psi|_{4.5}+|\psi_t|_{3.5}\right)\left(\ino{\eps^2\p_t^2 p^\pm}_{2,\pm}+P(\|v^\pm,b^\pm\|_{3,\pm},|\psi|_{3})\right)+\bno{\jump{\rho}\p_t^2\psi}_{2.5}\lesssim P(\EEW_4(t)).
\end{align} Setting $s=4$ in the remainder estimate \eqref{para RR}, we have $|\sigma\RR_\psi^\sigma|_{3.5}\leq C(|\psi|_{4.5})|\sigma\psi|_5\leq \sqrt{\sigma} C(\eew(t))\sqrt{\ee(t)}\leq \sqrt{\sigma}C(\EEW_4(t)).$ 

Summarizing the estimates above, we get 
\begin{align}
\ddt\ee(t)\lesssim \sigma^{-1} P(\EE_4(t))E_5(t),
\end{align} and under the stability condition \eqref{syrov 3D com}, we get
\begin{align}
\ddt(\ee(t)+\eew(t))\lesssim P(\EEW_4(t))\EW_5(t).
\end{align} 
Invoking Lemma \ref{lem para norm}, we actually prove the following uniform-in-$\eps$ estimate for fixed $\sigma>0$
\begin{align}\label{psi 4.5-1}
\ddt\left(\bno{\sqrt{\sigma}\psi}_5^2+\bno{\psi_t}_{3.5}^2\right)\lesssim  \sigma^{-1} P(\EE_4(t))E_5(t),
\end{align}
and the following uniform-in-$(\es)$ estimate under the stability condition \eqref{syrov 3D com}
\begin{align}\label{psi 4.5-2}
\ddt\left(\bno{\sqrt{\sigma}\psi}_5^2+\bno{\psi}_{4.5}^2+\bno{\psi_t}_{3.5}^2\right)\lesssim P(\EEW_4(t))\EW_5(t).
\end{align}

\subsubsection*{Step 3: Control of $|\psi_{tt}|_{2.5}$ and $|\sqrt{\sigma}\psi_t|_4$}
It remains to prove the uniform estimates for $|\psi_{tt}|_{2.5}^2$. We just need to take one more $\p_t$ in the paralinearized evolution equation \eqref{psi equ 1} to get
\begin{align}
(\rho^++\rho^-)\p_t^3\psi=&~-\frac{\sigma}{2}T_\Lam T_\hh \p_t\psi -(\rho^+ + \rho^-)(\fwc_i\fwc_j)\TP_i\TP_j\p_t\psi -2(\rho^++\rho^-)\fwc_i\TP_i\p_t^2\psi  \no\\
&+\left(\rho^+(\fbc_i^+\fbc_j^+ - \fuc_i\fuc_j) + \rho^-(\fbc_i^-\fbc_j^- - \fuc_i\fuc_j)\right)\TP_i\TP_j\p_t\psi\no\\
& -( N\cdot\nabp \p_tq_w^+ + N\cdot\nabp \p_tq_w^-) + \underline{\p_t\Psi^R} \no\\
& +[\rho^+ + \rho^-,\p_t]\p_t^2\psi - \frac{\sigma}{2}[\p_t,T_\Lam T_\hh]\psi-[\p_t, (\rho^+ + \rho^-)\fwc_i\fwc_j]\TP_i\TP_j\psi - 2 [\p_t, (\rho^++\rho^-)\fwc_i]\TP_i\p_t\psi\no\\
& +\left[\p_t, \rho^+(\fbc_i^+\fbc_j^+ - \fuc_i\fuc_j) + \rho^-(\fbc_i^-\fbc_j^- - \fuc_i\fuc_j)\right]\TP_i\TP_j\psi - [\p_t,N\cdot\nabp](q_w^++q_w^-). \label{psit equ 1}
\end{align} The proof follows in the same way as the above analysis for the $T_\MM T_\fn$-differentiated version of  \eqref{psi equ 1}, so we skip most of the repeated details and only list the differences. 

The first difference is that we should replace $T_\MM$ by $T_{\MM'}$ where the symbol $\MM'\in\Sigma^{\frac52}$ is defined by
\[
\MM':=(\mathfrak{m}^{(1.5)})^{\frac{5}{3}}=2^{\frac{5}{6}}|\xi|^{\frac{5}{2}}\left(1-\left|\frac{N}{|N|}\cdot\frac{\xi}{|\xi|}\right|^2\right)^{\frac{5}{4}}\in\Sigma^{\frac52}.
\] 

The second difference is essential: we notice that the remainder term $\Psi^R$ (defined in \eqref{PSIR}) already contains second-order time derivative $(\dnp-\dnm)\dnw^{-1}(\jump{\rho}\p_t^2\psi)$. Taking one more time derivative, we know the underlined term $\p_t\Psi^R$ contains the following term
\[
\p_t\left((\dnp-\dnm)\dnw^{-1}(\jump{\rho}\p_t^2\psi)\right)=(\dnp-\dnm)\dnw^{-1}(\jump{\rho}\p_t^3\psi)+[\p_t,(\dnp-\dnm)\dnw^{-1} ](\jump{\rho}\p_t^2\psi) +(\dnp-\dnm)\dnw^{-1}(\jump{\p_t\rho}\p_t^3\psi),
\] and thus we are required to control $|(\dnp-\dnm)\dnw^{-1}(\jump{\rho}\p_t^3\psi)|_{2.5}$. Using Lemma \ref{lem DtNR} and Lemma \ref{lem DtN-1}, we have
\[
\bno{(\dnp-\dnm)\dnw^{-1}(\jump{\rho}\p_t^3\psi)}_{2.5}\lesssim |\jump{\rho}|_{1.5}|\p_t^3\psi|_{1.5}.
\]Since we require $|\jump{\rho}|_{1.5}\lesssim\eps$, we can control the right side by $|\eps\p_t^3\psi|_{1.5}$ \textit{uniformly in }$\eps$. Under the stability condition \eqref{syrov 3D com}, this term is already a part of $\EEW_4(t)$. For fixed $\sigma>0$, we again invoke the kinematic boundary condition to get
\[
|\eps\p_t^3\psi|_{1.5}\lesssim \|\eps \p_t^2 v^\pm\|_{2,\pm}|\psi|_{2.5} + \|\eps \p_t v^\pm\|_{2,\pm}|\psi_t|_{2.5} + \|\eps v^\pm\|_{2,\pm}|\psi_{tt}|_{2.5},
\] where the right side is already controlled by $P(\EE_4(t))$. 

For $N\cdot \nabp \p_tq_w^\pm$, we shall control the $H^{2.5}(\Sigma)$ norm of them. Similarly as in Step 2 (the steps after the control of $I_{003}$), it suffices to obtain the bound for $\|\eps^2 \p_t^3 q^\pm\|_{2,\pm}$, but this is already a part of $E_5(t)$. 

Finally, we note that the commutators between $\p_t$ and paradifferential operators arising from the right side of \eqref{psit equ 1} can all be directly controlled by using the definition of paradifferential operators. For example, we have
\begin{align*}
\sigma[\p_t,T_\Lam T_\hh]\psi=\sigma[\p_t,T_\Lam]T_\hh\psi+\sigma T_\Lam([\p_t, T_\hh]\psi)=\sigma T_{\p_t\Lam}(T_\hh\psi) + \sigma T_{\Lam}T_{\p_t\hh}\psi.
\end{align*}Then using the concrete form of $\Lam$ and $T_\hh$, we get
\begin{align*}
|\sigma T_{\p_t\Lam}(T_\hh\psi)|_{2.5}\lesssim |\p_t\Lam|_{L^\infty}|\sigma T_\hh\psi|_{2.5}\lesssim P(|\cnab\psi,\cnab\psi_t|_{L^\infty}) |\sigma\psi|_{4.5}.
\end{align*} The term $\sigma T_{\Lam}T_{\p_t\hh}\psi$ and other commutators can also be similarly controlled. Thus, we can conclude the following uniform-in-$\eps$ estimate for fixed $\sigma>0$
\begin{align}
\ddt\left(\bno{\sqrt{\sigma}\psi_t}_4^2+\bno{\psi_{tt}}_{2.5}^2\right)\lesssim  \sigma^{-1} P(\EE_4(t))E_5(t),
\end{align}
and the following uniform-in-$(\es)$ estimate under the stability condition \eqref{syrov 3D com}
\begin{align}
\ddt\left(\bno{\sqrt{\sigma}\psi_t}_4^2+\bno{\psi_t}_{3.5}^2+\bno{\psi_{tt}}_{2.5}^2\right)\lesssim P(\EEW_4(t))\EW_5(t).
\end{align}

\subsection{Improved double limits without the boundedness of $\geq 2$ time derivatives}\label{sect limit well 0ST}

\subsubsection{Improved incompressible limit for fixed $\sigma>0$}
From the analysis in Section \ref{sect reduce EE4}-Section \ref{sect uniform psi eq}, we can prove the uniform-in-$\eps$ estimates for $\EE_4(t)$. For any $\delta\in(0,1)$
\begin{align}
\EE_4(t)\lesssim \delta\EE_4(t) + P(\EE_4(0)) + P(\EE_4(t))\int_0^t P(\sigma^{-1},\EE_4(\tau))+E_5(\tau)\dtau.
\end{align}For $1\leq l\leq 4$, since we do not change anything $\EE_{4+l}(t)$, we still have
\begin{align}
l=1,2,3:&~~\EE_{4+l}(t)\lesssim \delta\EE_{4+l}(t) + P(\EE_{4+l}(0)) + P(\EE_4(t))\int_0^t P\left(\sigma^{-1},\sum_{j=0}^{l}\EE_{4+j}(\tau)\right)+\EE_{4+l+1}(\tau)\dtau;\\
l=4:&~~\EE_8(t)\lesssim \delta\EE_8(t) + P(\EE_8(0)) + P(\EE_4(t))\int_0^t P(\sigma^{-1},\EE_8(\tau))\dtau.
\end{align}Therefore, we get the Gronwall-type energy inequality for $\EE(t)$
\begin{align}
\EE(t)\lesssim \delta\EE(t) + P(\EE(0)) + P(\EE(t))\int_0^t P(\sigma^{-1},\EE(\tau))\dtau.
\end{align}Choosing $\delta>0$ suitably small, the term $\delta\EE(t)$ can be absorbed by the left side. Thus, there exists a time $T_{\sigma}'>0$ depending on $\sigma^{-1}$ and the initial data, but independent of $\eps$, such that
\begin{align}
\sup_{0\leq t\leq T_{\sigma}'}\EE(t) \lesssim P(\sigma^{-1},\EE(0)).
\end{align}

With the uniform-in-$\eps$ estimates for $\EE(t)$, we now take the incompressible limit. Again, since $\|\p_t (v,b)\|_3$ is uniformly bounded with respect to $\eps$ and we still have $\eps$-independent bound $|\psi_t|_{3.5}$, the Aubin-Lions compactness lemma gives the same strong convergence result as in Section \ref{sect limit 0ST}.

\subsubsection{Improved double limits under the stability conditions}
With the estimates \eqref{psi 4.5-2} and \eqref{psitt 2.5-2} in Section \ref{sect uniform psi eq}, we can get
\[
|v_t\cdot N|_{2.5}^2+|b_t\cdot N|_{2.5}^2\lesssim P(\EEW_4(0)) + P(\EEW_4(t))\int_0^tP(\EEW_4(\tau))\EW_5(\tau)\dtau.
\]This finishes the control of $\EEW_4(t)$. Since $\EEW_{4+l}(t)=\EW_{4+l}(t)$ when $1\leq l\leq 4$ and the strategies to control them remain unchanged, we can now close the energy estimates for $\EEW(t)$, uniformly in $\eps$ and $\sigma$, under the stability condition \eqref{syrov 3D com} ($d=3$) or \eqref{syrov 2D com} ($d=2$).
\begin{align}
\EEW(t)\leq P(\EEW(0))+ P(\EEW(t))\int_0^tP(\EEW(\tau))\dtau.
\end{align}By Gr\"onwall's inequality, there exists $T'>0$ independent of $\sigma$ and $\eps$ such that
\begin{align}
\sup_{t\in[0,T]}\EEW(t)\leq P(\EEW(0)).
\end{align} Thus, by Aubin-Lions compactness lemma, we can prove the same convergence result as in Theorem \ref{thm CMHDlimit2}.

\paragraph*{Data avaliability.} This manuscript has no associated data.

\subsection*{Ethics Declarations}
\paragraph*{Conflict of interest.} The author declares that there is no conflict of interest.

\begin{appendix}
\section{Reynolds transport theorems}\label{sect transport}
We record the Reynolds transport theorems used in this paper. For the proof, we refer to Luo-Zhang \cite[Appendix A]{LuoZhang2022CWWST}
\begin{lem}\label{time deriv transport pre}
Let $f,g$ be smooth functions defined on $[0,T]\times \Omega$. Then:
\begin{align}
\ddt \int_\Omega fg \p_3\vp\dx= \int_\Omega (\pp_t f)g\p_3 \vp\dx +\int_\Omega f(\pk_t g)\p_3\vp\dx+\int_{x_3=0}fg\p_t \psi\dx'.\label{time deriv transport pre tilde}\\
\end{align}
\end{lem}

\begin{lem}[\textbf{Integration by parts for covariant derivatives}] \label{int by parts lem}
Let $f, g$ be defined as in Lemma \ref{time deriv transport pre}. Then:
\begin{align}
\int_\Omega (\pp_i f)g \p_3 \vp \dx= -\int_\Omega f(\pp_i g)\p_3 \vp\dx+ \int_{x_3=0} fg N_i\dx'.\label{int by parts tilde}
\end{align}
\end{lem}

The following theorem holds.
\begin{thm}[\textbf{Reynolds transport theorem}]\label{transport thm nonlinear}
Let $f$ be a smooth function defined on $[0, T]\times\Omega$. Then:
\begin{align}
\ddt  \io \rho |f|^2 \p_3\vp\dx = \io \rho (\Dtp f)f\p_3 \vp\dx.  \label{transpt nonlinear}
\end{align}
\end{thm}

Theorem \ref{transport thm nonlinear} leads to the following corollary. The first one records the integration by parts formula for $\Dtp$. 
\begin{cor}[\textbf{Reynolds transport theorem - a variant}] \label{transport thm without rho}
It holds that
\begin{align}
\ddt  \io fg \p_3\vp\dx =  \io (\Dtp f)g \p_3\vp\dx+\io f(\Dtp g)\p_3\vp\dx+\io (\nabp\cdot v) fg\p_3\vp\dx. \label{transpt nonlinear without rho}
\end{align}
\end{cor}

\section{Preliminary lemmas about Sobolev inequalities}\label{sect lemma}

\begin{lem}[Hodge-type elliptic estimates]\label{hodgeTT}
For any sufficiently smooth vector field $X$ and $s\geq 1$, one has
\begin{align}
\label{divcurlTT}\|X\|_s^2\leq C(|\psi|_s,|\cnab\psi|_{W^{1,\infty}})\left(\|X\|_0^2+\|\nabp\cdot X\|_{s-1}^2+\|\nabp\times X\|_{s-1}^2+\|\TP^{\alpha}X\|_0^2\right),\\
\label{divcurlNN}\|X\|_s^2\leq C'(|\psi|_{s+\frac12},|\cnab\psi|_{W^{1,\infty}})\left(\|X\|_0^2+\|\nabp\cdot X\|_{s-1}^2+\|\nabp\times X\|_{s-1}^2+|X\cdot N|_{s-\frac12}^2\right),\\
\label{divcurltt}\|X\|_s^2\leq C''(|\psi|_{s+\frac12},|\cnab\psi|_{W^{1,\infty}})\left(\|X\|_0^2+\|\nabp\cdot X\|_{s-1}^2+\|\nabp\times X\|_{s-1}^2+|X\times N|_{s-\frac12}^2\right),
\end{align} 
for any multi-index $\alpha$ with $|\alpha|=s$. The constant $C(|\psi|_s,|\cnab\psi|_{W^{1,\infty}})>0$ depends linearly on $|\psi|_s^2$ and the constants $C'(|\psi|_{s+\frac12},|\cnab\psi|_{W^{1,\infty}})>0$ and $C'(|\psi|_{s+\frac12},|\cnab\psi|_{W^{1,\infty}})>0$ depend linearly on $|\psi|_{s+\frac12}^2$. 
\end{lem}

\begin{lem}[Normal trace lemma]\label{ntrace}
For any sufficiently smooth vector field $X$ and $s\geq 0$, one has
\begin{align}
\bno{X\cdot N}_{s-\frac12}^2\lesssim C'''(|\psi|_{s+\frac12},|\cnab\psi|_{W^{1,\infty}}) \left(\|\jp^s X\|_0^2+\|\nabp\cdot X\|_{s-1}^2\right)
\end{align} 
where the constant $C'''(|\psi|_{s+\frac12},|\cnab\psi|_{W^{1,\infty}})>0$ depends linearly on $|\psi|_{s+\frac12}^2$.
\end{lem}

We list two lemmas for the estimates of traces in the anisotropic Sobolev spaces. Define $$L_T^2(H_*^m(\Om^\pm))=\bigcap_{k=0}^mH^k((-\infty,T];H_*^{m-k}(\Om^\pm))$$ with the norm $\|u\|_{m,*,T,\pm}:=\int_{-\infty}^T\|u(t)\|_{m,*,\pm}^2\dt$. Similarly, we define $$L_T^2(H^m(\Sigma))=\bigcap_{k=0}^mH^k((-\infty,T];H^{m-k}(\Sigma))$$ with the norm $|u|_{m,T}:=\int_{-\infty}^T|u(t)|_{m}^2\dt$.

\begin{lem}[Trace lemma for anisotropic Sobolev spaces]\label{trace}
Let $m\geq 1,~m\in\N^*$, then we have the following trace lemma for the anisotropic Sobolev space.
\begin{enumerate}
\item If $f\in L_T^2(H_*^{m+1}(\Omega^\pm))$, then its trace $f|_{\Sigma}$ belongs to $L_T^2(H^{m}(\Omega^\pm))$ and satisfies
\[
|f|_{m,T}\lesssim\|f\|_{m+1,*,T,\pm}.
\]

\item There exists a linear continuous operator $\mathfrak{R}_{T}^\pm: L_T^2(H^m(\Sigma)) \to L_T^2(H_*^{m+1}(\Om^\pm))$ such that $(\mathfrak{R}_{T}^\pm g)|_{\Sigma}=g$ and \[
\|\mathfrak{R}_{T}^\pm g\|_{m+1,*,T,\pm}\lesssim|g|_{m,T}.
\]
\end{enumerate}
\end{lem}
\begin{proof}
See \cite[Appendix B]{Zhang2023CMHDVS1}.
\end{proof}
There is one derivative loss in the above trace lemma, which is 1/2-order more than the trace lemma for standard Sobolev spaces. Indeed, for $\Om^\pm$ defined in this paper, we have the following estimate that will be applied to control the non-characteristic variables $q, v\cdot\NN$ and $b\cdot\NN$.
\begin{lem}[An estimate for traces of non-characteristic variables]\label{nctrace}
Let $\Om^\pm:=\T^{d-1}\times\{0\lessgtr  x_d\lessgtr \pm H\}$, $\Sigma=\T^{d-1}\times\{x_d=0\}$ and $\Sigma^\pm=\T^{d-1}\times\{\pm H\}$. Let $\TT^\alpha=(\omega(x_d)\p_d)^{\alpha_{d+1}}\p_t^{\alpha_0}\TP_1^{\alpha_1}\cdots\TP_{d-1}^{\alpha_{d-1}}\p_d^{\alpha_d}$ with $\len{\alpha}:=\alpha_0+\cdots+\alpha_{d-1}+2\alpha_d+\alpha_{d+1}=m-1,~m\in\N^*$. Let $q^\pm(t,x)\in H_*^m(\Om)$ satisfy $\|q^\pm(t)\|_{m,*,\pm}+\|\p_d q^\pm(t)\|_{m-1,*,\pm}<\infty$ for any $0\leq t\leq T$ and let $f^\pm\in H_*^{2}(\Om^\pm)\cap H^{\frac32}(\Om^\pm)$ be a function vanishing on $\Sigma^\pm$. Then we have
\begin{equation}\label{q IBP}
\is (\jp^{\frac12}\TT^\gamma q^\pm) \, (\jp f^\pm)\dx'\leq (\|\p_d q^\pm\|_{m-1,*,\pm}+\|q^\pm\|_{m,*,\pm})\|\jp^{\frac12} f^\pm\|_{1,\pm}
\end{equation}
In particular, for $s\geq 1$, we have the following inequality for any $g^\pm\in H_*^s(\Om^\pm)$ with $g^\pm|_{\Sigma^\pm}=0$.
\[
|g^\pm|_{s-1/2}^2\leq \|\jp^s g^\pm\|_{0,\pm}\|\jp^{s-1}\p_d g^\pm\|_{0,\pm}\leq \|g^\pm\|_{s,*,\pm}\|\p_dg^\pm\|_{s-1,*,\pm}.
\]
\end{lem}
\begin{proof}
See \cite[Appendix B]{Zhang2023CMHDVS1}.
\end{proof}

The following lemma concerns the Sobolev embeddings.
\begin{lem}\label{embedding}
We have the following inequalities
\begin{align*}
H^m(\Omega^\pm)\hookrightarrow H_*^m(\Omega^\pm)\hookrightarrow& H^{\lfloor m/2\rfloor}(\Omega^\pm),~~\forall m\in\N^*\\
\|u\|_{L^{\infty}(\Omega^\pm)}\lesssim\|u\|_{H_*^3(\Omega^\pm)},~~&\|u\|_{W^{1,\infty}(\Omega^\pm)}\lesssim\|u\|_{H_*^5(\Omega^\pm)},~~|u|_{W^{1,\infty}(\Omega^\pm)}\lesssim\|u\|_{H_*^{5}(\Omega^\pm)}.
\end{align*}
\end{lem}

We also need the following Kato-Ponce type multiplicative Sobolev inequality.
\begin{lem}[{\cite{KatoPonce1988}}]\label{KatoPonce}
Let $J=(1-\Delta)^{1/2}$, $s\geq 0$. Then the following estimates hold:
\begin{equation}\label{product}
\|J^s(fg)\|_{L^2}\lesssim \|f\|_{W^{s,p_1}}\|g\|_{L^{p_2}}+\|f\|_{L^{q_1}}\|g\|_{W^{s,q_2}},
\end{equation}where $1/2=1/p_1+1/p_2=1/q_1+1/q_2$ and $2\leq p_1,q_2<\infty$.
\begin{equation}\label{commutator}
\|[J^s,f]g\|_{L^p}\lesssim \|\p f\|_{L^{\infty}}\|J^{s-1}g\|_{L^p}+\|J^s f\|_{L^p}\|g\|_{L^{\infty}}
\end{equation}where  $s\geq 0$ and $1<p<\infty$.
\end{lem}

We also need the following transport-type estimate in order to close the uniform estimates for the nonlinear approximate system.
\begin{lem}\label{parabolic}
Let $f(t)\in W^{1,1}(0,T)$ and $g\in L^1(0,T)$ and $\kk>0$. Assume that $$f(t)+\kk f'(t)\leq g(t)~~\text{a.e. }t\in(0,T).$$ Then for any $t\in(0,T)$, 
\[
\sup_{\tau\in[0,t]}f(\tau)\leq f(0)+\mathop{\mathrm{ess~sup}}\limits_{\tau\in(0,t)}|g(\tau)|.
\]
\end{lem}

\section{Paraproducts and the Dirichlet-to-Neumann operator}\label{sect app para}
\subsection{Bony's paraproduct decomposition}
We already introduce the paradifferential operator in Section \ref{sect para pre}. Here we present the relations between paradifferential operators and paraproducts. The cutoff function $\tilde{\chi}(\xi,\eta)$ in the definition of $T_a u$ is 
\[
\tilde{\chi}(\xi,\eta)=\sum_{k=0}^{\infty}\Theta_{k-3}(\xi)\vartheta(\eta),
\]where $\Theta(\xi)=1$ when $|\xi|\leq 1$ and $\Theta(\xi)=0$ when $|\xi|\geq 2$ and 
\[
\Theta_k(\xi):=\Theta(\frac{\xi}{2}),~~k\in\Z,\quad \vartheta_0=\Theta,\quad \vartheta_k:=\Theta_k-\Theta_{k-1},~~k\geq 1.
\] Based on this, we can introduce the Littlewood-Paley projections $\LP_k$ and $\LP_{\leq k}$ as follows
\begin{align*}
\widehat{\LP_k u}(\xi):=\vartheta_k(\xi)\hat{u}(\xi),~~\forall k\geq 0,\quad \LP_k u:=0~~\forall k<0,\quad \LP_{\leq k} u:=\sum_{l\leq k}\LP_l u.
\end{align*} When the symbol $a(x,\xi)$ (in the paradifferential operator $T_a$) does not depend on $\xi$, we can take $\psi(\eta)\equiv 1$ and then we have $$T_a u=\sum_k \LP_{\leq k-3} a(\LP_k u)$$ which is the usual Bony's paraproduct. In general, the well-known Bony's paraproduct decomposition is
\[
au=T_a u+T_u a+R(u,a),\quad R(u,a)=\sum_{|k-l|\leq 2}(\LP_k a) (\LP_l u).
\]

We have the following estimates for the remainder $R(u,a)$
\begin{lem}[{\cite[Section 2.3]{ABZ2014IWW}}]
For $s\in\R,~r<d/2,~\delta>0$, we have \[
|T_a u|_{H^s}\lesssim \min\{|a|_{L^{\infty}}|u|_{H^s},|a|_{H^r}|u|_{H^{s+\frac{d}{2}-r}},|a|_{H^{\frac{d}{2}}}|u|_{H^{s+\delta}}\}
\]and for any $s>0,s_1,s_2\in\R$ satisfying $s_1+s_2=s+\frac{d}{2}$, we have
\[
|R(u,a)|_{H^s}\lesssim |a|_{H^{s_1}}|u|_{H^{s_2}}.
\]
\end{lem}
\subsection{Basic properties of the Dirichlet-to-Neumann operator}
Let the space dimension $d=3$ for simplicity. Given a function $f:\Sigma=\T^2\to\R$, we define the Dirichlet-to-Neumann (DtN) operator (with respect to $\psi$ and region $\Om^\pm$) by 
\[
\dnpm f:= \mp N\cdot\nabp (\HE_\psi ^\pm f)|_{\Sigma},\quad -\lapp (\HE_\psi^\pm f)=0\text{ in }\Om^\pm,~~\HE_\psi^\pm f|_{\Sigma}=f,~~\p_3(\HE_\psi^\pm f)|_{\Sigma^\pm}=0.
\]Here the Laplacian operator is defined by $\lapp:=\nabp\cdot\nabp = \p_i(\bm{E}^{ij}\p_j)$ with 
\[\bm{E}=\frac{1}{\p_3 \vp}
\begin{bmatrix}
\p_3 \vp&0&-\TP_1\vp\\
0&\p_3\vp&-\TP_2\vp\\
-\TP_1\vp&-\TP_2\vp&\frac{1+|\cnab\vp|^2}{\p_3\vp}
\end{bmatrix} = \frac{1}{\p_3\vp}\bm{P}\bm{P}^\top,~~
\bm{P}:=\begin{bmatrix}
\p_3 \vp&0&0\\
0&\p_3\vp&0\\
-\TP_1\vp&-\TP_2\vp&1
\end{bmatrix},
\] and $\vp(t,x):=x_3+\chi(x_3)\psi(t,x')$ is defined in \eqref{change of variable vp} as the extension of $\psi$ into $\Om^\pm$. The choice of $\chi(x_3)$ is slightly different from \cite{ABZ2011IWWST,ABZ2014IWW,AMDtN}, but it does not introduce any substantial difference because the expression of $\lapp$ is still written to be $\lapp:=\nabp\cdot\nabp = \p_i(\bm{E}^{ij}\p_j)$ and we have $\lapp \vp=0$ in $\Om^\pm$. The DtN operators satisfy the following estimates and we refer to \cite[Appendix A.4]{SWZ2015MHDLWP} for the proof.

\begin{lem}[Sobolev estimates for DtN operators]\label{lem DtNHs}
For $s>2+\frac{d}{2},~-\frac12\leq r\leq s-1$ and $\psi\in H^{s}(\T^d)$, we have $$|\dnpm f|_r  \leq C(|\psi|_s)|f|_{r+1}.$$
\end{lem}

\begin{lem}[Remainder estimates for DtN operators]\label{lem DtNR}
For $s>2+\frac{d}{2}$ and $\psi\in H^{s}(\T^d)$, we have $$\dnpm f = T_{\Lam^{(1),\pm}}f + R_1^\pm (f)$$ with $\Lam^{(1),+}=\Lam^{(1),\pm}$ defined in Proposition \ref{prop para DtN} and $$\forall r\in[\frac12,s-1],\quad |R_1^\pm  (f)|_{r}+|(\dnp - \dnm)f|_{r}\leq C(|\psi|_s)|f|_{r}.$$
\end{lem}

\begin{lem}[Sobolev estimates for the inverse of DtN operators]\label{lem DtN-1}
For $s>2+\frac{d}{2},~-\frac12\leq r\leq s-1$ and $\psi\in H^{s}(\T^d)$, we have $$|(\dnpm)^{-1} f|_{r+1}  \leq C(|\psi|_s)|f|_{r}.$$
\end{lem}

\end{appendix}

\end{document}